\documentclass[11pt]{article}
\usepackage{amssymb,amsmath,amsthm,verbatim}
\usepackage{fullpage}
\usepackage{xcolor}
\usepackage{url}
\usepackage{mathrsfs}
\usepackage[all]{xy}
  \SelectTips{cm}{10}
  \everyxy={<2.5em,0em>:}
\usepackage{tikz}
\usepackage{picinpar} 

\usepackage{fancyhdr}
\usepackage{ifpdf}
  \ifpdf
    \usepackage{hyperref} 
  \else
    
    \newcommand{\href}[2]{#2}
  \fi

\theoremstyle{plain}
  \newtheorem{lemma}[equation]{Lemma}
  \newtheorem{proposition}[equation]{Proposition}
  
  \newtheorem{corollary}[equation]{Corollary}
    \newtheorem{question}[equation]{Question}
    \newtheorem{conjecture}[equation]{Conjecture}

\theoremstyle{definition}
  \newtheorem{definition}[equation]{Definition}

\renewcommand{\thesection}{\arabic{section}}
\renewcommand{\theequation}{\thesection.\arabic{equation}}

 \DeclareFontFamily{U}{manual}{}
 \DeclareFontShape{U}{manual}{m}{n}{ <->  manfnt }{}
 \newcommand{\manfntsymbol}[1]{%
    {\fontencoding{U}\fontfamily{manual}\selectfont\symbol{#1}}}

\makeatletter
   \@addtoreset{section}{part}
   \@addtoreset{equation}{section}
   \@addtoreset{footnote}{section}

    {\hspace*{\fill}$\lrcorner$\endgraf\endgroup\end{trivlist}}
    
 \newenvironment{example}[1][]{
   \refstepcounter{equation}
   \begin{proof}[Example~\theequation%
   \@ifnotempty{#1}{ (#1)}.]
   }
  {\end{proof}}

 \newenvironment{remark}[1][]{
   \refstepcounter{equation}
   \begin{proof}[\emph{\textbf{Remark~\theequation}%
   \@ifnotempty{#1}{ (#1)}}.]
   }
  {\end{proof}}

\makeatother

  \DeclareFontFamily{OT1}{pzc}{}
  \DeclareFontShape{OT1}{pzc}{m}{it}{<-> s * [1.100] pzcmi7t}{}
  \DeclareMathAlphabet{\mathpzc}{OT1}{pzc}{m}{it}

\newif\ifhascomments \hascommentstrue
\ifhascomments
  \newcommand{\jordan}[1]{{\color{magenta}[[\ensuremath{\bigstar\bigstar\bigstar} #1]]}}
  \newcommand{\matt}[1]{{\color{red}[[\ensuremath{\spadesuit\spadesuit\spadesuit} #1]]}}
  \newcommand{\david}[1]{{\color{blue}[[\ensuremath{\heartsuit\heartsuit\heartsuit} #1]]}}
\else
  \newcommand{\jordan}[1]{}
  \newcommand{\matt}[1]{}
  \newcommand{\david}[1]{}
\fi

  \newcommand{\davidExtra}[1]{}

\renewcommand{\AA}{\mathbb{A}}

\DeclareMathOperator{\Aut}{\ensuremath{\mathcal{A}\kern-.125em\mathpzc{ut}}}

\newcommand{\bbar}[1]{\overline{#1}}

\newcommand{\C}{\mathcal C}
\newcommand{\CC}{\mathbb C}

\newcommand{\E}{\mathcal E}
\DeclareMathOperator{\Endo}{\ensuremath{\mathcal{E}\kern-.125em\mathpzc{nd}}}

\newcommand{\F}{\mathcal F}

\newcommand{\GG}{\mathbb G}
\DeclareMathOperator{\GL}{GL}

\let\hom\relax
\DeclareMathOperator{\hom}{Hom}

\DeclareMathOperator{\Hom}{\ensuremath{\mathcal{H}\kern-.125em\mathpzc{om}}}
\newcommand{\id}{\mathrm{id}}

\newcommand{\K}{\mathcal{K}}

\renewcommand{\L}{\mathcal L}

\renewcommand{\O}{\mathcal O}

\newcommand{\PP}{\mathbb{P}}

\newcommand{\RR}{\mathbb R}

\renewcommand{\setminus}{\smallsetminus}

\DeclareMathOperator{\spec}{Spec}

\DeclareMathOperator{\Spec}{Spec}

\newcommand{\T}{\mathcal T}

\newcommand{\W}{\mathcal W}
\newcommand{\X}{\mathcal{X}}
\newcommand{\Y}{\mathcal{Y}}
\newcommand{\Z}{\mathcal{Z}}
\newcommand{\ZZ}{\mathbb{Z}}

 \def\ari[#1]{\ar@{^(->}[#1]}
 \def\are[#1]{\ar[#1]^{\txt{\'et}}}
 \def\areh[#1]{\ar[#1]|{\txt{$H$-eq}}^{\txt{\'et}}}
 \def\ars[#1]{\ar@{->>}[#1]}
 \newcommand{\dplus}{\ar@{}[d]|{\mbox{$\oplus$}}}
 \newcommand{\dtimes}{\ar@{}[d]|{\mbox{$\times$}}} 

\usepackage{titlesec}

\titleformat{\paragraph}
{\normalfont\normalsize\bfseries}{\theparagraph}{1em}{}
\titlespacing*{\paragraph}
{0pt}{3.25ex plus 1ex minus .2ex}{1.5ex plus .2ex}

\DeclareMathOperator{\rDisc}{rDisc}
\DeclareMathOperator{\fake}{fake}

\DeclareMathOperator{\Fal}{Fal}
\DeclareMathOperator{\st}{st}
\DeclareMathOperator{\abs}{abs}

\DeclareMathOperator{\Res}{Res}
\DeclareMathOperator{\an}{an}

\DeclareMathOperator{\Sym}{Sym}
\DeclareMathOperator{\edd}{edd}

\DeclareMathOperator{\disc}{disc}

\DeclareMathOperator{\lcm}{lcm}

\DeclareMathOperator{\height}{ht}
\DeclareMathOperator{\Height}{Ht}

\DeclareMathOperator{\ind}{ind}
\DeclareMathOperator{\length}{length}

\DeclareMathOperator{\sqf}{sqf}

\DeclareMathOperator{\ord}{ord}
\DeclareMathOperator{\Pic}{Pic}
\DeclareMathOperator{\rank}{rank}

\DeclareMathOperator{\supp}{Supp}

\DeclareMathOperator{\Gal}{Gal}

\newcommand{\FF}{{\mathbb F}}

\newcommand{\cC}{\mathcal{C}}

\newcommand{\cO}{\mathcal{O}}
\newcommand{\cX}{\mathcal{X}}
\newcommand{\cV}{\mathcal{V}}
\newcommand{\cL}{\mathcal{L}}

\newcommand{\cZ}{\mathcal{Z}}
\newcommand{\cW}{\mathcal{W}}
\newcommand{\cM}{\mathcal{M}}
\newcommand{\cA}{\mathcal{A}}
\newcommand{\cU}{\mathcal{U}}
\newcommand{\Q}{\mathbf{Q}}
\renewcommand{\C}{\mathbf{C}}
\newcommand{\R}{\mathbf{R}}
\renewcommand{\Z}{\mathbf{Z}}
\renewcommand{\ZZ}{\mathbf{Z}}
\newcommand{\set}[1]{\{#1\}}
\newcommand{\tensor}{\otimes}

\renewcommand{\ra}{\rightarrow}
\renewcommand{\P}{\mathbb{P}}
\newcommand{\beq}{\begin{equation*}}
\newcommand{\eeq}{\end{equation*}}

\newcommand{\XX}{\mathcal{X}}
\newcommand{\OO}{\mathcal{O}}
\newcommand{\inj}{\hookrightarrow}

\newcommand{\defi}[1]{\textsf{#1}} 				

\makeatletter
\newcommand*{\doublerightarrow}[2]{\mathrel{
  \settowidth{\@tempdima}{$\scriptstyle#1$}
  \settowidth{\@tempdimb}{$\scriptstyle#2$}
  \ifdim\@tempdimb>\@tempdima \@tempdima=\@tempdimb\fi
  \mathop{\vcenter{
    \offinterlineskip\ialign{\hbox to\dimexpr\@tempdima+1em{##}\cr
    \rightarrowfill\cr\noalign{\kern.5ex}
    \rightarrowfill\cr}}}\limits^{\!#1}_{\!#2}}}
\newcommand*{\triplerightarrow}[1]{\mathrel{
  \settowidth{\@tempdima}{$\scriptstyle#1$}
  \mathop{\vcenter{
    \offinterlineskip\ialign{\hbox to\dimexpr\@tempdima+1em{##}\cr
    \rightarrowfill\cr\noalign{\kern.5ex}
    \rightarrowfill\cr\noalign{\kern.5ex}
    \rightarrowfill\cr}}}\limits^{\!#1}}}
\makeatother

\begin{document}
\title{Heights on stacks and a generalized Batyrev--Manin--Malle conjecture}
\author{Jordan S. Ellenberg, Matthew Satriano, David Zureick-Brown}

\date{}
\maketitle

\begin{abstract} We define a notion of {\em height} for rational points with respect to a vector bundle on a proper algebraic stack with finite diagonal over a global field, which generalizes the usual notion for rational points on projective varieties.  We explain how to compute this height for various stacks of interest (for instance:  classifying stacks of finite groups, symmetric products of varieties, moduli stacks of abelian varieties, weighted projective spaces). In many cases our uniform definition reproduces ways already in use for measuring the complexity of rational points, while in others it is something new.  Finally, we formulate a conjecture about the number of rational points of bounded height (in our sense) on a stack $\cX$, which specializes to the Batyrev--Manin conjecture when $\cX$ is a scheme and to Malle's conjecture when $\cX$ is the classifying stack of a finite group.
\end{abstract}

\tableofcontents

\section{Introduction}

Two subjects of central importance in arithmetic statistics are the enumeration of number fields of bounded discriminant (governed by {\em Malle's conjecture}) and the enumeration of rational points of bounded height on varieties (governed by the {\em Batyrev--Manin conjecture}). 

More specifically,  if $G$ is a subgroup of $S_n$, denote by $N_G(B)$ the number of degree $n$ number fields $K/\Q$ whose Galois closure has Galois group $G$, and whose discriminant has absolute value at most $B$.  Similarly, if $X$ is a projective Fano variety, denote by $N_X(B)$ the number of rational points in $X(\Q)$ whose height is at most $B$.  Malle's conjecture predicts that $N_G(B)$ is asymptotic to $c B^{a(G)} (\log B)^{b(G)}$, where $a(G)$ and $b(G)$ are explicitly computable constants.  The Batyrev--Manin conjecture predicts that $N_X(B)$ is asymptotic to $c B^{a(X)} (\log  B)^{b(X)}$, where $a(X)$ and $b(X)$ are explicitly computable constants.   (The prediction of $c$ is much more delicate:  see Peyre~\cite[D\'efinition 2.1]{peyre:hauteurs-et-mesures} for the Batyrev--Manin case, and Bhargava~\cite{bhargava:Mass-formulae-for-extensions-of-local-fields-and-conjectures-on-the-density-of-number-field-discriminants} for the Malle case, in the special case $G = S_n$. We make no attempt in the present paper to study the constants in our generalization of Batyrev--Manin--Malle, and we say only a bit about the powers of $\log B$; we confine our concrete predictions to the exponents $a$.)

The similarity between these two asymptotic predictions has not gone unremarked.  The relation between the two conjectures becomes even closer upon making the observation that a Galois $G$-extension of $\Q$ actually {\em is} a rational point:  not a rational point on a variety, but a rational point on an algebraic stack, in this case the classifying stack $BG$.  It is thus natural to ask how one might formulate a conjecture about counting rational points of bounded height on a stack $\XX$, which would specialize both to the Batyrev--Manin conjecture (when $\XX$ is a Fano variety) and to Malle's conjecture (when $\XX$ is the classifying stack of a finite group).

An obstacle appears immediately:  there is no agreed-upon definition of the height of a rational point on a stack.  The conventional definition of height, due to Weil, is a real-valued function on $X(\Q)$ where $X$ is a projective variety.  It suffices to define height on $\P^n(\Q)$, because given the projective embedding $\iota \colon X \ra {\P^n}$, we simply define $\height_X(x)$ to be $\height_{\P^n}(\iota(x))$ for every point $x \in X(\Q)$.  But a stack which, like $BG$, is {\em not} a scheme, does not embed in projective space.

The goal of the present paper is to propose a definition of height for rational points on stacks over arbitrary global fields $K$, and, using this definition, to formulate a conjecture of Batyrev--Manin--Malle type for the number of rational points on a stack $\XX$ of height at most $B$ (under certain assumptions which guarantee this number is finite).  Having made the definition, we find that our notion of height applies to many interesting stacks which are neither schemes nor classifying spaces of finite groups (e.g.~weighted projective spaces, moduli spaces, symmetric powers of varieties).  In many cases, our definition agrees with ad hoc notions of ``size'' of a rational point which already appear in the literature.

We remark on some existing work concerning heights on stacks.  One proposed definition for the height of a point on a Deligne--Mumford stack is given and used by Abramovich and V\'{a}rilly-Alvarado in  \cite{abramovichVA:campana-points, abramovichVA:level-structures-kodaira, abramovichVA:level-structures}; this notion of height is useful for moduli spaces but does not, for example, extend to an interesting height on $BG$.  Beshaj, Gutierrez, and Shaska~\cite{shaska} have a definition of height on weighted projective space which agrees with ours in that case, as does the earlier preprint of Deng~\cite{deng1998rational}.  Starr and Xu~\cite [\S 1.4 of arXiv v1]{starr2020rational} have another definition whose relation to the one used in the present work is roughly that between the minimal slope in the Harder--Narasimhan filtration of a vector bundle and the slope of that vector bundle.
And in very recent work, Nasserden and Xiao~\cite{nasserdenS:The-density-of-rational-points-P1-with-three-stacky-points} offer an alternative definition for stacky curves, and Ratko Darda~\cite[Theorem 1.5.7.1]{darda2021rational-arxiv} has proposed a definition for weighted projective stacks.

We have seen above that one cannot define the height of a rational point of a stack by imitating the standard definition for rational points on varieties.   Before sketching our definition, we explain some further reasons for the difficulty of defining heights on stacks.

\subsubsection*{Failure of additivity}  A central feature of the theory of heights on varieties is {\em additivity}.  Given a proper variety $X$, we can define a height function $\height_{\cL}$ on $X(\Q)$ corresponding to any line bundle $\cL$ on $X$, and we have 
\begin{equation}
\height_{\cL \otimes \cL'}(x) = \height_{\cL} (x) + \height_{\cL'} (x)
\label{eq:additivity}
\end{equation}
for any pair $\cL,\cL'$ of line bundles on $X$ and any $x \in X(\Q)$.

It turns out there is no choice but to discard this useful feature when we extend the theory of heights to stacks.  The following example shows why.  Let $\XX = B(\Z/2\Z)$ and let $K=\Q$.  A line bundle $\cL$ on $\XX$ is a representation of $\Z/2\Z$; we choose $\C$ to be the nontrivial $1$-dimensional representation.  Then the tensor product of $\cL$ with itself is the trivial line bundle; i.e., $\cL \otimes \cL = \cO$ in $\Pic(\XX)$. 
 Thus $\height_{\cL\otimes \cL'}(x) = 0$ for all $x \in \XX(\Q)$.  If our height functions satisfied \eqref{eq:additivity}, we would have $2\height_{\cL} (x) = 0$, and thus $\height_{\cL}$ would be identically $0$, and thus uninteresting.\footnote{One might suggest abandoning the requirement that height functions be real-valued instead of abandoning additivity.  This feels like a bad idea to us:  for one thing, if our goal is to count points of bounded height we want the target of the height function to carry a natural ordering.}

\subsubsection*{Failure of valuative criterion of properness}  Suppose $K = \FF_q(t)$, and $X_0/K$ is a projective variety.  In this case, the height of a point $x \in X_0(K)$ has a very nice geometric interpretation.  We may choose an projective integral model $X/\P^1$ whose generic fiber is $X_0$.  By the valuative criterion of properness, we can extend $x$ to a section $\overline{x}\colon  \P^1 \ra X$.  Then the height of $x$ is just the degree of the line bundle $\overline{x}^* \OO_X(1)$ on $\P^1$.  (Note that the height may depend on the choice of integral model.)  When $X$ is a proper stack instead of a projective scheme, the valuative criterion of properness does not allow us to ``spread out" a rational point in this fashion.  For instance, an $\FF_q(t)$-point of $B(\Z/2\Z)$ is a quadratic extension of $\FF_q(t)$.  On the other hand, a map from $\P^1$ to $B(\Z/2\Z)$ is an \'{e}tale double cover of $\P^1$, which can only be the disjoint union of two copies of $\P^1$.  In particular, the fiber of such a map over the generic point $\Spec \FF_q(t)$ must correspond to the {\em trivial} quadratic extension $\FF_q(t) \oplus \FF_q(t)$. 

\subsubsection*{Modification of Northcott property}  A useful feature of the height on a variety $X$ attached to an ample line bundle $L$ is the {\em Northcott property}; the set of points $x$ in $X(\overline{K})$ with $h_L(x) < B$ and which are defined over an extension $K'/K$ of degree at most $d$ is finite.  We will often consider heights here which we want to consider ``positive," but which do not have this property.  For example, when $x \in B(\Z/2\Z)(K)$ is a point corresponding to an {\em everywhere unramified} $G$-extension of $K$, and $L$ is a (the!) nontrivial line bundle on $B(\Z/2\Z)$, we will see below that $h_L(x) = 0$.  But there are infinitely many distinct degree-$d$ extensions of $\Q$ which have everywhere unramified double covers, so the Northcott property cannot hold in its usual sense.  What will typically be true, on the other hand, is that the heights of greatest interest to us will admit only finitely many points of bounded height over any {\em individual} global field.  This is the notion of Northcott we will use in the present paper, though it does not quite follow the usual convention.

\subsubsection*{Vector bundles}  The usual height machine assigns a height function on $X(K)$ to any line bundle on $X$.  For rational points on a stack $\XX$, it turns out that this point of view is not quite sufficient for our purposes.  Consider again the example of $BG$ where $G$ is a finite group.  The line bundles on $BG$ are the $1$-dimensional representations of $G$; in particular, the line bundles only ``see" the abelianization of $G$, not all of $G$.  When $G$ is non-abelian, this turns out to imply that no height function coming from a line bundle on $\XX$ can compute the discriminant of the $G$-extension $L/K$ corresponding to a $K$-rational point.  Rather, we need access to the entire representation theory of $G$, which is to say we need to study heights associated to vector bundles of higher rank on $BG$.

\subsubsection*{Our definitions of heights on stacks} 
We now sketch the main idea of our definition.  Suppose $K$ is a global field.  If $K$ is a function field, let $C$ be the smooth projective curve with function field $K$; if $K$ is a number field, let $C$ be $\Spec \OO_K$.  Given a rational point $x\colon \Spec K \ra \XX$ we may not, as mentioned above, be able to extend $x$ to a morphism from $C$ to $\XX$.  However, it turns out that we {\em can} extend $x$ to a map $\overline{x}\colon \cC \ra \XX$, where $\cC$ is a so-called {\em tuning stack} over $C$.  When $C$ is $\P^1/\FF_q$, for instance, $\cC$ is a ``stacky $\P^1$" which is generically isomorphic to $\P^1$ but has some points with nontrivial finite inertia groups.  In general, the structure map $\pi\colon \cC \ra C$ will be a coarse moduli map,.

Suppose $\cV$ is a vector bundle on $\XX$, which we take to be metrized at archimedean places if $K$ is a number field.  Then $\overline{x}^* \cV$ is a vector bundle on the tuning stack $\cC$, and $\pi_* \overline{x}^* \cV$ is a vector bundle on $C$, whose determinant is a line bundle on $C$. We now define
\beq
\height_{\cV} (x) = -\deg( \det(\pi_* \overline{x}^* \cV^{\vee})).
\eeq
In the number field case, $-\det(\pi_* \overline{x}^* \cV^{\vee})$ is a {\em metrized} line bundle on $C$, and degree means Arakelov degree.

We note that the reason for the failure of additivity is now apparent:  while the pullback $\overline{x}^*$ is compatible with tensor product of vector bundles, the pushforward $\pi_*$ is not. Moreover, it really is crucial to include the push forward $\pi_*$; otherwise, line bundles on $BG$, which are all torsion in the Picard group, would all give trivial height functions! 

In the Section~\ref{sec:Heights of rational points on stacks}, we define $\height_{\cV}$ rigorously and show that it does not depend on the choice of tuning stack.  In Section~\ref{sec:examples}, we compute several examples, which show that this notion captures arithmetic quantities of interest in many cases. In particular, we show that if 
\begin{itemize}
\item $G$ is a subgroup of $S_n$,
\item $\cV$ is the corresponding $n$-dimensional permutation representation of $G$,
\item and $x$ is a point of $BG(\Q)$, corresponding to a degree-$n$ extension $K/\Q$ whose Galois closure has Galois group $G$,
\end{itemize}
the height $\height_{\cV}(x)$ is precisely the discriminant of $K/\Q$; see Subsection \ref{subsec:BG}.  This realizes the goal of expressing the discriminant of a field extension as the height of a rational point on the classifying stack of a finite group.

We also work out in varying levels of detail several examples of natural stacks: stacks birational to $\P^1$, weighted projective spaces, symmetric powers of projective spaces, and moduli stacks of abelian varieties.

Finally, we turn to conjectures about point-counting in Section~\ref{sec:count-rati-points}.  Using geometric intuition derived from the function field case, we propose a heuristic rate of growth for the function $N_{\XX,\cV}(B)$, the number of rational points $x$ of a stack $\XX$ such that $\height_{\cV}(x) \leq B$. There is one further technical hurdle worthy of note in the introduction: in the case of the Batyrev--Manin conjecture for schemes $X$, the expected growth rate $B^a$ is governed by the anti-canonical height $\height_{-K_X}$; in the case of stacks, one cannot simply import the same formula since for many stacks of interest, e.g.~$\X=BG$, the anti-canonical bundle is trivial! Thus, we introduce a new function (see Definition \ref{def:edd}) which replaces the anti-canonical height function on stacks; it can be viewed as a suitable perturbation of the anti-canonical height. Our point-counting conjecture \ref{conj:weak-batman} includes (the weak versions of) both the Batyrev--Manin conjecture and Malle's conjecture as special cases, but it makes many more predictions as well, which we hope will be the subject of future research.

\subsection{Notation and Conventions}
\label{subsec:not-conv}
Throughout this paper, we treat the arithmetic and the geometric settings in unison, letting $C$ denote either $\Spec \mathcal{O}_K$ for a number field $K$, or a smooth proper curve over a field $k$ in which case we set $K = k(C)$. In the number field case, we implicitly assume that all vector bundles are metrized. Finally, if $L/K$ is a finite extension of function fields corresponding to a map $f\colon C'\to C$, we let $\disc(L/K)$ be the degree of the ramification divisor.

\section*{Acknowledgments}
It is a pleasure to thank Dan Abramovich, Jarod Alper, Eran Assaf, Frank Calegari, Antoine Chambert-Loir, Brian Conrad, Ratko Darda, Anton
Geraschenko,  Aaron Landesman, Pierre Le Boudec, Brian Lehmann, Aaron Levin, Siddharth Mathur, Lucia Mocz, Brett Nasserden, Martin Olsson, Fabien Pazuki, Rachel Pries, Tony Shaska, Jason Starr, Sho Tanimoto, Anthony V\'{a}rilly-Alvarado, John Voight, Melanie Matchett Wood, Takehiko Yasuda, and Xinyi Yuan for valuable conversations about the material in this paper.  The first author was supported by NSF grant DMS-1700884 and DMS-2001200, a Simons Foundation Fellowship, and a Guggenheim Fellowship. The second author was supported by a Discovery Grant from the National Science and Engineering Board of Canada. The third author was partially supported by NSF grant DMS-1555048.

\section{Heights of rational points on stacks}
\label{sec:Heights of rational points on stacks}

Recalling our notation and conventions (Section \ref{subsec:not-conv}), let $K$ be either a number field or a function field of transcendence degree $1$ over $k$. In the former case, let $C=\spec\O_K$ and in the latter case, let $C$ be the smooth proper curve over $k$ with $K=k(C)$. Next, let $p\colon \cX \to C$ be a normal proper Artin stack over $C$ with finite diagonal.
This implies by \cite{conrad:keelMori} that there is a coarse space morphism $q\colon \cX \to X$.

A $K$-\defi{rational point} $x \in \cX(K)$ is a section
\[
x\colon \Spec K \to \cX
\]
of $p$ over the generic point $\eta := \Spec K$ of $C$, and an \defi{integral point} is a section $\overline{x} \colon C \to \cX$ of $p$. Now in the case of proper schemes, the valuative criterion tells us that every rational point extends uniquely to an integral point. However, this is no longer true for proper stacks; instead there exists a (possibly ramified) surjection $C' \to C$ such that the point $x' \colon \Spec k(C') \to \cX$ extends to an integral point $C' \to \cX$. It is precisely this phenomenon that leads to difficulties in defining heights on stacks.

Before discussing how to define heights of rational points on stacks, let us start by describing heights of integral points. This is actually rather simple and not different from the case of schemes. Given a vector bundle $\cV$ on $\cX$, we let the \defi{height} $\height_{\cV}(\overline{x})$ of an integral point $\overline{x}\colon C\to\X$ be $-\deg \left( \overline{x}^*\cV^{\vee} \right)$. (In the arithmetic setting, $\cV$ is metrized, and we mean the Arakelov degree.) 
The notion of height of an integral point satisfies Weil's height machine, in that 
\[\height_{\cL^{\otimes n}}(\overline{x}) = n\height_{\cL}(\overline{x})\]
for a line bundle $\cL$.
As mentioned above, for proper schemes there is no difference between rational points and integral points, so for schemes it is enough to define heights for integral points. For stacks we must now deal with rational points that do not extend to integral points.

Let us now outline the general case of how we define heights of rational points on stacks. Given a rational point $x\colon C\dasharrow\X$, we know it extends to an integral point after allowing for a ramified extension of $C$. Unfortunately, there are many choices of such ramified extensions and so our first task is to construct a ``minimal'' such extension; this extension is no longer a curve, but rather a stack, which we call a \emph{tuning stack}. Precisely, we construct a commutative diagram
\[
\xymatrix{
\spec K\ar[dr]\ar[r] \ar@/^1.4pc/[rr]^{x} & \cC\ar[r]^-{\overline{x}}\ar[d]_-{\pi} & \X\ar[dl]^-{p}\\
& C & 
}
\]
where $\pi\colon\cC\to C$ is a birational coarse space map, and $\overline{x}\colon\cC\to\X$ is a representable morphism of stacks which extends the rational point $x\colon\spec K\to\X$. We can therefore think of $\cC$ as being a ``stacky version'' of $C$ and can think of $\overline{x}$ as an integral point of $\X$. We then define the \defi{stable height}  of the rational point $x\in\X(K)$ with respect to $\cV$ to be
\[\height^{\st}_{\cV}(x) := -\deg(\overline{x}^*\cV^{\vee})\]
and define the unstable height (which we will refer to as simply the height) of the rational point $x\in\X(K)$ with respect to $\cV$ to be
\[\height_{\cV}(x) := -\deg(\pi_*\overline{x}^*\cV^{\vee}).\]
In Subsection \ref{subsec:tuning-stacks-and-sheaves} we show that tuning stacks exist and discuss their basic properties. We then turn to the study of heights in Subsection \ref{subsec:heights}, and in Subsections \ref{subsec:local-stacky-heights} and \ref{subsec:computing-heights} discuss some details of the practical computation of heights. In Appendix \ref{ss:one-dimensional-artin} we gather technical facts about one dimensional normal Artin stacks with finite diagonal (i.e., the types of stacks that occur as tuning stacks). 

\subsection{Tuning stacks and tuning sheaves}
\label{subsec:tuning-stacks-and-sheaves}

Throughout we let $K$, $C$, and $\X$ be as at the start of Section~\ref{sec:Heights of rational points on stacks}. Motivated by the \emph{tuning module} of Yasuda--Wood \cite[Definition 3.3]{woodY:2015mass-I}, we begin by defining the central object of this subsection.
\begin{definition}
\label{def:tuning-stack}
Given $x\in\X(K)$, we say that $(\cC,\overline{x},\pi)$ is a \defi{tuning stack} for $x$ if $\cC$ is a normal Artin stack with finite diagonal, $\pi\colon \cC \to C$ is a birational coarse space map, and the diagram
\[
\xymatrix{
\spec K\ar[dr]\ar[r] \ar@/^1.4pc/[rr]^{x} & \cC\ar[r]^-{\overline{x}}\ar[d]^-{\pi} & \X\ar[dl]\\
& C &
}
\]
commutes. 
A morphism $(\cC',\overline{x}',\pi')\to(\cC,\overline{x},\pi)$ of tuning stacks for $x$ is a map $f\colon\cC'\to\cC$ such that $\pi\circ f=\pi'$ and $\overline{x}\circ f=\overline{x}'$. 
Finally, if $(\cC,\overline{x},\pi)$ is terminal among all tuning stacks, we say $\cC$ is a \defi{universal tuning stack} for $x$.
\end{definition}

We show the existence of a universal tuning stack after some preliminaries.

\begin{remark}
\label{rmk:extending rational pts to open subsets}
Given a rational point $x\colon\spec K\to\X$, there exists a non-empty open subset $U\subseteq C$ and a map $U\to\X$ over $C$ that extends the morphism $x$. Since $\X$ is of finite type over $C$, this follows, e.g., from \cite[Proposition B.1]{rydh-noetherian-approximation}.
\end{remark}

\begin{lemma}
\label{lemma:representable-morph-of-tunings-stacks-is-iso}
Let $x\in\X(K)$ and suppose $(\cC,\overline{x},\pi)$ and $(\cC',\overline{x}',\pi')$ are tuning stacks for $x$. Then the following hold.
\begin{enumerate}
\item\label{representable-morph-of-tunings::unique-map} If $\cC'\doublerightarrow{f}{g}\cC$ are two morphisms of tuning stacks, then $f$ and $g$ are isomorphic up to unique $2$-isomorphism.
\item\label{representable-morph-of-tunings::rep->iso} If $f\colon\cC'\to\cC$ is a representable morphism of tuning stacks, then $f$ is an isomorphism.
\item\label{representable-morph-of-tunings::rep-xs->iso} If $\overline{x}$ and $\overline{x}'$ are representable, then any map $f\colon\cC'\to\cC$ of tuning stacks is an isomorphism.
\end{enumerate}
\end{lemma}
\begin{proof}
We start with (\ref{representable-morph-of-tunings::unique-map}). Since $\pi$ and $\pi'$ are birational, there is a non-empty open subset $U\subseteq C$ over which both $\pi$ and $\pi'$ are isomorphisms. Then $f|_U=g|_U$. Since $\cC$ is normal and $\cC'$ is separated, \cite[Proposition A.1]{Fantechi--Mann--Nironi:smooth-toric-deligne-mumford-stacks} tells us there is a unique $2$-isomorphism $f\simeq g$.

We now turn to (\ref{representable-morph-of-tunings::rep->iso}) and (\ref{representable-morph-of-tunings::rep-xs->iso}). Since $\overline{x}'=\overline{x}\circ f$, if $\overline{x}$ and $\overline{x}'$ are representable then \cite[Corollary 2.2.7]{Conrad:arithmeticModuliOf} shows $f$ is also representable. Thus, (\ref{representable-morph-of-tunings::rep-xs->iso}) reduces to (\ref{representable-morph-of-tunings::rep->iso}). To handle case (\ref{representable-morph-of-tunings::rep->iso}), note that $\pi$ and $\pi'$ are birational, proper, and quasi-finite, so $f$ is as well. Then $f$ is a birational, proper, quasi-finite morphism of normal stacks, hence an isomorphism by Zariski's Main Theorem.
\end{proof}

The next result makes use of relative normalization for morphisms of stacks. We refer the reader to \cite[Definition 5.3]{tmf}. 

\begin{lemma}
\label{lemma:relative normalization for nonrepable maps}
Let $f\colon\Y\to\cZ$ be a 
quasi-compact quasi-separated morphism of stacks with finite diagonal. 
Let $\Y'\to\cZ$ be the relative normalization of $f$. If $\Y$ is normal, then $\Y'$ is normal.
\end{lemma}
\begin{proof}
By definition of the relative normalization, $f$ factors as $\Y\to \Y' := \underline{\Spec}_{\cZ} \mathcal{O}' \to \cZ$, where the sheaf $\mathcal{O}'$ is the integral closure of $\mathcal{O}_{\cZ}$ in $f_* \mathcal{O}_\Y$ (i.e., the integral closure relative to the morphism of sheaves $\mathcal{O}_{\cZ} \to f_* \mathcal{O}_\Y$ induced by the map $f$). Letting $Z\to\cZ$ be a smooth cover by a scheme, we have a cartesian diagram
\[
\xymatrix{
\W\ar[r]\ar[d] & W'\ar[r]\ar[d] & Z\ar[d] \\
\Y\ar[r] & \Y'\ar[r] & \cZ
}
\]
where $\W$ may not be a scheme since we have not assumed $f$ is representable. Since relative normalization commutes with smooth base change, $W'\to Z$ is the relative normalization of $\W\to Z$. Since $W'\to\Y'$ is a smooth cover, to show normality of $\Y'$ it suffices to prove $W'$ is normal. We have therefore reduced to the case where $\cZ$ is a scheme, which we will denote by $Z$.

We are now in the situation where $f\colon\Y\to Z$ and $Z$ is a scheme. Notice that $\Y'\to Z$ is affine, and so $\Y'=Y'$ is a scheme. Since $Z$ is a scheme, we know that $f$ factors as $\Y\stackrel{\pi}{\longrightarrow} Y\stackrel{g}{\longrightarrow} Z$ where $\pi$ is a coarse space map (which exists since $\Y$ has finite diagonal). By definition, $\mathcal{O}'$ is the integral closure of $\mathcal{O}_{Z}$ in $f_* \mathcal{O}_\Y=g_*\pi_*\mathcal{O}_\Y=g_*\mathcal{O}_Y$ where the last equality holds because $\pi$ is Stein. Thus, $Y'\to Z$ is the relative normalization of $Y\to Z$. Since $\Y$ is normal, $Y$ is as well so $Y'$ is normal by \cite[\href{https://stacks.math.columbia.edu/tag/035L}{Tag 035L}]{stacks-project}.
\end{proof}

We are now ready to show the existence of universal tuning stacks. 
 We thank Martin Olsson for suggesting this construction.

\begin{proposition}[Universal tuning stacks exist]
\label{prop:tuning-stacks-exist}
Let $x\in\X(K)$. If $U\to\X$ is any extension of $x$ as in Remark \ref{rmk:extending rational pts to open subsets}, then its relative normalization $\overline{x}\colon \cC \to \X$ is a universal tuning stack, and it is independent of the choice of extension $U\to\X$.
\end{proposition}
\begin{proof}
We abusively refer to the extended map $U\to\X$ as $x$. By definition of the relative normalization, $x$ factors as
\[
U\longrightarrow \cC := \underline{\Spec}_{\X} \mathcal{O}' \stackrel{\overline{x}}{\longrightarrow} \X,
\]
where the sheaf $\mathcal{O}'$ is the integral closure relative to the morphism of sheaves $\mathcal{O}_{\X} \to x_* \mathcal{O}_U$ induced by the map $x$. Lemma \ref{lemma:relative normalization for nonrepable maps} shows that $\cC$ is normal. Since $\overline{x}$ is representable, integral, and of finite type, it follows from \cite[\href{https://stacks.math.columbia.edu/tag/01WJ}{Tag 01WJ}]{stacks-project} that it is finite. Then finiteness of the diagonal for $\cC$ follows from finiteness of the diagonal for $\X$. Thus, $\cC$ has a coarse space map $\pi\colon\cC\to C'$. Since $\cC$ is normal, $C'$ is as well. The morphism $\overline{x}$ induces a map $q\colon C'\to C$. 

We next show that $\cC\to C$ is an isomorphism over $U$. Consider the cartesian diagram
\[
\xymatrix{
U\ar[r]^-{\alpha}\ar@{=}[d] & \cC_U\ar[r]^-{\beta}\ar[d] 
& \X_U\ar[r]^-{\gamma}\ar[d] & U\ar[d]\\
U\ar[r] & \cC\ar[r] & \X\ar[r] & C.
}
\]
Since relative normalization commutes with smooth base change, $\beta\colon\cC_U\to\X_U$ is the relative normalization of $\beta\circ \alpha\colon U\to\X_U$. Note that $\gamma\circ \beta \circ \alpha=\id_U$ is proper quasi-finite and $\gamma\colon\X_U\to U$ is separated, so $\beta\circ \alpha$ is proper quasi-finite, hence finite as it is representable. Thus, $\beta\alpha$ is integral 
so its relative normalization $\alpha\colon U\to\cC_U$ is an isomorphism. As a result, $\gamma\circ\beta\colon\cC_U\to U$ is an isomorphism.

Now that we have established $\cC\to C$ is an isomorphism over $U$, it follows that $q\colon C'\to C$ is an isomorphism over $U$. So, $q$ is a birational map of normal curves (or Dedekind schemes) hence an isomorphism. This shows that $\pi\colon\cC\to C'\simeq C$ is a birational coarse space map, and hence $\cC$ is a tuning stack.

Before turning to the claim concerning universality, we show that $\overline{x}\colon\cC\to\X$ is independent of the choice of open subset $U$ and extension $U\to\X$ of $x$. To see this it suffices to show that if $i\colon V\to U$ is the inclusion of a non-empty open subset, then the relative normalizations of $x\colon U\to\X$ and $x\circ i\colon V\to\X$ are the same. Letting $\overline{x}\colon\cC\to\X$ be the former normalization and $\overline{x}'\colon\cC'\to\X$ be the latter one, by functoriality of the relative normalization we obtain a morphism $f\colon\cC'\to\cC$ of tuning stacks. Lemma \ref{lemma:representable-morph-of-tunings-stacks-is-iso} (\ref{representable-morph-of-tunings::rep-xs->iso}) shows $f$ is an isomorphism.

To prove universality, let $(\cC',\overline{x}',\pi')$ be another tuning stack. By Lemma \ref{lemma:representable-morph-of-tunings-stacks-is-iso} (\ref{representable-morph-of-tunings::unique-map}), we need only show the existence of a map $f\colon\cC'\to\cC$ of tuning stacks. We let $\cC'\longrightarrow\cC''\stackrel{\overline{x}''}{\longrightarrow} \X$ be the relative normalization of $\overline{x}'$. Since $\pi$ and $\pi'$ are birational, we can choose a non-empty open subset $U\subseteq C$ over which $\pi$ and $\pi'$ are isomorphisms. We have just showed that $\cC$ is independent of the choice of $U$, so we have a commutative diagram
\[
\xymatrix{
U\ar[r]\ar[d] & \cC'\ar[r] & \cC''\ar[d]^-{\overline{x}''}\\
\cC\ar[rr]^-{\overline{x}}\ar@{-->}[rru]^-{g} & & \X
}
\]
where we obtain the morphism $g\colon\cC\to\cC''$ (shown as a dotted arrow above) from the universal property of the relative normalization of $x\colon U\to\X$. By Lemma \ref{lemma:relative normalization for nonrepable maps}, we know $\cC''$ is normal. We also know that $\overline{x}''$ is representable, integral, and of finite type, hence finite by \cite[Tag 01WJ]{stacks-project}. Then $\cC''$ has finite diagonal, so it has a coarse space. Since $\pi'$ is an isomorphism over $U$, we see $\cC''\to C$ is a coarse space which is an isomorphism over $U$; this follows from the same argument used to establish this fact for $\cC\to C$. So, $\cC''$ is a tuning stack for $x$. Finally, Lemma \ref{lemma:representable-morph-of-tunings-stacks-is-iso} (\ref{representable-morph-of-tunings::rep-xs->iso}) shows that $g$ is an isomorphism, and so $\cC'\longrightarrow\cC''\stackrel{g^{-1}}{\longrightarrow}\cC$ is our desired map of tuning stacks.
\end{proof}

\begin{corollary}
\label{cor:universal tuning stack iff representable}
  Let $(\cC',\overline{x}',\pi')$ be a tuning stack. Then $(\cC',\overline{x}',\pi')$ is a universal tuning stack if and only if $\overline{x}'$ is representable.
\end{corollary}

\begin{proof}
Let $(\cC,\overline{x},\pi)$ be the universal tuning stack constructed in Proposition \ref{prop:tuning-stacks-exist}. By construction, $\overline{x}$ is representable. Now if $(\cC',\overline{x}',\pi')$ is a universal tuning stack, by definition of universality, there is an isomorphism $f\colon\cC'\to\cC$ of tuning stacks. Then $\overline{x}'=\overline{x}\circ f$ shows that $\overline{x}'$ is representable.

Conversely, if $(\cC',\overline{x}',\pi')$ is a tuning stack, then by universality of $\cC$, we have a morphism $f\colon\cC'\to\cC$ of tuning stacks. The result then follows from Lemma \ref{lemma:representable-morph-of-tunings-stacks-is-iso} (\ref{representable-morph-of-tunings::rep-xs->iso}).
\end{proof}

\begin{remark}
\label{remark:tuning-stacks-inherit-properties}
  We note that the universal tuning stack $\cC$ inherits many properties of $\X$. For instance, if $\X$ is Deligne--Mumford, then so is $\cC$ (since the map $\cC \to \X$ is representable); similarly, $\mathcal{C}$ is separated. 
\end{remark}

\begin{example}[Root Stacks] 
\label{examp:root-stack}
Cadman~  \cite[Section 2]{cadman:using-stacks-to-impose-tangency} introduced the notion of a \defi{root stack}, which we will use repeatedly both in examples and in proofs.
Given an algebraic stack $Y$ and an effective Cartier divisor~$E$ on~$Y$, the \defi{root stack} $\widetilde{Y} \to Y$ of order~$r$ is obtained by formally adjoining an $r$th root $\widetilde{E}$ of $E$; in other words, for a scheme $T$ and a map $f\colon T \to Y$, a lift of $f$ to $\widetilde{Y}$ corresponds to an effective Cartier divisor $E'$ on $T$ and an equivalence $rE' \sim f^*E$. 
\end{example}

\begin{remark}
Not every tuning stack is universal. For example, given any tuning stack $(\cC,\overline{x},\pi)$ and a smooth non-stacky closed point $P$ of $\cC$, let $f\colon\cC' \to \cC$ be a root stack along $P$; then $f$ is an isomorphism away from $P$ and the composite $\overline{x} \circ f\colon\cC' \to \X$ is not representable. So Corollary \ref{cor:universal tuning stack iff representable} shows that $(\cC',\overline{x} \circ f,\pi\circ f)$ is a tuning stack which is not universal.

Occasionally we will need to work with the universal tuning stack itself, e.g., in Section~\ref{sec:count-rati-points} where we define the essential deformation dimension. 
However, we prove in Proposition \ref{prop:ht-indep-of-tuning-stack} that our notion of height is independent of the choice of tuning stack. 
In practice, it is frequently more convenient to construct a tuning stack via a more direct procedure than relative normalization, such as taking a quotient stack, or as a root stack; see Section~\ref{sec:examples} for examples.
\end{remark}

\begin{definition}
\label{def:tuning-sheaf}
Let $\cV$ be a vector bundle on $\X$. If $x\in\X(K)$ and $(\cC,\overline{x},\pi)$ is a choice of tuning stack, then we refer to $\pi_*\overline{x}^*\cV^{\vee}$ (which is a vector bundle by Corollary \ref{coro:vb-pushforward}) as the \defi{tuning sheaf} associated to $x$, $\cV$, and $\cC$.
\end{definition}

\subsection{Heights}
\label{subsec:heights}

We are now ready to give the definition of the height of a rational point on a stack (with respect to a given vector bundle). We define the height to be the degree of any associated tuning sheaf. The tuning sheaf is, in general, a vector bundle, so by degree we mean the degree of the top wedge power, which is now a line bundle (metrized in the arithmetic case) on $C$. We show that this is well-defined in Proposition \ref{prop:ht-indep-of-tuning-stack}.

\begin{definition}
\label{def:height-of-rat-pt}
Let $\X$ be a stack over $C$ and let $K = K(C)$. Let $\cV$ be a vector bundle on $\X$ and $x\in\X(K)$ be a rational section.  If $\cC$ is any tuning stack for $x$ and $\T_{x,\cV,\cC}$ is the associated tuning sheaf, we let $\height_\cV(x)=-\deg(\T_{x,\cV,\cC})$. In other words, the \defi{height of the rational point $x\in\X(K)$ with respect to $\cV$} is
\[
\height_\cV(x):=-\deg(\pi_*\overline{x}^*\cV^{\vee}),
\]
where $(\cC,\overline{x},\pi)$ is any choice of tuning stack for $x$.
\end{definition}

If $L$ is a finite extension of $K$, we can define the height of a point of $\XX(L)$ by letting $C'$ be $\Spec \OO_L$ (if $K$ is a number field) or the smooth projective curve with function field $L$ (if $K$ is a function field) and consider $\X' = \X \times_C C'$, which carries a vector bundle obtained by pulling back $\cV$.  Then we define the height of a point of $\X(L)$ to be the height of the corresponding point of $\X'(L)$.

At this point, we need to comment on a piece of notation.  When $C$ is a curve over a finite field $k$, the degree of a divisor $D= P_1 + \cdots + P_r$ on $C$ is understood to be $\sum_i \log |k_{P_i}|$, where $k_{P_i}$ is the residue field of the closed point $P_i$.  In particular, $\deg D$ does not lie in $\Z$, but in $(\log q)\Z$, where $q = |k|$.  This choice of notation is most natural in a context, as here, where we want to write down theorem statements and arguments which treat the case of number fields and function fields at once.  The reader who wants to work in the context where $C$ is a curve over a fixed finite field $k$ and avoid the number field case is free to take heights to be integers, which just modifies everything in this paper by a multiplicative factor of $\log q$.

The reader may wonder why the height is defined as the negative of the degree of a bundle obtained from $\cV^\vee$, rather than as the degree of a bundle obtained from $\cV$ itself.  The answer is that, in cases arising naturally, the heights as defined here will typically be bounded below (Northcott property) while a height defined to be $\deg(\pi_*\overline{x}^*\cV)$ will often take values unbounded both above and below, or only bounded above (Southcott property).

Another natural question: why do we not define the height of $x$ as $\deg_\cC \overline{x}^* \cV$ (where degree is defined in Definition \ref{defn:degree-length}), which would be more similar to the usual definition?  The main reason is that, as we shall see, $\deg_\cC \overline{x}^* \cV$ is identically zero for many choices of $\cX$ and nontrivial $\cV$ (e.g., for any line bundle on $BG$).  Nonetheless, this function will play a key role for us (it will differ from $\height_\cV(x)$ by local terms supported on the stacky locus of $\cC$, as we will see in \S~\ref{subsec:local-stacky-heights}), so we give it a name here.

\begin{definition}
\label{def:stable-height}
 Let $\cX$, $\cV$ and $K$ be as in Definition \ref{def:height-of-rat-pt}. Then \defi{stable height} $\height^{\st}_{\cV}(x)$ is defined by
 \beq
 \height^{\st}_{\cV}(x) = -\deg_\cC \overline{x}^* \cV^{\vee}
 \eeq
 for any choice of tuning stack $\cC$.
 \end{definition}

We justify the name ``stable height'' in Proposition \ref{prop:stable-height-is-stable} below. When $x$ is an {\em integral} point of $\XX$, we may take $C$ itself to be the tuning stack; in this case, $\pi$ is the identity and $\height(x)$ and $\height^{\st}(x)$ agree.

\begin{proposition}[Height and stable height are independent of tuning stack]
\label{prop:ht-indep-of-tuning-stack}
If $(\cC_1,\overline{x}_1,\pi_1)$ and $(\cC_2,\overline{x}_2,\pi_2)$ are two choices of tuning stacks for $x\in\X(K)$, then $-\deg(\pi_{1*}\overline{x}_1^*\cV^{\vee})=-\deg(\pi_{2*}\overline{x}_2^*\cV^{\vee})$ and $-\deg(\overline{x}_1^*\cV^{\vee})=-\deg(\overline{x}_2^*\cV^{\vee})$ for all vector bundles $\cV$ on $\X$.
\end{proposition}

In fact we show more:  not only the height, but the isomorphism class of the tuning sheaf is independent of the choice of tuning stack.

\begin{proof}
Let $(\cC,\overline{x},\pi)$ be the universal tuning stack for $x$ whose existence we have shown in Proposition \ref{prop:tuning-stacks-exist}. 
By the universal property, there exist unique morphisms $f_i\colon\cC_i\to\cC$ of tuning stacks. Thus, we reduce immediately to the case where $\cC_1$ is universal and $f\colon \cC_2 \to \cC_1$ is a map of tuning stacks. Now, let
\[
\cC_2 \to \underline{\Spec}\, f_*\cO_{\cC_2} \to \cC_1
\]
be the Stein factorization. Then $\Spec f_*\cO_{\cC_2} \to \cC_1$ is a birational, finite, 
 representable map with normal codomain and hence an isomorphism by Zariski's Main Theorem. 

In particular $f$ is Stein (i.e.~the map  $\cO_{\cC_1} \to f_*\cO_{\cC_2}$ is an isomorphism). Then for any vector bundle $\mathcal{W}$ on $\mathcal{C}_1$, 
\[
\mathcal{W} \simeq \cO_{\cC_1} \otimes_{\cO_{\cC_1}} \mathcal{W}\simeq f_*\cO_{\cC_2} \otimes_{\cO_{\cC_1}} \mathcal{W}\simeq f_*f^*\mathcal{W}
\]
where the third isomorphism is the projection formula. Applying $\pi_{1*}$ to the above isomorphism with $\mathcal{W} = \overline{x}_1^*\cV^{\vee}$, we see $\pi_{1*}\overline{x}_1^*\cV^{\vee}\simeq \pi_{2*}\overline{x}_2^*\cV^{\vee}$ and so height is independent of the choice of tuning stack. (In the Arakelov case, we note that the tuning stacks are all birational, so that the metric does not change.) Independence of the stable height follows from Lemma \ref{lemma:functorality-of-degree} applied to $f_i$.
\end{proof}

The justification for the name ``stable height" is as follows.  As we shall see, the height $\height_{\cV}(x)$ does not behave well under ramified base change.  That is:  if $L/K$ is a finite extension, and $x_L$ the point of $\cX(L)$ obtained by composing $x\colon \Spec K \ra \cX$ with the structure map $p\colon\Spec L \ra \Spec K$, the relationship between $\height_{\cV}(x)$ and $\height_{\cV}(x_L)$ is not in general very transparent. For example, if $\cX = BG$ and $x \in \cX(K)$ corresponds to a Galois extension $L/K$ with Galois group $G$, then  $\height_{\cV}(x_L) = 0$, but $\height_{\cV}(x) \neq 0$ in general. For stable height, by contrast, the situation is much as we are used to from heights on schemes.

\begin{proposition}[Stable height is stable under base change]
\label{prop:stable-height-is-stable}
With $\cX$, $\cV$, $x$ and $x_L$ as above, and $L/K$ is a separable extension, then 
\beq
\height^{\st}_\cV(x_L) = [L:K] \height^{\st}_\cV(x). 
\eeq
\end{proposition}

\begin{proof} 
If $L$ is a number field, then let $C'=\spec\O_L$; if $L$ is a function field, then let $C'$ be the projective normal curve with function field $L$. Let $\mathcal{C}$ be a tuning stack for $x_K$. Then the normalization $\mathcal{C}'$ of the fiber product $\mathcal{C} \times_{C} C'$ is a tuning stack for $x_L$, and we compute that 
\beq
\height^{\st}_\cV(x_L) := \deg \overline{x}_L^*\cV = \deg g \cdot \deg \overline{x}^*\cV= [L:K] \height^{\st}_\cV(x),
\eeq
where $g$ is the projection $\mathcal{C}' \to \mathcal{C}$ and the middle inequality is Lemma \ref{lemma:functorality-of-degree}.
\end{proof}

When $\XX$ is a scheme, we can take $\cC = C$ and so stable height and height are the same.  More generally, height agrees with stable height whenever the vector bundle $\cV$ is pulled back from a vector bunde on a scheme.

\begin{proposition} Suppose $f\colon \cX \ra Y$ is a morphism over $C$, where $Y$ is a scheme.  Let $V$ be a vector bundle on $Y$.  Then, for all $x \in X(K)$,
\beq
\height_{f^* V}(x) = \height^{\st}_{f^* V}(x).
\eeq
\end{proposition}

\begin{proof}
Let $\cC$ be a tuning stack for $x$, and let $\overline{x}\colon \cC \ra \XX$ be an extension of $x$ to $\cC$.  The map $f \circ \overline{x}\colon \cC \ra Y$ factors as $g \circ \pi$ for some $g\colon C \ra Y$, by the universal property of the coarse space.  So the vector bundle $\overline{x}^* f^* V$ can be written as $\pi^* g^* V$.  Noting that duality commutes with pullback, we now have
\beq
\height_{f^* V}(x) = - \deg_C \pi_* \pi^* g^* V^\vee
\eeq
and
\beq
\height^{\st}_{f^*V}(x) = -\deg_{\mathcal{C}} \pi^* g^* V^\vee = -\deg_{C} g^* V^\vee
\eeq
(where the last equality follows from Lemma \ref{lemma:functorality-of-degree} since $\deg \pi = 1$).
The result now follows from the fact that for any bundle $W$ on $C$, 
\[
W \cong \mathcal{O}_C \otimes_{\mathcal{O}_C} W \cong \pi_*\mathcal{O}_{\mathcal{C}} \otimes_{\mathcal{O}_C} W \cong \pi_*\pi^*W;
\]
the last isomorphism is the projection formula, and the middle follows since the coarse map is Stein \cite[Theorem 6.12]{rydh:quotients}.
\end{proof}

\begin{remark}
\label{rmk:ht-under-pullback-of-V}
  Similarly, if $f\colon \mathcal{X} \to \mathcal{Y}$ is a morphism of stacks and $\mathcal{V}$ is a vector bundle on $\mathcal{Y}$, then for any $x \in \mathcal{X}(K)$, 
\[
\height_{f^* \mathcal{V}}(x) = \height_{\mathcal{V}}(f \circ x),
\]
since a tuning stack for $x$ is also a tuning stack for $f \circ x$. 
\end{remark}

\begin{definition}
\label{def:northcott}
We say that a vector bundle $\cV$ on $\cX/K$ satisfies the \defi{Northcott property} if for every finite extension $L/K$ and every integer $B$,
\[
\{x \in \cX(L) \colon \height_{\cV}(x) \leq B\} 
\]
is finite.
\end{definition}

This definition is slightly unsatisfactory, because it will be too lenient for some choices of $\cX$.  For instance, if $\cX$ is a curve of genus at least $2$, it has finitely many points over every global field, so under this definition the Northcott property will be satisfied by {\em every} vector bundle.  In the present paper, however, we will almost always be considering stacks $\cX/K$ which have infinitely many $K$-rational points.  Under such circumstances we expect $\cV$ to satisfy the Northcott property if $\cV$ is ``positive enough'', which we demonstrate through several examples; see Section~\ref{sec:examples}.  (Be warned, however, that the Northcott vector bundles do not form a cone in any sense. For instance, it is possible that a line bundle $\mathcal{L}$ is Northcott but positive multiples $\mathcal{L}^{\otimes n}$ of it are not; the non-trivial line bundle on $B{\mu_2}$ has this property.)   It is in order to ensure that natural examples exhibit the Northcott property that we use $\cV^{\vee}$ rather than $\cV$ when defining height.

\begin{definition} Let $\cX$, $\cV$ and $K$ be as in Definition \ref{def:height-of-rat-pt}, with $\cV$ Northcott.  We define the \defi{counting function} associated to $\cV$ and $K$ to be 
\[
N_{\cV,K}(B) := \#\{x \in \cX(K) \colon \Height_{\cV}(x) \leq B\}.
\]
\end{definition}

\begin{remark}
  
In case $\cV$ is a vector bundle of rank greater than $1$, it would probably be better still to consider a definition of height which associates to $x$ the tuning sheaf  $\T_{x,\cV,\cC}$ {\em itself}, rather than its degree.  One might call such a height a ``lattice height."  For instance, the lattice height of a $\Q$-point on $\X$ would be a lattice $\Lambda$ in $\RR^{\rank \cV}$, rather than a real number; the height we study in the present paper would be the covolume of $\Lambda$.  This point of view is interesting even when $\cX$ is a scheme; see for instance the notion of {\em slopes of a rational point} introduced by Peyre in \cite[\S 4.2]{peyre:liberte-et-accumulation} and \cite{peyre:beyondheights}, and the related work of Browning and Sawin in the Hardy--Littlewood regime~\cite{browningS:free-rational-curves}.  On the other hand, when $\cX$ is $BG$ and $\cV$ is a permutation representation $G \inj S_n$, the lattice height of a rational point of $\cX$ corresponding to a degree-$n$ number field $L/Q$ is the ring of integers of $\OO_L$ considered as a lattice in $L \tensor_\Q \R$; the covolume of this lattice is the absolute value of the discriminant of the number field, which is indeed the height in the sense considered in this paper.  This lattice is often called the ``shape" of the number field, and the problem of counting number fields subject to constraints on shape is already an area of substantial activity; see for instance \cite{harron:cubics, harronharron:v4}.
\end{remark}

\subsection{Computing heights:  local discrepancies}
\label{subsec:local-stacky-heights}

We now turn to the problem of practical computation of heights of points on stacks.  

As above, let $C$ be the spectrum of the ring of integers of a number field or a smooth curve over a finite field, let $K$ be the fraction field of $C$, and let $\cX$  a normal proper Artin stack over $C$ with finite diagonal.   Let $\cV$ be a vector bundle on $\cX$, where we recall once again that if $K$ is a number field, $\cV$ is a metrized vector bundle, as defined in \S~\ref{ss:stackybundles}.

Let $x\colon \Spec K \ra \cX$ be a rational point of $\cX$, let $\cC$ be a tuning stack, $\pi\colon \cC \ra C$ the coarse moduli map, and $\overline{x}\colon \cC \ra\cX$ an integral extension of $x$.

By Definition~\ref{def:height-of-rat-pt}, the height of $x$ is 
\beq
\height_{\cV}(x) = -\deg \pi_* \overline{x}^* \cV^{\vee},
\eeq
 and by Definition~\ref{def:stable-height} we have
\beq
\height_{\cV}^{\st}(x) = -\deg \overline{x}^* \cV^{\vee}.
\eeq
Our goal in this section is to study the {\em difference} between height and stable height.  To this end, we recall the natural map of vector bundles on $\cC$ 
\begin{equation}
  \label{eq:counit}
  \pi^* \pi_* \overline{x}^* \cV^{\vee} \ra \overline{x}^* \cV^{\vee}
\end{equation}
whose cokernel is a sheaf $M(\overline{x}^* \cV^{\vee})$ on $\cC$ with trivial generic fiber. This map is the counit of adjunction and we claim that it is injective. Indeed, we can check injectivity locally and assume that $C$ is affine, in which case $\pi_* \overline{x}^* \cV^{\vee} = \Gamma(\overline{x}^* \cV^{\vee})$ and the map \eqref{eq:counit} is thus the inclusion
\[
\Gamma(\overline{x}^* \cV^{\vee}) \otimes_{\mathcal{O}_C}\mathcal{O}_{\cC} \to \overline{x}^* \cV^{\vee}
\]
of global sections.

Let $C'$ be a smooth proper curve (or in the arithmetic case, $\Spec \OO_L$ for some \'{e}tale algebra $L/K$)  endowed with a finite flat surjection $p\colon C' \ra \cC$ whose degree we denote by $m$; such a $C'$ exists by Proposition \ref{prop:tuning-finite-cover}.   The sheaf $p^*  M(\overline{x}^* \cV^{\vee})$ is now a generically trivial and finitely generated sheaf on $C'$, which is to say it is a finite abelian group with the structure of an $\OO_{C'}$-module.  It follows from Proposition~\ref{proposition:degree-length-exactness}  and exactness of $p^*$ that
\begin{eqnarray*}
\log |p^* M(\overline{x}^* \cV^{\vee})| & = & \deg p^* \overline{x}^* \cV^{\vee} - \deg p^* \pi^* \pi_* \overline{x}^* \cV^{\vee} \\
& = & m(\deg \overline{x}^* \cV^{\vee} - \deg \pi^* \pi_* \overline{x}^* \cV^{\vee}) \\
& = & m(\height_{\cV}(x) - \height^{\st}_{\cV}(x)). 
\end{eqnarray*}

Now $p^*  M(\overline{x}^* \cV^{\vee})$ is a finite $\OO_{C'}$-module and as such has a canonical decomposition as a finite direct sum $\oplus _v p^*  M(\overline{x}^* \cV^{\vee})_v$, where $v$ varies over nonarchimedean places of $C'$.

\begin{definition}  With all notation as above, the \defi{local discrepancy} $\delta_{\cV;v}$ is defined as
\beq
\delta_{\cV;v}(x) = \frac{1}{m} \log |p^*  M(\overline{x}^* \cV^{\vee})_v|. 
\eeq
\end{definition}

We thus arrive at the formula
\begin{equation}
\label{eq:lsh}
\height_{\cV}(x) = \height^{\st}_{\cV}(x) + \sum_v \delta_{\cV;v}(x).
\end{equation}

One can think of the structural information imparted by \eqref{eq:lsh} as follows.  The height $\height_{\cV}$ is a non-additive function which changes under field extensions and lacks a canonical decomposition into local terms.  However, it canonically decomposes into two pieces; one of which, $\height^{\st}_{\cV}$, is additive and stable under field extensions, while the other, $\sum_v \delta_{\cV;v}$, canonically decomposes into local terms.  These good features of the summands often make it manageable to compute them individually.

\medskip

Concretely, we may think of local discrepancy as follows.  Write $K_v$ for the completion of $K$ at $v$.   Define $L_v = K_v \tensor_C C'$, so that $L_v$ is an \'{e}tale algebra over $K_v$.  We can thus write $C'_v = \Spec \OO_{L_v}$.  Choose an identification of $\overline{x}^* \cV^\vee |_{\Spec K_v}$ with $K_v^r$.  Then the generic stalk of $p^* \overline{x}^* \cV^\vee$ is identified with $L_v^r$, and we can think of $p^* \overline{x}^* \cV^\vee$ as a $C'_v$-lattice $\Lambda$ in the vector space $L_v^r$.  Then the $\OO_{K_v}$-module $\pi_* \overline{x}^* \cV^{\vee}$ is $\Lambda  \cap K_v^r$. and so $p^* \pi^* \pi_* \overline{x}^* \cV^{\vee}$ is
\beq
(\Lambda \cap K_v^r) \tensor_{\OO_{K_v}} \OO_{L_v} \subset \Lambda
\eeq
and 
\beq
\delta_{\cV;v}(x) = \frac{1}{m} \log \left| \frac{\Lambda}{(\Lambda \cap K_v^r) \tensor_{\OO_{K_v}} \OO_{L_v}} \right|.
\eeq

\begin{remark}
One particularly illustrative example is when $L_v$ is a degree $d$ Galois extension of $K_v$ with Galois group $G\subset S_d$, and $\overline{x}^* \cV^\vee$ is the $G$-representation obtained from the permutation representation of $S_d$. In this case $\Lambda$ is the $\cO_{L_v}$-module $\cO_{L_v}^{\oplus d}$ and $\sigma\in G$ acts on the $i$-th basis vector $e_i$ by $\sigma(e_i)=e_{\sigma(i)}$. Since $\Lambda$ is $G$-linearized, it follows that $\sigma(\alpha e_i)=\sigma(\alpha)e_{\sigma(i)}$ for any $\alpha\in L_v$. Said another way, $\Lambda$ is the $G$-linearized $\cO_{L_v}$-module given by the skew group ring $G\ast\cO_{L_v}$. If we label the elements of $G$ by $\sigma_1,\dots,\sigma_d\colon L_v\to L_v$, then we see $\Lambda\cap K_v^d=\Lambda^G$ is the set of sums of the form $\sum_i\sigma_i(\alpha) e_i$ with $\alpha\in L_v$. From this description, it is clear that the permutation representation is related to the discriminant. This relation will be further expanded upon in \S\ref{subsec:BG}.
\end{remark}

\begin{proposition} Let $E_v$ be an unramified extension of $K_v$ of degree $d$, let $x$ be a point of $\cX(K_v)$, and let $x_E$ be the corresponding point of $\cX(E_v)$.   Then
\beq
\delta_{\cV;v}(x_E) = d \delta_{\cV;v}(x).
\eeq
\label{pr:etaleld}
\end{proposition}

\begin{proof}
(This proof is essentially the same as that of the ``geometric" part of \cite[Lemma 3.4]{woodY:2015mass-I}.)  

Write $\Lambda_E$ for $(\Lambda \tensor_{\OO_{K_v}} \OO_{E_v})$.  Observe first that 
\beq
\Lambda_E \cap E_v^r = (\Lambda \cap K_v^r) \tensor_{\OO_{K_v}} \OO_{E_v}
\eeq
since the condition of being in $K^r$ is cut out by $K$-linear conditions on $L^r$ considered as a $K$-module; the same linear conditions applied to $(L \tensor_K E)^r$ cut out $E^r$.  We then get an equality
\beq
d \delta_{\cV;v}(x) = d \frac{1}{m} \log \left| \frac{\Lambda}{(\Lambda \cap K_v^r) \tensor_{\OO_{K_v}} \OO_{L_v}} \right|
= \frac{1}{m} \log \left| \frac{\Lambda_E}{(\Lambda \cap K_v^r) \tensor_{\OO_{K_v}} \OO_{L_v} \tensor_{\OO_{K_v}} \OO_{E_v}} \right|.
\eeq
On the other hand, writing $F_v$ for the etale algebra $E_v \tensor_{K_v} L_v$, we have
\beq
\delta_{\cV;v}(x_E) = \frac{1}{m} \log \left| \frac{\Lambda_E}{(\Lambda_E \cap E_v^r) \tensor_{\OO_{E_v}} \OO_{F_v}}\right|
=  \frac{1}{m} \log \left| \frac{\Lambda_E}{(\Lambda \cap K_v^r) \tensor_{\OO_{K_v}} \OO_{F_v}} \right|.
\eeq
The desired equality now follows from the fact that, since $E_v / K_v$ is unramified, we have $\OO_{F_v} = \OO_{L_v} \tensor_{\OO_{K_v}} \OO_{E_v}$.
\end{proof}

There will be some cases where our life is simpler if we can ignore a specified finite set of primes.  The following proposition is useful when we need to show this negligence does not perturb our height functions by very much.
\begin{proposition}
\label{pr:boundeddisc}
Suppose $K$ is a number field.  There is a constant $C(\cX,\cV,v)$ such that
\beq
\delta_{\cV;v}(x) \leq C(\cX,\cV,v)
\eeq
for all $x$ in $\cX(K_v)$.
\end{proposition}

\begin{proof}
(The following proof is adapted from a nice proof of Hilbert 90 that we learned from \cite[Lemma 3.3]{berhuy}.)

There is some constant $B$ such that every point $x \in \cX(K_v)$ extends to an integral point of $\cX(L_v)$ for some finite Galois extension $L$ of $K$ of degree at most $B$; this follows from the fact that $\cX$ has a finite cover by a scheme, see 
\cite[Theorem B]{rydh-noetherian-approximation}. Since $K$ is a number field, there are only finitely many isomorphism classes of extensions of $K_v$ of degree at most $B$.  We may thus prove the required bound for a single choice of $L_v$.

Write $G$ for $\Gal(L/K)$.  Write $\alpha_1, \ldots, \alpha_m$ for a subset of $\OO_{L_v}$ which freely spans $\OO_{L_v}$ as an $\OO_{K_v}$-module.   Let $\lambda$ be an element of $\Lambda$, and for each $i$ in $1,\ldots,m$ define
\beq
\lambda_i = \sum_{g \in G} (\alpha_i \lambda)^g.
\eeq
The action of $G$ permutes the summands above, so $\lambda_i$ is fixed by $G$ and thus lies in $\Lambda \cap K_v^r$.

We can also write
\begin{equation}
\lambda_i = \sum_{g \in G} (\alpha_i^g)(\lambda^g).
\label{eq:lambdai}
\end{equation}
Write $A$ for the $m \times m$ matrix in with rows indexed by $\alpha_1, \ldots, \alpha_m$ and columns by the elements of $G$; by Dedekind's lemma this matrix lies in $\GL_m(L_v)$. Write $\overrightarrow{\lambda}$ for the vector $\lambda_1, \ldots, \lambda_m \in L_v^m$, and $\overrightarrow{\mu}$ for the vector whose entries are $\set{\lambda^g}_{g \in G}$.  With this notation,\eqref{eq:lambdai} becomes
\beq
\overrightarrow{\lambda} = A \overrightarrow{\mu}
\eeq
which we can rewrite as
\beq
\overrightarrow{\mu} = A^{-1} \overrightarrow{\lambda}.
\eeq
In particular, we can write
\begin{equation}
\lambda = \sum a_i \lambda_i
\label{eq:lambdasum}
\end{equation}
where $a_i$ are entries of $A^{-1}$.  But note that $A$ depends only on the choice of $\alpha_i$; in particular, there is some constant $C$ such that the entries of $A^{-1}$ lie in $C^{-1} \OO_{L_v}$.  Thus, \eqref{eq:lambdasum} expresses an arbitrary $\lambda \in \Lambda$ as a linear combination of the $\lambda_i$, which lie in $\Lambda \cap K_v^r$, with coefficients in $C^{-1} \OO_{L_v}$.  We conclude that
\beq
\Lambda \subset C^{-1}[(\Lambda \cap K_v^r) \tensor_{\OO_{K_v}} \OO_{L_v}]
\eeq
which provides a bound for 
\beq
\left| \frac{\Lambda}{(\Lambda \cap K_v^r) \tensor_{\OO_{K_v}} \OO_{L_v}} \right|
\eeq
depending only on $L_v$, as required.
\end{proof}

We note that Proposition~\ref{pr:boundeddisc} does {\em not} hold in general when $K$ has characteristic $p$.  For instance, we will see that the local discrepancy at $v$ for a point of $BG$, with $\cV$ the regular representation of $G$, computes the discriminant of the local extension:  but we know that the discriminant of a $\Z/p\Z$-extension of $\FF_p((t))$ can be arbitrarily large, by contrast with the discriminant of a $\Z/p\Z$-extension of $\Q_p$.

\subsection{Computing heights:  line bundles on $\cX$ with globally generated powers}
\label{subsec:computing-heights}

In this section, we consider the special case where $\cV$ is a line bundle $\cL$.  It turns out that, speaking loosely, if some tensor power $\cL^m$ has ``enough sections," we can use these sections to compute heights of rational points on $\cX$ with little explicit reference to stacks.  (Whether this is a virtue depends on the reader's taste.)

Suppose $\cX$ is a stack over $C$, $\L$ is a metrized line bundle on $\cX$, and $s_1, \ldots, s_k$ are sections of $\L$.  We say $\L$ is \defi{generically globally generated} by $s_1, \ldots, s_k$ if the cokernel $\F$ of the corresponding morphism of sheaves
\beq
\OO_{\cX}^{\oplus k} \ra \cL
\eeq
vanishes over the generic point of $C$.  In particular, this implies that $\F$ is supported at finitely many places $v$ of $C$.  More specifically; for each nonarchimedean $v$ with uniformizer $\pi_v \in \OO_{C_v}$,  there is an integer $m_v$ such that the restriction of $\F$ to $\cX \times_C \OO_{C_v}$ is killed by $\pi_v^{m_v}$ (since $\cX$ is finite type, it suffices to check this on a finite flat cover).
In the case where $C$ has no archimedean places, we say $\L$ is globally generated by $s_1, \ldots, s_k$ when the map from $\OO_{\cX}^{\oplus k}$ to $\cL$ is surjective.  We write $q_v$ for the order of the residue field at $v$, if $v$ is a non-archimedean place; when $v$ is archimedean we can take $q_v = e$.

\begin{proposition}
\label{pr:heightcompute}
  Suppose $\cX$ is a stack, and suppose $\L$ is a metrized line bundle on $\cX$ such that some power $\L^n$ is generically globally generated by sections $s_1, \ldots, s_k$.  Let $K$ be a global field and let $x\colon \Spec K \ra \cX$ be a point of $\cX(K)$.  Choose an identification of $x^*\cL$ (whence also $x^* \cL^n$) with $K$, and write $x_1, \ldots, x_k$ for the pullbacks of $s_1, \ldots, s_k$ by $x$.  Then
\beq
\height_{\L}(x) = \sum_{v} \left\lceil (1/n) \log_{q_v}  \max(|x_1|_v, \ldots, |x_k|_v) \right\rceil \log q_v + E(x) 
\label{eq:ggheightgeneral}
\eeq
where $E(x)$ is a function bounded above and below on $\cX(K)$.  When $C$ has no archimedean places and $\L^n$ is globally generated by $s_1, \ldots, s_k$ we have
\beq
\height_{\L}(x) = \sum_{v} \left\lceil (1/n) \log_{q_v}  \max(|x_1|_v, \ldots, |x_k|_v) \right\rceil \log q_v
\label{eq:ggheightexact}
\eeq
exactly.

\end{proposition}

From now on we denote a bounded function on $\XX(K)$ by  $O_{\XX(K)}(1)$.
Note that, when $K=\Q$, we may take $x_1, \ldots, x_k$ to be integers, with the property that, for every $p$, there is some $x_i$ which is not a multiple of $p^n$.  We say such a tuple $(x_1, \ldots, x_k) \in \Z^k$ is in \defi{minimal form}.  Suppose $(x_1, \ldots, x_k)$ corresponds to a point $x$ of $\cX(\Q)$ as in Proposition~\ref{pr:heightcompute}.  The hypothesis of minimal form implies that the non-archimedean contributions all vanish, and we are left with
\begin{equation}
\height_{\L}(x) = (1/n) \log \max_i |x_i|_{\R} + O_{\XX(K)}(1)
\label{eq:heightqgg}
\end{equation}
up to a function bounded on $\cX(\Q)$.  (The ceiling function can now be neglected, since, having restricted to a single summand, the difference between a number and its floor is bounded and can be absorbed into the error term.)

We now prove Proposition~\ref{pr:heightcompute}.

\begin{proof}
 
We note, first of all, that we have not specified the choice of metric on $\cL$ at archimedean places, but this choice can be absorbed in the error term; if $\cL$ and $\cL'$ are line bundles which differ only with respect to the archimedean metric, it is easy to see from the proof below that $\height_{\cL}- \height_{\cL'} = O_{\XX(K)}(1)$.  (At the moment when we say ``Fubini-Study metric on $\OO(1)$ on complex projective space," just insert your own favorite metric, which differs from Fubini-Study by a bounded function.)

We begin by computing the degree of $\overline{x}^* \cL^n$ on $\cC$.  Let $L/K$ be a finite extension of some degree $m$ such that the pullback of $x$ to $\Spec L$ extends to a morphism $y\colon C' \ra \cX$, where $C'$ is the curve (or Dedekind domain) with fraction field $L$.  We then have a commutative diagram: 
 \[
\xymatrix{
\spec L\ar[r]^\iota\ar[d]^\phi & C'\ar[r]^-{\overline{y}}\ar[d]^p & X\ar[d]^q\\
\spec K\ar[r]\ar@/^1pc/[rr]^{x} & \cC\ar[r]^-{\overline{x}}\ar[d]^\pi & \X \\
& C & 
}
\]
 
Now $\deg_{\cC} \overline{x}^* \cL^n = (1/m) \deg_{C'} p^* \overline{x}^* \L^n$.  The latter is a metrized line bundle on $\OO_L$ whose degree we can compute by means of a section.  For ease of notation, write $\Lambda$ for $p^* \overline{x}^* \L^n$.
\beq
\deg_{C'} p^* \overline{x}^* \L^n = \log |\Lambda/ s_1 \OO_L| - \sum_{\sigma \colon  L \ra \C} |\sigma^* s_1|_\sigma.
\eeq
Write $\Lambda'$ for the submodule of $\Lambda$ spanned by $s_1, \ldots, s_k$.  By hypothesis, there is a bound independent of $x$ for the size of $\Lambda / \Lambda'$.  Thus, we may replace $\Lambda$ with $\Lambda'$ and get
\beq
\deg_{C'} p^* \overline{x}^* \L^n = \log |\Lambda'/ s_1 \OO_L| - \sum_{\sigma \colon L \ra \C}  |\sigma^* s_1|_\sigma + O_{\XX(K)}(1).
\eeq

Now the torsion $\OO_L$-module $\Lambda'/ s_1 \OO_L$ can be broken up into $v$-adic components $T_v$, one for each nonarchimedean place $v$ of $K$, and by the explicit description of $\Lambda'$ we have 
\beq
\log |T_v| = m (\log \max_i |x_i|_v - \log |x_1|_v).
\eeq
Thus we have
\beq
\log |\Lambda'/ s_1 \OO_L| = \sum_{v \nmid \infty} m(\log \max_i |x_i|_v - \log |x_1|_v).
\eeq
We now turn to the archimedean places, which requires us to specify the metric on $\cL^n$.  The sections $s_1, \ldots, s_n$ provide a map of complex manifolds $f\colon \cX(\C) \ra \P^{k-1}(\C)$, and $\cL^n|\XX(\C)$ is pulled back from $\OO(1)$ under $f$.  So we may choose for our metric on $\cL^n|\XX(\C)$ the pullback of the Fubini-Study metric on $\OO(1)$.  Having done so, we have
\beq
\sum_{\sigma\colon L \ra \C}  |\sigma^* s_1|_\sigma = \sum_{v | \infty} m (\log|x_1|_v - \log \max_i |x_i|_v ) + O_{\XX(K)}(1).
\eeq

To sum up, we have computed that
\beq
\log |\Lambda'/ s_1 \OO_L| - \sum_{\sigma\colon L \ra C}  |\sigma^* s_1|_\sigma = - \sum_v m |x_i|_v + \sum_v m \log \max_i |x_i|_v  = \sum_v \log \max_i |x_i|_v  + O_{\XX(K)}(1).
\eeq
whence 
\beq
\deg_{C'} p^* \overline{x}^* \L^n = (\sum_v m\log \max_i |x_i|_v) + O_{\XX(K)}(1)
\eeq
whence
\beq
\deg_{C} \overline{x}^* \L^\vee = -(1/n) (\sum_v \log \max_i |x_i|_v)  + O_{\XX(K)}(1).
\eeq

We note that, in the case where $K$ is a function field and $s_1,\ldots,s_k$ globally generate $\cL^n$, the expression $(\sum_v m\log \max_i |x_i|_v)$ is just the usual expression for the degree of a line bundle pulled back from $\OO(1)$ on $\P^{k-1}$ by a morphism with coordinates $(x_1: \ldots : x_k)$.

Having computed this degree, which is the negative of the stable height $\height^{\st}_{\cL}(x)$, we can compute $\height_{\cL}(x)$ by adding local discrepancies as in the previous section.  First of all, if $v$ is one of the finitely many nonarchimedean places where $\cL^n$ is not generated by $s_1, \ldots, s_k$, we observe that $\delta_{\cL;v}(x)$ is $O_{\XX(K)}(1)$ by Proposition~\ref{pr:boundeddisc}, and since the number of such places is bounded independently of $x$, we can absorb the contribution of those local discrepancies $\delta_{\cL;v}(x)$ into the error term.

So let $v$ be a nonarchimedean place where $\cL^n$ is generated by $s_1, \ldots, s_k$.  Then, given our choice of identification of $x^* \cL^n$ with $K$, and writing $L_v$ for the etale algebra $L \tensor_K K_v$, we can write $\overline{x}^* \cL^n$ as the Galois-stable lattice $I$ in $L_v$ spanned as an $\OO_{L_v}$-module by $x_1, \ldots, x_k$.  Then $\overline{x}^* \cL^\vee$ is the submodule $I^{-1/n}$ of $L_v$ consisting of all $\alpha \in L_v$ such that $\alpha^n I \subset \OO{L_v}$.  The pushforward $\pi_* \overline{x}^* \cL^\vee$ is then the submodule $I^{-1/n} \cap K_v$ of $K_v$ consisting of all $\beta \in K_v$ with $\beta^n x_i \subset \OO_{K_v}$ for all $i$, which is to say it is the fractional ideal $m_v^{c_v}$ where
\beq
c_v = \lceil -(1/n)\min_i \ord_v x_i \rceil = \lceil (1/n) \log_{q_v} \max_i |x_i|_v \rceil.
\eeq
So
\beq
\delta_{\cL;v}(x) = (1/m) \log |I^{-1/n} / I^{-1/n} \cap K| = (\log q_v) \lceil (1/n) \log_{q_v} \max_i |x_i|_v \rceil - (1/n) \log \max_i |x_i|_v. 
\eeq
Recalling from above that
\beq
\height_{\cL}^{\st}(x) = (1/n) \sum_v \log \max_i |x_i|_v  + O_{\XX(K)}(1)
\eeq
we conclude that
\beq
\height_{\cL}(x) = \height^{\st}(x) + \sum_v \delta_{\cL;v}(x) = \sum_v \lceil (1/n) \log_{q_v} \max_i |x_i|_v \rceil \log q_v+ O_{\XX(K)}(1)
\eeq
which was the desired result.
\end{proof}

\section{Examples}
\label{sec:examples}

In this section we show how to compute heights of points on various stacks that often arise in practice, emphasizing the fact that in these cases the output of our definition often recovers an invariant which was already widely used to measure the ``size" of the objects parametrized by rational points on those stacks.

\subsection{Heights on $BG$}
\label{subsec:BG}

Let $G$ be a constant finite group scheme over $C$,  let $\cX$ be the classifying stack $BG/C$, and let $q\colon C \ra BG$ be the universal $G$-cover.  Let $x\colon \Spec K \ra \cX$ be a rational point and let $\overline{x}\colon \cC \ra BG$ be the extension of $x$ to a tuning stack.  Then we have a commutative diagram

\[
\xymatrix{
C' \ar[r]^{x_{C'}}\ar[d]^p & C\ar[d]^q \\
\cC \ar[r]^{\overline{x}}  & BG \\
}
\]
 where $C'$ is a smooth proper curve (not necessarily irreducible) whose fiber over $\Spec K$ is an \'{e}tale $G$-algebra $L/K$.

Let $\cV$ be a vector bundle of rank $r$ on $BG$; in other words, $\cV$ is an $r$-dimensional representation $V$ of $G$ over $C$.  Then, by \eqref{eq:lsh}, we have
\beq
\height_{\cV}(x) = \height^{\st}_{\cV}(x) + \sum_v \delta_{\cV;v}(x).
\eeq

First of all, note that $p^* \overline{x}^* \cV^\vee = x_{C'}^* q^* \cV$ is a vector bundle on $C'$ pulled back from the trivial bundle on $C$, and thus has degree $0$.  So
\beq
\height^{\st}_\cV(x) = -\deg  \overline{x}^* \cV^{\vee} = - (\deg p)^{-1} \deg p^* \overline{x}^* \cV^{\vee} = 0.
\eeq

We have thus reduced ourselves to the local problem of computing $\delta_{\cV;v}(x)$ at the finite set of non-archimedean places $v$ of $K$ where $L/K$ is ramified.  Let $v$ be such a place.

The pullback of $\cV^\vee$ along $x_{C'}^*$ from $C$ to $C'$ is $\OO_{C'} \tensor_{\OO_C} V^{\vee}$.

Thus, locally, the $G$-stable lattice $\Lambda_v \subset L_v^r$ we use to compute the local discrepancy can be written as
\beq
\OO_{L_v} \tensor_{\OO_{K_v}} V^\vee.
\eeq

We note that this is precisely the $G$-module studied by Yasuda and Wood in section 3 of \cite{woodY:2015mass-I}.  (The free rank $r$ $\OO_{L_v}$-module we call $\Lambda_v$ is identified with $\OO_{L_v}^r$ in their notation.)  In particular, the free rank $r$ $\OO_{K_v}$-module $\Lambda^G$ is precisely the {\em tuning submodule} in \cite[Def 3.1]{woodY:2015mass-I}, and the local discrepancy $\delta_{\cV;v}(x)$ is exactly the quantity denoted $\mathbf{v}_\tau(\rho)$ in \cite[Def 3.3]{woodY:2015mass-I}.  Thus, we can make use of their results to compute the local discrepancies explicitly.

The case where $\cV$ is a permutation representation is an important example; in this case, we find that the discriminant of a field extension can be understood as a height on $BG$ in the sense of this paper.  In particular: when $\cV$ is a degree-$n$ permutation representation of $G$, and $x$ is a point of $BG(K)$, we can associate to $x$ a map
\beq
\rho_x\colon \Gal(K) \ra G \ra S_n
\eeq
which in turn specifies a degree-$n$ \'{e}tale algebra $L/K$.

\begin{proposition}
\label{prop:heightdisc}
 Let $\cV$ be a vector bundle on $BG$ corresponding to a degree-$n$ permutation representation $\rho$ of $G$, let $x$ be a point in $BG(K)$, and let $L/K$ be the algebra corresponding to $x$ as described above.  Then
\begin{equation}
\label{eq:heightdisc}
\height_{\cV}(x) = (1/2) \log |\Delta_{L/K}|.
\end{equation}
\label{pr:heightdisc}
\end{proposition}

\begin{proof}

It follows from \cite[Theorem 4.8]{woodY:2015mass-I} that
\beq
\delta_{\cV;v}(x) = (1/2) a_v(\rho_x)
\eeq
where $a_v$ is the Artin conductor of $\rho_x|_{K_v}$, which is precisely the local component at $v$ of $\Delta_{L/K}$.  Thus,
\begin{equation}
\label{eq:heightdisc}
\height_{\cV}(x) = \sum_v \delta_{\cV;v}(x) = (1/2) \log |\Delta_{L/K}|
\end{equation}
where by $|\Delta_{L/K}|$ we mean the absolute norm of the discriminant, i.e., the order of the finite group $\OO_C / \Delta_{L/K}$.
\end{proof}

In other words, the general definition of height introduced here, when applied to a $G$-extension (thought of as a point of $BS_n$), recovers the discriminant.  Of course, a point of $BS_n$ can be thought of as a $G$-extension in different ways; one might have in mind a degree-$n$ extension, the Galois $S_n$-extension obtained by applying Galois closure, or some other number field with the same Galois closure.  Each such field corresponds to a permutation representation of $S_n$ (in the first and second case above, the standard representation and the regular representation) and the discriminant of the field is computed by the height with respect to the vector bundle $\cV$ specified by the corresponding permutation representation.

The case $\cX = BG$ demonstrates the necessity of computing heights with respect to vector bundles of arbitrary rank, not only line bundles.  Line bundles on $BG$ correspond to $1$-dimensional representations of $G$.  If, for example, $G$ is a finite group with trivial abelianization, there are no nontrivial line bundles at all.  In order to have a theory of heights rich enough to capture the invariants of $G$-extensions, we have no alternative than to consider vector bundles of higher rank on $BG$.

The work of Yasuda and Wood is not limited to permutation representations.  For example, Wood and Yasuda work out in \cite[Example 4.10]{woodY:2015mass-I} the example where $G = \Z/p\Z$, $K$ is a function field of characteristic $p$, and $\cV$ is the $2$-dimensional non-semisimple representation of $\Z/p\Z$ over $K$.  A rational point of $BG$ corresponds to a $\Z/p\Z$-extension $L/K$.  If $v$ is a place of $K$, we denote by $j_v$ the largest integer $i$ such that the higher ramification group $G_i$ at $v$ surjects onto $\Z/p\Z$.  Then Yasuda and Wood's computation shows
\begin{equation}
\height_v(x) = 1 + \left\lfloor \frac{j_v}{p} \right\rfloor.
\label{eq:localas}
\end{equation}
When $K = \FF_q(t)$ with $q$ a power of $p$, the points of $B(\Z/p\Z)(K)$ correspond to Artin--Schreier curves, and the height of an Artin--Schreier curve with respect to this $\cV$ is the sum of the local terms \eqref{eq:localas} over all places $v$ of $\FF_q(t)$ which are ramified in the Artin--Schreier cover.   We do not know if this notion of height of an Artin--Schreier curve corresponds to anything that has appeared in previous literature, but we note that the expression above is closely related to that appearing in the computation of dimensions of irreducible components of moduli space for Artin--Schreier curves of specified $p$-rank in the work of Pries and Zhu \cite[Theorem 1.1]{prieszhu}.

This example also illustrates the important point that the height function $\height_{\cV}$ is {\em not} determined by the class of $\cV$ in $K_0$ of the category of vector bundles; the vector bundle above is an extension of the trivial line bundle by the trivial line bundle, but its associated height function is not zero.\footnote{This is specifically due to the fact that $B(\Z/p\Z)$ is not a tame stack over $\FF_q(t)$, so $\pi_*$ is not exact. Although $\overline{x}^*\cV^\vee$ is an extension of $\O_\cC$ by itself, $\pi_*\overline{x}^*\cV^\vee$ is no longer the extension of $\O_C$ by itself.}

\subsection{Heights on $B\mu_n$}
\label{subsection:the-ballad-of-Bmu-n}

Suppose $\cX = B\mu_n$, and $\cL$ is the line bundle on $B\mu_n$ corresponding to the standard $1$-dimensional representation $\mu_n \ra \GG_m$.  In that case, $\cL^n$ is the trivial bundle on $\cX$ and thus admits a generating section $s$.  On the other hand, if $x$ is a $K$-point of $B\mu_n$, the pullback $x^* \cL$ is isomorphic to $K$.  The obstruction to $x^*s \in \Gamma(\Spec K, x^* \cL^n)$ being an $n$th power of an nonzero section of $x^* \cL$ now yields a class in $K^* / (K^*)^n$.  Put another way:  choosing an identification of $x^* \cL$ with $K$ induces an identification of $x^* \cL^n$ with $K$, under which $x^* s$ is identified with an element $x_0 \in K^*$, which represents the class in  $K^* / (K^*)^n$ corresponding to $x$.  Note that a change in the choice of $s$ will apply a translation to the identification $B\mu_n(K) \cong K^* / (K^*)^n$, but such a change will modify heights by a bounded quantity, and if $K$ is a function field over a finite field $k$ and we require $s$ to globally generate $\cL^n$, the ambiguity in $s$ imposes translation by $k^*$, which will not change the heights we compute at all.  (If we want to remove this ambiguity entirely, we can fix for all time a choice of universal $\mu_n$-torsor $q\colon \Spec K \ra B\mu_n/K$ and an identification of $q^* \cL$ with $K$; having done so, we can require that $s$ pull back under $q$ to an element of $(K^*)^n$.)

We note that the above setup applies even when $\text{char}\, K$ divides $n$.

In particular:  Proposition~\ref{pr:heightcompute} yields
\beq
\height_{\L}(x)  = \sum_v \left\lceil \frac{1}{n} \log_{q_v}  |x_0|_v \right\rceil \log q_v.
\eeq
We note that our formula for $\height_{\L}(x)$ is unchanged, as it must be, when $x_0$ is modified by an element of $(K^*)^n$.  

By the computation above, when $K = \Q$ we see that the height of a point $x$ of $B \mu_n (\Q) = \Q^\times / (\Q^\times)^n$ is obtained as follows:  the class of $\Q^\times / (\Q^\times)^n$ corresponding to $x$ is represented uniquely by an integer $N$ with no $n$th power divisor, and as in \eqref{eq:heightqgg} we have
\beq
\height_{\L}(x) = \log |N|^{1/n}.
\eeq
(In the examples we will often suppress the $O_{\XX(K)}(1)$ error term when no confusion is likely.)

Once again, the height recovers the measure of complexity most frequently used in practice; when enumerating the elements of $\Q^*/(\Q^*)^n$, one typically identifies the elements of the group with $n$th power-free integers and lists in order of absolute value.

Of course, this choice $\L$ is not the only option.  Suppose, for instance, $K = \Q$ and $n=3$; then there are two equally good choices of nontrivial line bundle on $\cX$, namely $\L$ and $\L^2$.  Suppose $x \in B\mu_3(\Q)$ corresponds to $N M^2 \in \Q^\times / (\Q^\times)^3$, with $N$ and $M$ coprime and squarefree.  Then, as we have already observed above,
\beq
\height_{\L}(x) = \log |NM^2|^{1/3} = (1/3) \log N + (2/3) \log M.
\eeq
On the other hand, consider $\L' = \L^2$.  Then, having chosen $s$ as above, $s^2$ is a generating section of $(\L')^3$, so we can take $x_1$ to be $x^* s^2$, which corresponds to $N^2 M^4 \in \Q^\times$.  Putting this integer in minimal form modifies it to $N^2 M$, and another application of \eqref{eq:heightqgg} shows that
\beq
\height_{\L'}(x) = (2/3) \log N + (1/3) \log M.
\eeq

As a final illustration, we can see how the above two computations combine to yield Proposition \ref{prop:heightdisc} for $B\mu_3$. Let $\cV$ be the vector bundle $\cL \oplus \cL^2 \oplus \OO_{\XX}$.  Then 
\beq
\height_{\cV}(x) = \log N + \log M
\eeq
which we note is also $(1/2) \Delta_{L/\Q}$, where $L = \Q((NM^2)^{1/3}) = \Q((N^2M)^{1/3}) $ is the cubic extension of $\Q$ arising from $x$.  This is as it must be, as we now explain.  First, note that $\height_{\cV}^{\st}(x) = 0$ for all $x$ just as in the case $\cX = BG$, because $\cV$ pulls back to a trivial bundle on a finite cover of $B\mu_3$.  So
\beq
\height_{\cV}(x) = \sum_v \delta_{\cV;v}(x).
\eeq
Now the size of $\delta_{\cV;3}(x)$ is bounded by Proposition~\ref{pr:boundeddisc}, so at the expense of a bounded error term we can write
\beq
\height_{\cV}(x) = \sum_{v \neq 3} \delta_{\cV;v}(x).
\eeq
Let $x'$ be the point of $B(\mu_3)(\Q(\zeta_3))$ obtained by base change from $x$.  Since every prime other than $3$ is unramified in $\Q(\zeta_3)/\Q$, Proposition~\ref{pr:etaleld} tells us that
\beq
\height_{\cV}(x') = 2 \sum_{v \neq 3} \delta_{\cV;v}(x) = 2\height_{\cV}(x).
\eeq
On the other hand, over $\Q(\zeta_3)$, there is an isomorphism between $B(\mu_3)$ and $B(\Z/3\Z)$, which carries $\cV$ to the reduced permutation representation of $\Z/3\Z$, which we denote by $\cW$.  In fact, this isomorphism extends to $\Z[\zeta_3][1/3]$.  Let $y$ be the point of $B(\Z/3\Z)(\Q(\zeta_3))$ corresponding to $x'$ under this isomorphism, which we can also think of as the point associated to the Galois $\Z/3\Z$-extension $L(\zeta_3)/\Q(\zeta_3)$.  Then
\beq
\delta_{\cW;v}(y) = \delta_{\cV';v}(x')
\eeq
for all places $v$ of $\Q(\zeta_3)$ not dividing $3$.  We conclude that (as always, up to bounded error)
\beq
\height_{\cW}(y) = \sum_{v \neq 3} \delta_{\cW;v}(y) =  \sum_{v \neq 3} \delta_{\cV';v}(x') = 2\height_{\cV}(x).
\eeq
On the other hand, by \eqref{eq:heightdisc} we have
\beq
\height_{\cW}(y) = (1/2) \log |\Delta_{L(\zeta_3)/\Q(\zeta_3)}| = \log \Delta_{L/\Q}
\eeq
which shows that $\height_{\cV}(x) = (1/2)  \log |\Delta_{L/\Q}|$.

\subsection{Heights on weighted projective space and weighted projective stacks}
\label{subsec:weightedprojective}

In this section we consider rational points on the {\em weighted projective space} $\XX = \P(a_0, \ldots, a_k)$.  This stack is, by definition, the quotient $[\AA^{k+1}\setminus 0/\GG_m]$ where $\GG_m$ acts by the rule
\beq
\lambda\cdot(X_0, \ldots, X_k) = (\lambda^{a_0} X_0, \ldots, \lambda^{a_k} X_k).
\eeq
Then $\P(a_0, \ldots, a_k)$ is a smooth proper stack, and $\AA^{k+1} \setminus 0$ is the total space of a line bundle on $\XX$, whose dual is the {\em tautological bundle} $\OO_{\P(a_0, \ldots, a_k)}(1)$; for simplicity of notation, we denote the tautological bundle by $\cL$ for the rest of this section.  The coordinate function $X_i$ is a section of $\cL^{a_i}$.  Writing $A$ for the least common multiple of the $a_i$, the $k+1$ sections $X_i^{A/a_i}$ of $\cL^A$ generate $\cL^A$.  So we can compute heights of points in $\P(a_0, \ldots, a_k)(K)$ by applying Proposition~\ref{pr:heightcompute}, as we now explain.

Let $x$ be a point of $\P(a_0, \ldots, a_k)(K)$.  As in Proposition~\ref{pr:heightcompute}, we choose an identification of $x^* \cL$ with $K$; this assigns a value in $K$ to each of the $k+1$ coordinates, which values we denote $x_0, \ldots, x_k$.  Changing the identification of $x^*\cL$ with $K$ modifies this tuple by elementwise multiplication by tuples of the form $\lambda^{a_0}, \ldots, \lambda^{a_k}$, and we say that two tuples $x_0, \ldots, x_k$ are equivalent if they differ by such a transformation.  Then Proposition~\ref{pr:heightcompute} tells us that
\begin{equation}
\height_{\cL}(x) = \sum_v \lceil \log_{q_v}  \max_i |x_i|_v^{1/{a_i}} \rceil \log q_v.
\label{eq:weightedheight}
\end{equation}

In particular, when $K = \Q$, a rational point $x$ of $\P^1(a_0, \ldots, a_k)(\Q)$ can be identified with a tuple of integers $(M_0: \ldots :M_k)$ such that there is no prime $p$ with $p^{a_i} | M_i$ for all $i$.  Given a tuple which is in minimal form in this sense, the nonarchimedean primes contribute nothing to \eqref{eq:weightedheight}, and we get 
\begin{equation}
\height_{\cL}(x) = \log \max_i |M_i|^{1/a_i}.
\label{eq:wpsq}
\end{equation}

We note that this definition recovers the notion called ``naive height" for points of weighted projective space in \cite{shaska}.

Here is another means by which it is often practical to compute heights on weighted projective space when $K$ is a global function field.  Let $F$ be a section of $\cL^A$ -- for instance, it might be $X_i^{A/a_i}$ for some $i$ -- and let $y$ be the pullback of $F$ along $x$ to $x^*\cL^A$, which we have identified with $K$.  We define the {\em minimal valuation} of $F$ at a place $v$ of $K$ as follows.  Let $\pi_v \in K^*$ be an element which is a uniformizer at $v$, and define
\beq
c_v = \min \lfloor (1/a_i) \ord_v x_i \rfloor.
\eeq
Note that $c_v = 0$ if and only if all the $x_i$ are integral at $v$ and there is some $i$ such that $\ord_v x_i < a_i$.  In this case, we say that $(x_0, \ldots, x_k)$ is in minimal form.  If $(x_0, \ldots, x_k)$ is not in minimal form, we find an equivalent tuple in minimal form by modifying each $x_i$ by $\pi_v^{-a_i c_v}$; the effect of this transformation on $y$ is multiplication by $\pi_v^{-A c_v}$.  We therefore define the minimal valuation of $F$ to be 
\beq
\ord^{\min}_v F =  \ord_v y - A c_v =  \ord_v y -  A \min \lfloor (1/a_i) \ord_v x_i \rfloor.
\eeq
We note that this quantity does not depend on the identification of $x^* \cL$ with $K$, but only on $F$ and $v$.  Furthermore, we have
\beq
\sum_v \ord^{\min}_v F = \sum_v \ord_v y - \sum_v A \min \lfloor (1/a_i) \ord_v x_i. \rfloor = A \sum_v \max \lceil (1/a_i) \log_{q_v} \max |x_i|_v \rceil \log q_v
\eeq
and, by Proposition~\ref{pr:heightcompute}, this last quantity, taking $X_i^{A/a_i}$ to be the sections generating $\cL^A$, is exactly $A \height_{\cL} x$.  We conclude that
\begin{equation}
\height_{\cL} x = (1/A) \sum_v \ord^{\min}_v F \log q_v.
\label{eq:heightmin}
\end{equation}

\bigskip

The classical theory of Weil heights is often set up by defining heights on projective spaces, and then defining a height $\height_{\OO(1)}$ on $X(K)$ for other projective schemes $X$ by restriction. In a similar manner, one can define height functions on weighted projective stacks $\P(a_0,\dots,a_n)$ and obtain a height function $\height_{\cL}$ on $\cX(K)$ whenever $\cL$ is a generically globally generated power as in Section~\ref{subsec:computing-heights}. However, we stress that this naive approach does \emph{not} apply to all stacks of interest. Indeed, if $\cX$ is any stack with a non-abelian stabilizer group, it does not embed into a weighted projective stack, hence the necessity of our construction of heights given in Section~\ref{subsec:heights}.

\bigskip

One example of weighted projective stacks which is of great interest is the moduli stack of elliptic curves $\overline{\cM}_{1,1}$. If $K$ is a field of characteristic not equal to 2 or 3, this stack is isomorphic over $K$ to the weighted projective line $\P(4,6)$:~concretely, given an elliptic curve $E/K$, we can write it in Weierstrass form $y^2 = x^3 + Ax + B$ with $A,B$ in $K$.  This Weierstrass form is unique up to transformations $(A,B) \ra (\lambda^4 A, \lambda^6 B)$.  So $(A:B)$ is a well-defined point on $\P(4,6)$.  Moreover, the isomorphism takes the line bundle $\OO(1)$ on $\P(4,6)$ to the Hodge bundle $\cL$ on $\overline{\cM}_{1,1}$ (the bundle whose $k$th powers have weight $2k$ modular forms as sections). We conclude that, if $E/K$ is an elliptic curve over a global field of characteristic at least $5$, with Weierstrass equation $y^2 = x^3 + Ax+B$, thought of as a $K$-point of $\overline{\cM}_{1,1}$, we have
\beq
\label{eq:naiveheight}
\height_{\cL} E =  \log \max (|A|^{1/4},|B|^{1/6}).
\eeq
In other words, the familiar ``naive height" of an elliptic curve is indeed a height in the sense of this paper. 

When $K$ is a number field, the identification of $\overline{\cM}_{1,1}/\Q$ with $\P(4,6)/\Q$ does not extend to $\Spec \Z$, but only to $\Spec \Z[1/6]$.  However, this is enough to ensure that $\cL^{12}$ is still {\em generically} globally generated by $A^3$ and $B^2$ in the sense of Proposition~\ref{pr:heightcompute}, and so \eqref{eq:naiveheight} still holds up to a bounded error term.

When $K$ is a global function field of characteristic at least $5$, we can also apply \eqref{eq:heightmin}; here $A = \lcm(4,6) = 12$ and the discriminant $\Delta$ is a natural section of $\cL^{12}$ to use.  So we find
\begin{equation}
\height_{\cL} E = \frac{1}{12}\sum_v \ord_v^{\min} \Delta
\label{eq:heightecmindisc}
\end{equation}
where $\ord_v^{\min} \Delta$ is the valuation of the discriminant of a Weierstrass equation for $E$ which is minimal at $v$.

We will return to the interesting case where $K$ is a global function field of characteristic $2$ or $3$ in Section~\ref{ss:av}.

More generally, the moduli space of hyperelliptic curves over $K$ with a marked Weierstrass point can be thought of as a weighted projective space as long as the characteristic of $K$ is large enough:  if $Y \ra \P^1$ is the hyperelliptic map, we can move the image of the marked Weierstrass point to $\infty$ and (assuming the characteristic of $K$ is not $2$) complete the square in $y$, so that the curve has affine equation
\beq
y^2 = x^{2g+1} + a_1 x^{2g} + \cdots + a_{2g+1}
\eeq
then (again throwing out a finite set of characteristics for $K$) modify by the automorphism $x \ra x+\frac{a_1}{2g+1}$ of $\P^1$ in order to make $a_1 = 0$.  We now have an equation for $Y$ of the form
\begin{equation}
y^2 = x^{2g+1} + a_2 x^{2g-1} + \cdots + a_{2g+1}
\label{eq:weierstrasshe}
\end{equation}
which is unique up to the operation of multiplying $a_i$ by $\lambda^{2i}$ for $\lambda \in K^*$.  In other words, the moduli stack of hyperelliptic curves with marked Weierstrass point is isomorphic over $K$ to the weighted projective $(2g-1)$-space $\P(4,6,8,\ldots,4g+2)$.  So a hyperelliptic curve over $K$ can be thought of as a point $x$ on $\P(4,6,8,\ldots, 4g+2)$, whose height with respect to $\OO(1)$ we have computed above.  In particular, if $Y$ is a hyperelliptic curve over $\Q$ with Weierstrass equation \eqref{eq:weierstrasshe}, where the $a_i$ are chosen to be integers so that there is no prime $p$ with $p^{2i}|a_i$, the height of $Y$ is $\log \max |a_i|^{1/2i}$, which again is equivalent to the notion of height typically used for hyperelliptic curves with a specified Weierstrass point as in, e.g., the work of Bhargava and Gross~\cite{bhargavaG:average-2-selmer-hyperelliptic}. 

\begin{question}  A weighted projective space is an example of a toric stack, as in \cite{GeraschenkoSatrianoToricI}. What is the height of a rational point on a more general toric stack?
\end{question}

\subsection{Heights of abelian varieties}
\label{ss:av}

We have established above in \eqref{eq:naiveheight} that, when $K$ is a global field of characteristic at least $5$, the height of an elliptic curve with respect to the Hodge bundle on $\overline{\cM}_{1,1}$ is the same as the customary naive height.  There is another natural height on an elliptic curve over a global field: the {\em Faltings height} $\height_{\Fal}(E)$.  In this section we study the extent to which Faltings height can be seen as a height in the sense of the present paper. 

We note first that Faltings height satisfies some of the same formal properties as the heights defined in this paper do.  For example:  if $L/K$ is a field extension, it is not necessarily the case that $\height_{\Fal}(E/L)$ is $[L:K]\height_{\Fal}(E/K)$; however, this equality {\em does} hold if $E/K$ has everywhere semistable reduction, so we can define a stable Faltings height $\height_s(E/K)$ to be $[L:K]^{-1}\height_{\Fal}(E/L)$ for any $L/K$ such that $E/L$ has everywhere semistable reduction.  The height $\height_{\cV}$ for any vector bundle on $\overline{\cM}_{1,1}$ has the same properties, since an elliptic curve over $L = K(C')$ with everywhere semistable reduction is an integral point of $\overline{\cM}_{1,1}$, i.e., a morphism from $C'$ to $\overline{\cM}_{1,1}$. Lastly, $\height_{\Fal}(E/K) - \height_s(E/K)$ has a canonical local decomposition, just as does $\height_{\cV}(E/K) - \height^{\st}_{\cV}(E/K)$, see (\ref{eq:lsh}).

It is thus natural to ask whether Faltings height is $\height_\cV$ for some vector bundle $\cV$, or at least whether the two heights differ by a bounded function.  One can even guess which vector bundle one might use; for everywhere semistable $E/K$, or in other words morphisms $f\colon C \ra \overline{\cM}_{1,1}$, we have
\beq
\height_{\Fal}(E) = \deg f^* \cL
\eeq
where $\cL$ is the Hodge bundle $\Omega^1_{\mathcal{E}/\overline{\cM}_{1,1}}$ and $\mathcal{E}$ the universal semielliptic curve over the moduli stack.  

So does $\height_{\Fal}$ differ from $\height_{\cL}$ by a bounded function?  Unfortunately, the answer is in general no -- remember, in the number field case, $\height_{\cL}$ is naive height, and the difference between the naive height and the Faltings height of an elliptic curve over a number field $K$ is {\em not} bounded, as one can see for instance in the proof of \cite[Lemma 3.2]{pazuki:Modular-invariants-and-isogenies}.

The reason for this is the following.  When $K$ is a number field, the specification of the degree above requires a choice of metrization on $\cL$ at the archimedean places; for Faltings height, the appropriate Hermitian norm actually has a singularity at the cusp of moduli space, and in the present paper we have not considered metrized line bundles in this level of generality; rather, we have assumed that our choice of metrization on $\cL$ is defined on all of $\overline{\cM}_{1,1}(\C)$, including the cusp.

However, when $K$ is a global function field, this archimedean issue is absent, and we find the following.

\begin{proposition}  Let $K$ be a global function field of characteristic at least $5$, and let $\cL$ be the Hodge bundle on $\overline{\cM}_{1,1}$ as above.  Then
\beq
\height_{\Fal}(E) = \height_{\cL}(E)
\eeq
for all elliptic curves $E/K$.
\end{proposition}

\begin{proof}  For global function fields of characteristic larger than $3$, the Faltings height of $E/K$ is $(1/12)$ times the sum over all places of the valuation of a minimal discriminant: see e.g., \cite[Def 2.2]{bandini}. We have already seen in \eqref{eq:heightecmindisc} that $\height_\cL(E)$ is given by the same expression.
\end{proof}

The case of small characteristic is a different story.  Let $K$ be the function field of a curve $C$ in characteristic $p$. Then the Faltings height of an elliptic curve over $K$ is still the valuation of a minimal discriminant divisor on $C$, even if the characteristic of $K$ is $2$ or $3$, and the Faltings height has the Northcott property.\footnote{We do not know a citation for this fact in the published literature, but learned it via personal communication from Xinyi Yuan.}  

On the other hand, $\height_\cL$ is {\em not} Northcott in this setting,  Note for instance that $\overline{\cM}_{1,1}/\mathbf{F}_3$ contains as a closed substack a copy of $BG$ lying over the coarse point $j=0=1728$, where $G$ is the automorphism group scheme of an elliptic curve with $j$-invariant $0$ in characteristic $3$.  The group scheme $G$ has order $12$ and sits in an exact sequence
\beq
1 \ra A \ra G \ra \mu_4 \ra 1
\eeq
where $A \cong \Z/3\Z$ (see for instance \cite[Exercise A.1]{silverman:aec}) and $\lambda \in \mu_4$ acts on $A$ by multiplication by $\lambda^2$.  The pullback of $\cL$ to $BG$ is a line bundle on $BG$, which is necessarily trivial on the commutator subgroup $A$.  So $\cL$ pulls back to the zero bundle under the composition $BA \ra BG \ra \overline{\cM}_{1,1}$, which means that any point $x$ in the image of $BA(K) \ra \overline{\cM}_{1,1}(K)$ has $\height_{\cL}(x) = 0$.  There are infinitely many such points, corresponding to the $\Z/3\Z$-extensions of $K$.  Concretely, elliptic curves given by Weierstrass equations of the form
\begin{equation}
y^2 = x^3 - x - f(t)
\label{eq:badec3}
\end{equation}
all have height $0$ with respect to $\cL$.  Another way to see this is to observe that the space of sections of $\cL^{12}$ -- that is, of weight-$12$ modular forms of level $1$ in characteristic $3$ -- is two-dimensional and is spanned by $\Delta$ and $b_2$, where $b_2$ is the Hasse invariant.~\cite[Prop 6.2]{deligne:formulaire}.  Any Weierstrass equation of type \eqref{eq:badec3} has $b_2(E) = 0$ and $\Delta(E) = 1$ (\cite[Appendix A, Prop 1.1.~b)]{silverman:aec}).  So by Proposition~\ref{pr:heightcompute}, using the fact that $\Delta$ is constant, we see again that $\height_{\cL}(E) = 0$ for any such $E$.

This does not mean, however, that Faltings height is a different kind of height from those discussed in this paper; it only means it does not agree with the height arising from the Hodge bundle or any of its powers.  But, as explained in a paper of Meier \cite{meier:vb}, there are other vector bundles!  When $K$ is a field of characteristic greater than $3$, every vector bundle on $\overline{\cM}_{1,1}$ is isomorphic to a direct sum of line bundles, which can only be powers of the Hodge bundle~\cite[Cor 3.6]{meier:vb}), essentially because $\overline{\cM}_{1,1}$ is a weighted projective line in this case.  But in characteristic $2$ and $3$, Meier constructs indecomposable higher-rank vector bundles on $\overline{\cM}_{1,1}/K$.\footnote{Meier only describes these bundles on $\cM_{1,1}$ but it is not hard to show they extend to the compactification.}  Thus, the following question still makes sense.  

\begin{question}[A. Landesman]
When $K$ is a global field of characteristic $2$ or $3$, is there a vector bundle $\cV$ on $\overline{\cM}_{1,1}/K$ such that $\height_{\cV} =  c \height_{\Fal}$ for some $c \in \Z$?
\label{q:meier}
\end{question}

\medskip

We originally asked this question with $c=1$; i.e., is the Faltings height itself a height in our sense?  Landesman showed in his thesis~\cite{landesman:thesis} that this is too much to hope for; in characteristic $3$, there is no vector bundle $\cV$ on $\overline{\cM}_{1,1}/\mathbf{F}_3$ with $\height_{\cV} =  \height_{\Fal}$.  However, Landesman then raises the question stated above, which remains open:  is there a vector bundle which computes some integer multiple of the Faltings height in characteristic $3$?  For that matter, is there even a vector bundle whose associated height is Northcott?


Furthermore, one may ask the same question about abelian varieties of higher dimension.  The Faltings height is usually thought of as being related to the Hodge bundle on the moduli stack $\bar{\cA}_g$.  But the stacky height associated to this line bundle, or {\em any} line bundle, will not be Northcott on $\bar{\cA}_g$, for the same reason it failed to be Northcott for $\overline{\cM}_{1,1}$ in low characteristic; there are abelian varieties of dimension $d$ with non-abelian automorphism group, which give rise to maps $BG \inj \bar{\cA}_g$ for nonabelian $G$, and no line bundle on $BG$ can be Northcott.  This problem can be avoided by computing heights on $\bar{\cA}_g$ with respect to the rank-$g$ vector bundle $\cV = e^* \Omega^1_{A/\bar{\cA}_g}$, where $A$ is the universal principally polarized abelian variety over the moduli stack, rather than with respect to its determinant, the Hodge bundle.  There will still be problems in low characteristic, as we have seen from the case of elliptic curves.  One way of understanding the difficulty with curves of the form \eqref{eq:badec3} is that a wildly ramified extension of $K$ is necessary in order to arrive at a curve with semistable reduction; this cannot be the case for elliptic curves over fields of characteristic $5$ or greater.  The following question thus seems reasonable.

\begin{question}
When $K$ is a global function field, $\cV$ is the vector bundle $e^* \Omega^1_{A/\overline{\cA}_g}$ on $\overline{\cA}_g$, and $A/K$ is an abelian variety that becomes semistable over an everywhere tamely ramified extension of $K$, is it the case that
\beq
\height_{\cV}(A) = \height_{\Fal}(A)?
\eeq
\label{q:av}
\end{question}

If Questions~\ref{q:meier} and \ref{q:av} both have a positive answer, one might well ask the common descendant of both questions:  are there ``exotic" vector bundles on $\cA_g$ in small (relative to $g$) characteristic which compute the Faltings height of abelian varieties that require a wild extension to become semistable?

Finally, we return for a moment to the number field case.  Because of the singularity at the boundary of $\overline{\cA}_g$ of the Faltings metric, we cannot expect $\height_{\cV}$  to match $\height_{\Fal}$ exactly.  But there is a way to ask whether the two heights agree ``apart from the archimedean place."  Namely, we can ask the following.

\begin{question}
Let $K$ be a global field, let $v$ be a nonarchimedean place of $K$, and let $A/K$ be an abelian variety which becomes semistable over a tamely ramified extension of $K_v$.  Is the component at $v$ of $\height_{\Fal}(A) - \height_s(A)$ equal to $\delta_{\cV;v}(A)$?
\label{q:localfaltings}
\end{question}

This is a purely local question which has to do with the behavior of the tangent space to the N\'{e}ron model of $A$ under ramified base change.  A positive answer to Question~\ref{q:localfaltings} would imply a positive answer to Question~\ref{q:av}, as follows.  The stable Faltings height $\height_s(A)$ agrees with $\height^{\st}_{\cV}$ in this setting, because both are given by the degree of the pullback of the Hodge bundle to an integral point $C' \ra \overline{\cA}_g$, where $C'$ is a cover of $C$.  And since there are no archimedean places, the positive answer to Question~\ref{q:localfaltings} shows that
\beq
\height^{\Fal}(A) - \height_s(A) = \sum_v \delta_{\cV;v}(A) = \height_{\cV}(A) - \height^{\st}_{\cV}(A).
\eeq

\subsection{Heights on footballs}
\label{ss:footballs}

A football\footnote{The ``football" here is understood to be an American football, which has two singular points.  In the professional sporting context, the residual gerbes at these points are not specified.} $\mathcal{F}(a,b)$ is a $\P^1$ rooted at $0$ and $\infty$, with residual gerbes $\mu_a$ and $\mu_b$ respectively.  Let $K$ be a global field; we emphasize that $K$ is allowed to have any characteristic, including characteristics dividing $a$ or $b$.  (When $K$ has one of these characteristics, $\mathcal{F}(a,b)$ is a tame Artin stack but not a Deligne--Mumford stack.)  As an illustration of the (moderate) subtlety of the Northcott condition in the stacky case, we will work out which line bundles on $\mathcal{F}(a,b)$ are Northcott.

There are three kinds of $K$-points of $\mathcal{F}(a,b)$, which may be treated separately.
\begin{itemize}
\item The points supported at $0$; these are naturally identified with $K$-points of $B(\mu_a)$, which are in turn identified with the set $K^* / (K^*)^a$;
\item The points supported at $\infty$; these are naturally identified with $K$-points of $B(\mu_b)$, which are in turn identified with the set $K^* / (K^*)^b$;
\item The rest of the points, which are naturally identified with the points on $\P^1(K)$ other than $0$ and $\infty$; that is, these points are in bijection with $K^*$.
\end{itemize}

Any divisor on $\mathcal{F}(a,b)$ is linearly equivalent to one of the form $d[P] + n[0] + m[\infty]$, where $P$ is some point on $\mathbb{G}_m$; such a divisor has degree $d + n/a + m/b$.   This expression is not unique, but is subject to the relations $a[0] \sim b[\infty] \sim [P]$.  Take $\mathcal{L}$ to be the line bundle on $\mathcal{F}(a,b)$ corresponding to $d[P] + n[0] + m[\infty]$.  We now explain how to compute $\height_{\mathcal{L}}(x)$ for $x \in \mathcal{F}(a,b)(K)$.

For the first two types of points, this computation of height  has already been carried out in Section~\ref{subsec:computing-heights}.  For a point $x$ of the first type, $d$ and $m$ are irrelevant.  The class in $K^* / (K^*)^a$ associated to $x$ is represented by a function $f \in K^*$, and the height of $x$ is  a sum over places of $K$:
\beq
\height_{\L}(x) = \sum_v \left\lceil \frac{n}{a} \ord_v(f) \right\rceil.
\eeq

Similarly, for a point of the second type, represented by the class of $g$ in $K^* / (K^*)^b$  the height is 
\beq
\height_{\L}(x) = \sum_v \left\lceil \frac{m}{b} \ord_v(g) \right\rceil.
\eeq

We now treat points of the third, or generic type.  For simplicity of description, take $K$ to be the function field of a smooth proper curve $C/\FF_q$.  Then $x$ affords a rational map $\phi$ from $C$ to $\mathcal{F}(a,b)$.  Write $\phi_c\colon C \ra \P^1$ for the composition of $\phi$ with the coarse moduli map, denote $\deg \phi_c = \deg \phi$ by $e$, and write  $\sum e_i P_i$ for the divisor $\phi_c^* [0]$ and $\sum e_i' Q_i$ for the divisor $\phi_c^* [\infty]$. Then $\sum e_i \deg P_i = \sum e'_i  \deg Q_i = e$. 

We may take $\cC$ to be a root stack with residual gerbe $\mu_a$ at the $P_i$ and  $\mu_b$ at the $Q_i$.  Then $\overline{x}^* \mathcal{L}^\vee$ is the divisor 
\beq
-\left(d\phi^{-1}(P) + \sum_i \frac{e_i n}{a} P_i + \sum_i \frac{e'_i m}{b} Q_i\right)
\eeq
whose degree, as it must be, is $-e \deg \mathcal{L}$.

We then have
\begin{eqnarray}
\height_{\mathcal{L}}(x) & = & - \deg \pi_* \overline{x}^* \mathcal{L} = -\left(-ed + \sum_i \left\lfloor -\frac{e_i n}{a} \right\rfloor \deg P_i + \sum_i \left\lfloor -\frac{e'_i m}{b} \right\rfloor \deg Q_i\right) \log q \\
& = & \left(ed + \sum_i \left\lceil \frac{e_i n}{a} \right\rceil \deg P_i + \sum_i \left\lceil \frac{e'_i m}{b} \right\rceil \deg Q_i \right) \log q.
\label{eq:footballheight}
\end{eqnarray}

In particular, we note that $\height_{\mathcal{L}}(x) \geq e \log q \deg \mathcal{L}$, with equality holding just when every $e_i$ is a multiple of $a$ and every $e'_i$ is a multiple of $b$, which is to say, when $x$ actually extends to an integral point of $\mathcal{F}(a,b)$.

This description suffices to tell us which line bundles have the Northcott property.  We already see that the set of Northcott line bundles does not form a cone, because it is not closed under addition.  (Indeed, we could have already seen that from the case $B(\Z/2\Z)$, where the nontrivial line bundle $\mathcal{L}$ is Northcott and $\mathcal{L}^{\otimes 2}$, which is trivial, is not Northcott.)

\begin{proposition} Choose $a,b$ coprime integers and let $K$ be the function field of a curve $C$. A divisor $L = d + n[0] + m[\infty]$ on $\mathcal{F}(a,b)$ is Northcott if and only if $\deg L > 0$ and $(n,a) = (m,b) = 1$.
\end{proposition}

\begin{proof}
Suppose $(n,a) = r > 1$.  Then any point of $\PP(a,b)$ of the first type which corresponds to $f \in (K^*)^{a/r}/(K^*)^a \subset K^* / (K^*)^a$ has height $0$ with respect to $L$, which contradicts Northcott.  The argument is just the same if $(m,b) > 1$.

We observe that there are infinitely many maps $\P^1 \ra \mathcal{F}(a,b)$; namely, those whose coarse map $\P^1 \ra \P^1$ is of the form $[B(s,t)^b:A(s,t)^a]$.  Any such map, pulled back to $C$ via a map $C \ra \P^1$, gives an integral point $C \ra \mathcal{F}(a,b)$ of some coarse degree $e$, whose height is $e \deg L$; we can make $e$ as large as we want, which shows that $L$ cannot be Northcott if $\deg L \leq 0$.

Suppose, on the other hand, that all three conditions are met.  We have already shown that points $x$ of the third type have $\height_L(x) \geq e \log q \deg L$; since $\deg L$ is positive, $\height_L(x)$ gives an upper bound for $e$, which makes the set of possible $x$ finite.  For points of the first type represented by $f \in (K^*)/(K^*)^a$, we observe that
\beq
\height_{\cL}(x) = \sum_v \left\lceil \frac{n}{a} \ord_v(f) \right\rceil = \sum_v \left\{\frac{n}{a} \ord_v(f) \right\}
\eeq
the latter equality following from $\sum_v \ord_v(f) = 0$.  So a bound on the height of $x$ yields a bound on the number of places where $\frac{n}{a} \ord_v(f)$ is not an integer; since $(n,a) = 1$, this bounds the number of places where (more precisely:  the degree of the divisor where) $\ord_v(f)$ is not a multiple of $a$.  Bounding this quantity places $f$ within a finite set of cosets of $(K^*)^a$, so we are done.  The case of points of the second type is exactly the same.
\end{proof}

\subsubsection{Consistency check: footballs and weighted projective lines}

When $a$ and $b$ are relatively prime, the football $\mathcal{F}(a,b)$ is isomorphic to the weighted projective line $\PP(a,b)$; on $K$-points, the isomorphism $\psi$ from $\PP(a,b)$ to $\mathcal{F}(a,b)$ sends $(s:t)$ to the point $t^a/s^b$ when $st \neq 0$.  Let $m,n$ be integers such that $ma+nb = 1$; then the line bundle  $\cL = n[0] + m[\infty]$ on $\FF(a,b)$ has degree $1/ab$, and its pullback to $\PP(a,b)$ is the tautological bundle $\OO_{\PP(a,b)}(1)$.  If $x$ is a point of $\PP(a,b)(K)$, we have
\beq
\height_{\OO_{\PP(a,b)}(1)}(x) = \height_{\cL}(\psi(x))
\eeq
This provides an opportunity to check consistency between the formulas we have given for the height of a point on weighted projective space and the height of a point on a football.  Let $x = (s:t)$ be a point of $\PP(a,b)$.  Then by \eqref{eq:weightedheight} we have
\begin{equation}
\height_{\OO_{\PP(a,b)}(1)}(x) = \sum_v \lceil \log_{q_v}  \max (|s|_v^{1/a},|t|_v^{1/b}) \rceil \log q_v.
\label{eq:pab}
\end{equation}
We now compute  $\height_{\cL}(\psi(x))$.  Recall that $\psi(x)$ is the point on $\FF(a,b)$ corresponding to the point $t^a/s^b$ of $\P^1(K)$.  In the notation of the above section, the points $P_i$ correspond to those places $v$ of $K$ where $a \ord_v t - b \ord_v s > 0$, and the points $Q_i$ to those places where $a \ord_v t - b \ord_v s < 0$.  When $v$ is a prime with $a \ord_v t - b \ord_v s > 0$, we have, again maintaining the notation of \eqref{eq:footballheight},
\beq
e_i = a \ord_v t - b \ord_v s 
\eeq
and
\beq
\deg P_i = \log q_v / \log q.
\eeq
So the contribution of $v$ to \eqref{eq:footballheight} is
\beq
\left\lceil \frac{(a \ord_v t - b \ord_v s)n}{a} \right\rceil \log q_v =  \left( n \ord_v t  - \left\lceil \frac{nb \ord_v s}{a} \right\rceil \right)\log q_v
= \left( n \ord_v t + m \ord_v s - \left\lceil \frac{\ord_v s}{a} \right\rceil \right) \log q_v.
\eeq
By a similar argument, one shows that when $a \ord_v t - b \ord_v s < 0$ one gets a contribution of
\beq
 \left( n \ord_v t + m \ord_v s - \left\lceil \frac{\ord_v t}{b} \right\rceil \right) \log q_v
\eeq
Since the first case obtains exactly when $\ord_v s / a < \ord_v t / b$, we can express the contribution of $v$ uniformly as
\beq
\left( n \ord_v t + m \ord_v s - \left\lceil \min \left(\frac{\ord_v s}{a} , \frac{\ord_v t}{b} \right) \right\rceil \right) \log q_v.
\eeq
Summing over $v$, the first two terms vanish by the product formula, and we are left with
\beq
\height_{\cL}(\psi(x)) = -\sum_v \min \left(\left\lceil \frac{\ord_v s}{a} \right\rceil, \left\lceil \frac{\ord_v t}{b} \right\rceil \right) \log q_v
\eeq
which is just \eqref{eq:pab} in another form.

\subsection{Heights on symmetric powers of varieties}
\label{s:sympowers}

There is a substantial literature about points on varieties of bounded algebraic degree.  We explain how these questions look through the lens of heights on stacks.  Let $X$ be a smooth proper scheme of dimension $n$ over $K$.  A point on $X$ of algebraic degree $m$ over $K$ can be thought of as a $K$-point on the stack $\Sym^m X = [X^m / S_m]$.  In this section, we explain how to compute the height of such a point.  Slightly more generally, let $G$ be a subgroup of $S_m$, and let $\cX$ be the quotient $[X^m / G]$; when $G = S_m$, our stack $\cX$ is $\Sym^m X$. 

In order to talk about height, we need to choose a vector bundle $\cV$ on $\cX$; this is the same thing as an $G$-equivariant vector bundle on $X^m$.  The choice we make is as follows:  let $V_0$ be some vector bundle of rank $r$ on $X$, and let $\pi_1,\ldots, \pi_m\colon X^m \ra X$ be the $m$ projections.   Then $\widetilde{V} = \oplus_i \pi_i^* V_0$ is an $G$-equivariant vector bundle of rank $mr$, which descends to a vector bundle $\cV$ of rank $mr$ on $\cX$.

Let $x$ be a point of $\cX(K)$.  We begin by computing the stable height $\height^{\st}_\cV(x)$.  The Cartesian square
  \[
\xymatrix{
 \Spec L\ar[r]^{x_L} \ar[d] & X^m\ar[d] \\
 \Spec K  \ar[r]^{x}  & \cX
}
\]
provides an \'etale algebra $L$ over $K$ which carries an $S_m$-action, and a rational point $x_L$ which extends to an integral point $C \ra X^m$.  By Proposition~\ref{prop:stable-height-is-stable},
\beq
\height^{\st}_{\cV}(x) = [L:K]^{-1}\height^{\st}_{\widetilde{V}}(x_L).
\eeq
(Should $L$ be an \'{e}tale algebra which is not  a field but rather a direct sum $\oplus_i F_i$, our convention is that the height of a point of $X^m(L)$ is $\sum_i \height(P_i)$, where $P_i \in X^m(F_i)$ are the points corresponding to the restriction of $x_L\colon \Spec L \ra X^m$ to connected components of $\Spec L$.) 

Since $X^m$ is a scheme, we have
\beq
\height^{\st}_{\widetilde{V}}(x_L) = \height_{\widetilde{V}}(x_L).
\eeq
The latter quantity is a very natural one, what you might call the ``absolute height" of $x$.  Suppose, for instance, that $L/K$ is a field extension, necessarily Galois with Galois group $G$. Then $x_L$ is a point of $X^m(L)$ on which $\Gal(L/K)$ acts by permutations; in other words, it is an element $(\alpha_1, \ldots, \alpha_m)$ where the $\alpha_i$ are conjugate and each $\alpha_i$ is contained in a degree-$m$ extension $L_i/K$ whose Galois closure is $L$.  The (unordered) set $\alpha_1, \ldots, \alpha_m$ can be thought of as a $K$-rational Galois orbit of points on $X$, and the height of $x_L$ is then given by the usual Weil height on $X$:
\beq
\height_{\widetilde{V}}(x_L) = \sum_i \height_{V_0;L} \alpha_i = m \height_{V_0;L} \alpha_1
\eeq
where the subscript $L$ is indicating that the  height of $\alpha_i$ is understood to mean the height of $\alpha_i$ as a point of $X(L)$, not of $X(L_i)$; to sum up, this means that 
\beq
\height^{\st}_{\cV}(x) = |G|^{-1} \height_{\widetilde{V}}(x_L) = m |G|^{-1} \height_{V_0;L} \alpha_i = \height_{V_0;L_i} \alpha_i
\eeq
which is the same for every $i$.  In fact, the reader will note that nothing we did actually used the hypothesis that $L$ was a field, so the description of the stable height of $x$ is valid also in the case where $L$ is an \'etale algebra other than a field.  For instance, if $L$ splits completely as a product of copies of $K$, then $L_i$ is isomorphic to $K^m$, and our point $x \in \XX(K)$ may be thought of as an unordered $m$-tuple $\set{Q_1, \ldots, Q_m} \subset X(K)$; in that case,
\beq
\height^{\st}_{\cV}(x) =  \height_{V_0;L_i} (Q_1, \ldots Q_m) = \sum_{i=1}^m \height_{V_0;K} Q_i.
\eeq

We now consider the discrepancy $\delta_{\cV}(x) = \height_{\cV}(x) - \height^{\st}_{\cV}(x)$.

\begin{proposition} The value of $\delta_{\cV}(x)$ is the same for any $V_0$ and $V'_0$ of the same rank $r$.
\label{pr:indepofv0}
\end{proposition}

\begin{proof} We write $\widetilde{V}',\cV'$ for the vector bundles on $X^m$ and $\cX$ respectively obtained from $V'_0$ as $\widetilde{V},\cV$ were obtained from $V_0$.

The discrepancy is a sum of local terms $\delta_{v;\cV}(x)$ where $v$ ranges over a finite list of non-archimedean places $v$ of $C$ where $x$ does not extend to an $\OO_{K_v}$-point; in particular, this list depends only on $x$, not on the choice of $\cV$.  Choose such a $v$; denoting by $\cC_v$ the infinitesimal neighborhood of the tuning stack $\cC$ over $v$, we have a commutative diagram
  \[
\xymatrix{
 \Spec \OO_{L_v} \ar[r]^-{\overline{x}_{L_v}} \ar[d] & X^m\ar[d] \\
 \cC_v  \ar[r]^{\overline{x}_{K_v}}  & \cX
}
\]
where $L_v$ denotes $L \tensor_K K_v$, so $\OO_{L_v}$ is a disjoint union of dvrs.  Composing $\overline{x}_{L_v}$ with the projection maps $p_1, \ldots, p_m$ yields maps $q_1, \ldots, q_m\colon \Spec \OO_{L_v} \ra X$ which are permuted by composition with the action of $G$ on $\Spec \OO_{L_v}$.  We may take $U \subset X$ to be an open subscheme containing the image of the $q_i$ on which $V_0$ and $V'_0$ become isomorphic (and indeed we may choose $U$ to make both isomorphic to $\OO_U^r$). 

Now $\overline{x}_{L_v}^* \widetilde{V}$ can be described as $\oplus_i q_i^* V_0$, where the action of $G$ permutes the factors; we note that this is {\em $G$-equivariantly} isomorphic to $\overline{x}_{L_v}^* \widetilde{V}' =  \oplus_i q_i^* V'_0$.  Thus,  the vector bundle $\overline{x}_{K_v}^* \cV$, which is the descent of $\overline{x}_{L_v}^* \widetilde{V}$, is isomorphic to $\overline{x}_{K_v}^* \cV'$.   Since $\delta_{v;\cV}(x)$ depends only on $\overline{x}_{L_v}^* \widetilde{V}$, we conclude that
\beq
 \delta_{v;\cV}(x) = \delta_{v;\cV'}(x)
\eeq
as desired.
\end{proof}

Given Proposition~\ref{pr:indepofv0}, we are free to take $V_0 = \OO_X^r$ when computing $\delta_{\cV}(x)$.  In this case, $V$ is the direct sum of $r$ copies of the vector bundle on $\cX$ obtained by taking $V_0 = \OO_X$; so we may simply take $V_0 = \OO_X$ and multiply by $r$ at the end.  

In this case, we can describe $\cV$ very concretely; in the diagram
\[
\xymatrix{
 X^m \ar[r]^{} \ar[d] & \ast \ar[d] \\
 \cX  \ar[r]^{c}  & BG
}
\]
the rank-$m$ vector bundle $\cV$ on $\cX$ is just $c^* \rho$, where $\rho$ is the rank-$m$ vector bundle on $BG$ corresponding to the $m$-dimensional permutation representation of $G$ afforded by our embedding $G \inj S_m$.  This description makes it easy to compute $\height_{\cV}(x)$.  Extending the diagram above to
  \[
\xymatrix{
 \Spec L\ar[r]^{x_L} \ar[d] & X^m\ar[r] \ar[d] & \ast \ar[d]\\
 \Spec K  \ar[r]^{x}  & \cX \ar[r]^{c} & BG
}
\]
we have that $\height_{\cV}(x) = \height_{\rho}(c \circ x)$, where $c \circ x$ is the morphism from $\Spec K$ to $BG$ corresponding to the etale $G$-extension $L/K$.  It follows from Proposition \ref{pr:heightdisc} that $\height_{\rho} (c \circ x) = (1/2) \log \Delta_{L_i/K}$ (which is the same for all $i$).  The pullback of $\rho$ to $\ast$ is trivial, so $\height^{\st}_{\rho}$ is identically $0$, whence the discrepancy $\delta_{\rho}(c \circ x)$ is also $(1/2) \log \Delta_{L_i/K}$.  We can now conclude from the discussion above that, for any choice of $V_0$,
\beq
\delta_{\cV}(x) = (r/2) \log \Delta_{L_i/K}.
\eeq
Combining this with our computation of $\height^{\st}_{\cV}$, we finally arrive at a description of the height of a rational point on $\cX$ with respect to $\cV$.  Recall that a point $x \in \cX(K)$ provides us with a degree-$m$ etale extension $L_1/K$ and a point $\alpha_1 \in X(L_1)$.  Denote by $\height^W_{L_1}(\alpha_1)$ the usual Weil height of $\alpha_1$ under the map $X(L_1) \ra \R$ afforded by $V_0$.  Then
\beq
\height_{\cV}(x) = \height^W_{L_1}(\alpha_1)+ (r/2) \log \Delta_{L_1/K}.
\eeq

\section{Counting rational points by height:  a conjecture of Batyrev--Manin--Malle type}
\label{sec:count-rati-points}

In this section, we formulate a conjecture of Batyrev--Manin--Malle type for rational points of bounded height on a stack $\cX$. When $\cX$ is a scheme, we recover the weak Batyrev--Manin conjecture about rational points on schemes; when we take $\cX=BG$, we recover the weak Malle conjecture. We thus think of our conjecture as interpolating between the two conjectures, while at the same time generating many new cases of interest.  As was the case for the original Batyrev--Manin, we develop our heuristics by consideration of the case $K = k(t)$ and the corresponding geometric problem of   studying spaces of rational curves on $\cX$.

By ``weak" in the above paragraph we mean that we propose conjectures that bound counting functions between $X^a$ and $X^{a+\epsilon}$ for a specified exponent $a$.  The ``strong" versions of Batyrev--Manin and Malle make a more precise conjecture, that counting functions are asymptotic to $X^a (\log X)^b$ for specified $a,b$.  In work posted after the original version of this paper was released, Darda and Yasuda~\cite{dardayasudabatman} have proposed a ``strong" conjecture about point-counting on stacks, with an explicit predicted power of $\log X$.

One could go further still and ask whether the counting functions discussed here are of the form $c X^a (\log X)^b + o(X^a (\log X)^b)$, with an explicit constant $c$; this has been quite an active area of investigation in both the Batyrev--Manin and the Malle context.  One remark in this regard: to get constants right, it is presumably important to remember that $\XX(K)$ is naturally not a set but a groupoid, and counts of points should probably be weighted inversely to the size of the point's automorphism group.  But issues of this kind will not be relevant for the coarser heuristics considered here.

\subsection{Expected deformation dimension: stacky anti-canonical height}
\label{subsec:edd}

In the Batyrev--Manin Conjecture for a scheme $X$, when counting rational points with respect to a line bundle $\cL$, the expected growth rate is given by $B^{a(\cL)}$ where the \emph{Fujita invariant} $a(\cL)$ is the infimum of all $a$ for which $a\cL+K_X$ is effective. A technical hurdle we must overcome in defining $a(\cL)$ for stacks $\cX$ is that for many stacks of interest, e.g.~$\cX=BG$, the canonical bundle $K_\cX$ is trivial! Thus, the anti-canonical height is not suitable for the purposes of obtaining the expected growth rate of point counts on stacks. Our solution is to introduce a new quantity, the expected deformation dimension (or $\edd$), which is a suitable perturbation of the anti-canonical height.

Before giving the definition of $\edd$, we wish to sufficiently motivate it through geometric intuition. In the case of a proper scheme $X$ over a function field $\CC(t)$, a rational point $x\colon\Spec\CC(t)\to X$, by the valuative criterion, extends to a map $\overline{x}\colon\PP^1\to X$. By Riemann--Roch, the anti-canonical height $\height_{-K_X}(x)=\deg(\overline{x}^*T_X)$ differs from $\chi(\overline{x}^*T_X)$ by a constant, and $\chi(\overline{x}^*T_X)$ is the expected dimension of the deformation space of $\overline{x}^*$.

The deformation theoretic point of view serves as our launching point for the definition of $\edd$. Given a rational point $x\colon\Spec K\to\cX$ of a stack, we can extend $x$ to a \emph{universal} tuning stack $\overline{x}\colon\cC\to\cX$, see Definition \ref{def:tuning-stack}. The expected deformation dimension of $\overline{x}$ is then given by $\chi(L^\vee_{\overline{x}}[1])$ where $L_{\overline{x}}$ is the cotangent complex for the representable map $\overline{x}$. For the sake of motivational purposes, suppose both $\cX$ and $\cC$ are smooth tame Deligne--Mumford stacks, in which case the tangent complexes $L^\vee_{\cX}$ and $L^\vee_{\cC}$ are vector bundles, denoted by $T_\cX$ and $T_\cC$. Then
\[
\chi(L^\vee_{\overline{x}}[1])=\chi(\overline{x}^*T_\cX)-\chi(T_\cC),
\]
which up to constants are the same as
\begin{equation}
\label{eqn:prelim-edd-def}
\deg(\pi_*\overline{x}^*T_\cX)-\deg(\pi_*T_\cC).
\end{equation}
Note that $\deg(\pi_*\overline{x}^*T_\cX)=-\height_{K_\cX}(x)$. We next calculate $\deg(\pi_*T_\cC)$. Letting $\pi\colon\cC\to C$ be the coarse space, we have

\begin{equation}
\label{eqn:stacky-curve-canonical}
  \Omega^1_\cC=\pi^*\Omega^1_C\otimes\O_\cC\left(\sum (1-e_p^{-1})p\right)
\end{equation}
by \cite[Lemma 5.5.3 and Proposition 5.5.6]{voightZB:canonical-ring-of-a-stacky-curve}. So,
\[
\pi_*T_\cC=T_C\otimes\O_C\left(\sum \left\lfloor e_p^{-1} - 1 \right\rfloor p\right)=T_C(-R);
\]
since the floors are equal to $-1$ if $e_p$ is nontrivial and 0 otherwise, $R$ is the divisor given by the ramified points taken without multiplicity. So, up to constants, $\deg(\pi_*T_\cC)=-\deg(R)$.

In practice, however, we will want to consider stacks $\cX_0$ over $K$ for which we do not have in mind a particular model $\XX/C$ which is normal and Deligne--Mumford, or for which we do have in mind a model but it isn't Deligne--Mumford; for example, we don't want to exclude a stack like $B\mu_n / \Spec \Z$ which fails to be Deligne--Mumford in characteristics dividing $n$. Tuning stacks for rational points of such stacks are also generally not Deligne--Mumford.  Presumably a more complicated definition involving the tangent \emph{complex} would work, but in the interest of simplicity we have chosen for now to apply a technical work-around.

  First, the universal tuning stack $\cC$ of a rational point $x \in \cX(K)$ is generically a scheme (and thus generically Deligne--Mumford). The coarse space map $\pi\colon\cC\to C$ is birational and $\cC$ is normal; if $\cC$ is tame then it is a root stack. To promote our working definition of edd (Equation \ref{eqn:prelim-edd-def}) to the general setting we are tempted to define
  \[
    \Omega^{1,\fake}_\cC=\pi^*\Omega^1_C\otimes\O_\cC\left(\sum (1-e_p^{-1})p\right).
  \]
If $p$ is a Deligne--Mumford point of $\cC$ but is not tame, then one defines $e_p$ via wild ramification \cite[Proposition 7.1]{kobin:Artin-Schreier-Root-Stacks}. But if $p$ is not a Deligne--Mumford point it is unclear how to define $e_p$. If $p$ is tame, then the stabilizer of $p$ is isomorphic to $\mu_m$ for some integer $m$, and it is tempting to define $e_p$ to be $1/m$. This is ad hoc, but worse, not general enough: the stabilizer could be a group which is neither \'etale nor tame (such as $\mu_p \times \Z/p\Z$). 

Our perspective is that the precise definition of $e_p$ does not matter, as long as it is nontrivial at a stacky point. What we mean is: for the part of the definition of edd that relies on the universal tuning stack we only ever consider the quantity ``$\deg(\pi_*T_\cC)$''. Since $T_\cC$ is the dual of $\Omega^{1,\fake}_\cC$, 
  \[
\pi_*T_\cC=T_C\otimes\O_C\left(\sum \left\lfloor e_p^{-1} - 1 \right\rfloor p\right).
\]
In particular, since we are taking floors the quantity $\left\lfloor e_p^{-1} - 1 \right\rfloor$ is 0 is $p$ is not stacky and is $-1$ otherwise. In Equation \ref{eqn:prelim-edd-def} we thus abstain from defining $T_\cC$ and instead replace $\deg(\pi_*T_\cC)$ with the following quantity.

\begin{definition}[Reduced discriminant]
  \label{definition:reduced-discriminant}
  Let $\pi \colon \cC \to C$ be a tuning stack of a rational point $x \in \cX(K)$. We define the \defi{reduced discriminant} $\rDisc(x)$ of $x$ to be the sum
\[
\rDisc(x) = \sum \log q_v
\]
over the stacky points $v$ of $\cC$, where $q_v$ is the cardinality of the residue field of the point $v$.
\end{definition}

To make sense more generally of the other term of Equation \ref{eqn:prelim-edd-def}, for the rest of this section, in addition to the assumptions of Subsection \ref{subsec:not-conv} and Section~\ref{sec:Heights of rational points on stacks}, we assume that the generic fiber $\cX_K$ of our proper Artin stack $p\colon \cX \to C$ is Deligne--Mumford, so that it makes sense to talk about the canonical sheaf $K_{\cX_K}$ of the generic fiber.
\begin{definition}
We say a line bundle on $\XX$ is \defi{generically canonical} if its restriction to $\cX_K$ is $K_{\cX_K}$.  
\end{definition}

We now define $\edd$ as follows, guided by the motivation above.

\begin{definition}[Expected deformation dimension]
  \label{def:edd}
Let $K$ be a global field and let $C$ be either $\Spec O_K$ in the number field case or a smooth proper curve with function field $K$ in the function field case.  Let $\cX$ be a proper Artin stack over $C$ with finite diagonal such that $\cX$ is a smooth proper Deligne--Mumford stack over $K$.  Let $\widetilde{K}$ be a generically canonical line bundle on $\cX$.  Given $x \in \cX(K)$, let $(\cC,\overline{x},\pi)$ be its universal tuning stack. The \defi{expected deformation dimension} of $x$ is
\[
\edd(x):=-\height_{\widetilde{K}}(x)+\rDisc(x).
\]
\end{definition}

\begin{remark}
Implicit in this definition is a conjecture:  that the definition is independent of choices.  More precisely, we expect that, given two different models of $\cX_K$, and two different extensions of $K_{\cX_K}$ to these models, the two functions $\edd(x)$ would differ by a function that is bounded as $x$ ranges over $\cX(K)$.  In the examples that follow, we will simply choose a model $\cX$ and choose a generically canonical line bundle on $\cX$.
\end{remark}

\begin{remark}\label{rmk:edd-special-cases}
If $\cX=X$ is a scheme, then the universal tuning stack is a curve, and $\edd$ agrees with the anti-canonical height since $\edd(x):=-\height_{K_X}(x)=\deg(\overline{x}^*T_X)=\height_{-K_X}(x)$. On the other extreme, if $\cX=BG$ then $K_\cX$ is trivial, so $\edd(x)$ is the reduced discriminant of the field extension corresponding to $x$.
\end{remark}

\begin{example}[Extending a stacky curve and its canonical bundle]\label{example:stacky-curve-model}
  Let $\cX_0$ be a smooth tame Deligne--Mumford stacky curve over $K$ and suppose that the coarse space map $\phi_0\colon \cX_0 \to X_0$ is birational (equivalently, $\cX_0$ has trivial generic inertia).   By  \cite[Theorem 1 and Remark 4]{geraschenkoS:bottomUp}, such an $\cX_0$ is isomorphic to a root stack over its coarse space. Let $p_1,\ldots,p_k \in X_0$ be the ramification locus of $\phi_0$; since $\cX_0$ is a root stack, the stabilizer group over each $p_i$ is isomorphic to $\mu_{e_i}$ for some integer $e_i \geq 2$, and $\cX_0$ is the root stack of $X_0$ rooted along each $p_i$ with order $e_i$.

  The coarse space $X_0$ is a smooth proper curve over $K$ and extends to a proper relative curve $X \to C$. Let $D_i$ be the closure of $p_i$. After a possible normalization and sequence of blow ups, we can assume that $X$ is regular and that the $D_i$ do not intersect each other or the singular points of the fibers of $X \to C$. Define $\phi \colon \cX \to X$ to be the root stack of $X$ rooted along each $D_i$ with order $e_i$. The relative stacky curve $\cX$ is a model of $\cX_0$ and is tame. If there is some point $v$ of $C$ and some $i$ such that the residue characteristic of $v$ divides $e_i$, then $\cX$ is an Artin stack which is not Deligne--Mumford; if $C = \Spec \mathcal{O}_K$ for some number field $K$, then there is always some such $v$ and $i$.

  As discussed above (see Equation \ref{eqn:stacky-curve-canonical}) the canonical sheaf of $\cX_0$ is
\[
  \Omega^1_{\cX_0}=\phi_0^*\Omega^1_{X_0}\otimes\O_{\cX_0}\left(\sum (1-e_i^{-1})p_i\right).
\]  
Define
\[
  \Omega^{1,\fake}_{\cX}=\phi^*\omega_{X/C}\otimes\O_{\cX}\left(\sum (1-e_i^{-1})D_i\right)
\]
by the same ``formula''. Then $\Omega^{1,\fake}_{\cX}$ is a generically canonical sheaf.
\end{example}  

We have seen in Remark~\ref{rmk:edd-special-cases} that when $\cX$ is a scheme, $\edd$ agrees with anticanonical height, i.e., the height of the tangent bundle.  It turns out that the same identity holds when $\cX$ is a smooth, tame Deligne--Mumford stacky curve with no generic inertia, at least away from the accumulating subvarieties.

\begin{proposition}[Curves with stacky points]
\label{pr:eddcurve}
Let $\cX_0$ be a smooth tame Deligne--Mumford stacky curve over $K$ and suppose that $\cX_0$ is birational to its coarse space. Let $\cX$ be the model of $\cX_0$ given by extending the root data as in Example \ref{example:stacky-curve-model} and let $T_{\cX}$ be the dual of the generically canonical bundle from Example \ref{example:stacky-curve-model}. Let $x$ be a point of $\cX(K)$.  Then
\beq
\edd(x) = \height_{T_{\cX}}(x). 
\eeq

\end{proposition}

\begin{proof}

Let $\cC$ be a tuning stack and $\overline{x}\colon \cC \ra \cX$ the extension of $x$, as usual.  The pullback $\overline{x}^* T_{\cX}^\vee$ is a line bundle on $\cC$.  We first note that 
\beq
\height_{T_{\cX}}(x) + \height_{T_{\cX}^\vee}(x) = \sum_v (\delta_{T_{\cX};v}(x) + \delta_{T_{\cX}^\vee;v}(x))
\eeq
since
\beq
\height^{\st}_{T_{\cX}}(x) + \height^{\st}_{T_{\cX}^\vee}(x) = 0.
\eeq

For each closed point $v$ of $C$, the point $x$ either reduces to a non stacky point or reduces to a unique stacky point $p$ with stabilizer group $\mu_m$ for some integer $m \geq 2$. Let $k$ be the multiplicity of the reduction of $x$ to $p$ (i.e., the multiplicity of the intersection of the images of $x$ and $p$ in the coarse space $X$). If $m$ divides $k$ then we can take the tuning stack $\cC$ to be a scheme in a neighborhood of $v$, in which case the discrepancies are 0.  Otherwise $\cC_v$ is a root stack which can be resolved by adjoining to $K_v$ an $m$th root of a uniformizer.  Denote the resulting field extension by $L_w$.  So as in Section~\ref{subsec:local-stacky-heights}, the restriction of $\overline{x}^* T_{\cX}$ to $\OO_{K_v}$ is identified with an ideal $\Lambda$ in $\OO_{L_w}$, and we have
\beq
 \delta_{T_{\cX};v}(x) = (1/m) \log \left| \frac{\Lambda}{(\Lambda \cap K_v) \tensor_{\OO_{K_v}} \OO_{L_w}} \right|.
\eeq
Taking $\pi_w$ to be a uniformizer of $\OO_{L_w}$, we may write $\Lambda = \pi_w^{-k} \OO_{L_w}$, and so 
\beq
\delta_{T_{\cX};v}(x) = ((-k/m) - \lfloor -k/m \rfloor) \log q_v.
\label{eq:deltarootedp1}
\eeq

The restriction $\overline{x}^* T_{\cX}^{\vee}$, by the same argument, is identified with the ideal $\pi_w^{k} \OO_{L_w}$.  We conclude that
 \beq
 \delta_{T_{\cX};v}(x) +  \delta_{T_{\cX}^\vee;v}(x) = ((-k/m) - \lfloor -k/m \rfloor + k/m - \lfloor k/m \rfloor) \log q_v
 \eeq
 which is $\log q_v$ unless $m|k$, in which case it is zero.  In other words,
 \beq
 \height_{T_{\cX}}(x) + \height_{T_{\cX}^\vee}(x) = \sum_v (\delta_{T_{\cX};v}(x) + \delta_{T_{\cX}^\vee;v}(x)) = \rDisc(x)
\eeq
since $\rDisc(x)$ is precisely the sum of $\log q_v$ over the stacky points $v$ of $\cC$.  We conclude that 
\beq
\edd(x) = -\height_{T_{\cX}^\vee}(x) +  \rDisc(x) =  \height_{T_{\cX}}(x)
\eeq
as claimed.
\end{proof}

\begin{remark}
  If $\cX'$ is a second model of the stacky curve $\cX_0$ from Proposition \ref{pr:eddcurve} and if $\cX'$ is tame, one can show that away from finitely many points of $C$, $\cX'$ is a root stack and isomorphic to $\cX$; shrinking $C$ further the generically canonical sheaves agree.  By Proposition~\ref{pr:boundeddisc}, the value of $\delta_{T_{\cX'};v}(x) + \delta_{T_{\cX'}^\vee;v}(x)$ is bounded on $\cX'(K)$, and thus the $\edd$ associated to the model $\cX'$ will only differ by a constant which depends on $\cX_0$ and $K$.
\end{remark}

\subsection{Weak form of the Stacky Batyrev--Manin--Malle Conjecture}
\label{subsec:weak-stacky-batman-conj}

Having now defined $\edd$, we are ready to state a heuristic for counting rational points of bounded height on a stack. We then show that our heuristic recovers the weak form of the Batyrev--Manin when $\X$ is a scheme, and recovers the weak form of the Malle conjecture when $\X=BG$. 

Of course, we cannot expect to count points of bounded height unless the height function satisfies some kind of positivity property.  In the Batyrev--Manin setting, this is achieved by restricting to heights corresponding to ample line bundles.  One does not have as clear a geometric picture of vector bundles on stacks as one does in the setting of line bundles on schemes, so we use for the moment the following definition.  We recall that stable height is well-behaved under field extension (Proposition \ref{prop:stable-height-is-stable}), so we can define an absolute $\height^{\st;\abs}_{\cV}$ as a function on $\cX(\overline{K})$ by the usual rule:
\beq
\height^{\st;\abs}_{\cV}(x) = [L:K]^{-1}\height^{\st}_{\cV}(x)
\eeq
for points of $\cX(L)$.

\begin{definition} We say a vector bundle $\cV$ on a stack $\cX$ is \defi{semipositive} if the quantity $\height^{\st;\abs}_{\cV}(x)$ is bounded below on $\cX(\overline{K})$.

\label{def:gsp}
\end{definition}

We note that the property of being semipositive is stable under field extensions by Remark \ref{rmk:ht-under-pullback-of-V}.

\begin{definition} 
Let $f$ be a real-valued function on $\cX(\overline{K})$.  We say $f$ is \defi{generically bounded below} if there is a proper closed substack $\cZ$ of $\cX$ and a constant $B$ such that the set of $x \in \cX(\overline{K})$ such that $f(x) < [K(x):K] \cdot B$ is contained in $\cZ(\overline{K})$, where $K(x)$ is the residue field of $x$.
\label{def:gbb}
\end{definition}

Suppose $\cV$ is a semipositive vector bundle on $X$. We consider the function
\beq
D_{a,\cV}(x) = a \height_{\cV}(x) - \edd(x)
\eeq
on $\cX(\overline{K})$.  We note that if $a' > a$ then
\beq
D_{a',\cV}(x) = D_{a,\cV}(x) + (a'-a) \height_{\cV}(x) \geq D_{a,\cV}(x) + (a'-a)\height^{\st}_{\cV}(x) =  D_{a,\cV}(x) + (a'-a) [K(x):K]\height^{\st;\abs}_{\cV}(x).
\eeq

Since $\cV$ is semipositive, for fixed $a'$ and $a$ the quantity
\[
  (a'-a) [K(x):K]\height^{\st;\abs}_{\cV}(x) > (a'-a) \height^{\st;\abs}_{\cV}(x)
\]
is bounded below on $\cX(\overline{K})$.  It follows that if $D_{a,\cV}$ is generically bounded below, so is $D_{a',\cV}$.  So the set of $a$ such that $D_{a,\cV}(x)$ is generically bounded below is an interval, extending infinitely in the positive direction.

\begin{definition} With notation as above, the \defi{Fujita invariant} $a(\cV)$ of a semipositive $\cV$ is the infimum of all positive real numbers $a$ such that $D_{a,\cV}$ is generically bounded below.  If $D_{a,\cV}$ is never generically bounded below we take $a = \infty$.
\label{def:fujita}
\end{definition}

The main goal of this section is to propose a heuristic for counting points of bounded height on stacks.  If $\cX$ is a stack over $C$, $\cU$ is an open dense substack of $\cX$, and $\cV$ is a Northcott vector bundle (as in Definition~\ref{def:northcott}) on $\cX$, define a counting function
\beq
N_{\cU,\cV,K}(B) = |\set{x \in \cU(K): \height_{\cV}(x) \leq \log B}|.
\eeq

The Batyrev--Manin conjecture is customarily stated for {\em Fano} varieties, those with ample anticanonical bundle.  As mentioned above, it is not clear what the right analogue of this condition is for stacks.  For instance, we certainly do not want to exclude stacks like $BG$, on which all vector bundles have degree $0$ and are thus in some sense not ``strictly positive," but we {\em do} want to exclude stacks like abelian varieties, whose anticanonical bundle is trivial.  To this end, we make the following defintion.  Let $\cX$ is a smooth proper Deligne--Mumford stack over a number field $K$, let $m>0$ and $B$ be real numbers, and let $d \geq 1$ an integer.  We then define $S(\cX,m,d,B)$ to be the set of pairs $(L,P)$ with $L$ a degree-$d$ extension and $P \in \cX(L)$, satisfying
\beq
\edd(P) + m\Delta_{L/K} < B.
\eeq
We provisionally say $\cX$ is {\em Fanoish} if $S(\cX,m,d,B)$ is finite for all $m,d$, and $B$. 
\\


We are now ready to state the heuristic that motivates this part of the paper.

\begin{conjecture}
\label{conj:weak-batman}
Let $K$ be a number field and let $C = \Spec \O_K$.  Let $\cX$ be a stack over $C$ whose generic fiber $\cX_K$ is a smooth proper Deligne--Mumford stack over $K$. Suppose further that $\cX_K$ is Fanoish and that $\cX(K)$ is Zariski dense in $\cX_K$.  If $\cV$ is a semipositive vector bundle on $\cX$, then there exists an open dense substack $\cU$ of $\cX$ such that, for every $\epsilon > 0$, there is a nonzero constant $c_\epsilon$ such that
\beq
c_\epsilon^{-1} B^{a(\cV)} \leq N_{\cU,\cV,K}(B) \leq c_\epsilon B^{a(\cV) + \epsilon}
\label{eq:batmaneq}
\eeq
where $a(\cV)$ is the Fujita invariant defined in Definition~\ref{def:fujita}.
\end{conjecture}
 
\begin{remark} Our point of view throughout has been to let $K$ be a global field of any characteristic, however in Conjecture \ref{conj:weak-batman} we restrict to the case where $K$ has characteristic $0$. The reason for this is that we aim to emulate the Batyrev--Manin conjecture, and the form that conjecture should take for global fields of characteristic $p$ is not fully settled.  Indeed, there are counterexamples to the most naive formulations of Batyrev--Manin, even for the anticanonical height; see Starr--Tian--Zong~\cite[Lemma 5.1]{starr2018weak} and recent work of Beheshti, Lehmann, Riedl, and Tanimoto~\cite{blrt:batmanhosed}. 
\end{remark}

\begin{remark} The condition that $\cX(K)$ is Zariski dense is present to handle cases where, for instance, $\cX(K)$ is empty or supported on a closed subvariety due to a local obstruction.
\end{remark}

\begin{remark}[Accumulating loci can be $0$-dimensional] One difference between this case and the traditional Batyrev--Manin conjecture is that the accumulating locus $\cX \backslash \cU$ can be $0$-dimensional; indeed, on a stacky $\P^1$, the stacky points are accumulating subvarieties.  An example of this phenomenon can be seen in the recent paper of Pizzo, Pomerance, and Voight~\cite{pizzopomerancevoight}, which counts points on the moduli stack $X_0(3)$ with respect to (in our language) the height arising from the Hodge bundle.  They find that the preponderance of points are those supported at the single (stacky) point over $j=0$, and compute a lower-order asymptotic for points on the complement $\cU$ of this point.
\end{remark}

\begin{remark} Conjecture \ref{conj:weak-batman} corresponds to the weak version of the Batyrev--Manin conjecture.  An analogue of the strong version would be an assertion that $N_{\cU,\cV,K}(B)$ is asymptotic to a constant multiple of $B^{a(\cV)} (\log B)^{b(\cV,K)}$ for some explicit constant $b(\cV,K)$.  Getting the power of $\log B$ correct  (not even to speak of the constant!) is very subtle even in the Batyrev--Manin setting where $\cX$ is a scheme; we will not attempt to pin it down here, but it seems a rich problem for further investigation.
\end{remark}

\begin{remark}  One could, in the same way, propose a heuristic for counting points on $\cX$ of bounded {\em stable} height.  Just above, one could define $D^{\st}_{a,\cV}(x)$ to be  $\height^{\st}_{\cV}(x) - \edd(x)$ and define the stable Fujita invariant to be the infimum of those $a$ such that $D^{\st}_{a,\cV}$ is generically bounded below.  This gives nothing new in the case where $\cX$ is a scheme (where stable height and height are the same) or where $\cX = BG$ (in which case stable height is $0$) but is of interest in other cases: see Section~\ref{s:sympowers} for an example.  In the same vein, and in some sense analogously to the central case of Batyrev--Manin where we count by anticanonical height, one could count the number of points $x$ of $\cX(K)$ with $\edd(x) < \log B$, even though $\edd$ is not always a height in the sense of this paper.  One could reasonably expect this count to be bounded between constant multiples of $B$ and $B^{1+\epsilon}$.  For example, when $\cX = BS_3$ and $K=\Q$, this would amount to counting cubic fields $L/\Q$ such that the product of the primes ramified in $L$ is at most $\log B$.  This counting problem will be addressed in forthcoming work of Shankar and Thorne, where it is shown that the count is on order $B \log B$. 
\label{rem:stablebatman}
\end{remark}

\subsection{The case where $\cX$ is a scheme:  the Batyrev--Manin conjecture}

Suppose $\cX$ is a scheme $X$.  Then, since $\height_{\cV} = \height_{\wedge^r \cV}$ for any rank $r$ vector bundle $\cV$ on $\cX$, we may assume $\cV$ is a line bundle $\cL$.  We have seen in Remark \ref{rmk:edd-special-cases} that $\edd(x)=\height_{-K_X}(x)$ for any $x \in X(\overline{K})$.  So if $X$ is Fano, it is Fanoish because the anticanonical height is an ample height and thus has the Northcott property.  It is not immediately obvious that a Fanoish scheme is Fano, but it is also not unreasonable to hope so. To begin, $-K_X$ is nef:~if there were a curve $C$ on $X$ with $-K_X |_C$ of negative degree, then for some $d$, there is a degree-$d$ map $C\to\P^1$ which provides many degree-$d$ algebraic points with more and more negative $-K_X$-height, not counteracted by $m \Delta_{L/K}$ if we make $m$ small enough. We also note that a variety with trivial canonical sheaf may be expected {\em not} to be Fanoish; a K3 surface, for instance, is expected (though not in general known) to have a Zariski-dense set of points over some extension $L$ of $K$, which implies that $X$ is non-Fanoish since all these point have $-K_X$-height $0$ and $\Delta_{L/K}$ fixed.

The question of which schemes ``should" satisfy the Batyrev--Manin conjecture is not wholly understood, but is probably not limited to Fano schemes alone; if it turns out that ``Fanoish" delineates a class of schemes including some to which Batyrev--Manin does not apply, we will narrow the notion.

The condition that $\cL$ is semipositive simply says that $\cL$ is nef; a nef height is bounded below, and if $\cL$ is not nef, there is a curve on which $\cL$ has negative degree, whose $\overline{K}$-points thus have heights which are not bounded below.

Now

\beq
D_{a,\cL}(x) = a \height_\cL(x) - \edd(x) = \height_{a\cL + K_X}(x)
\eeq
and $a(\cL)$ is the minimal $a$ such that $\height_{a\cL + K_X}(x)$ is generically bounded below.  

What does this say about the line bundle $a\cL + K_X$?  First of all, if $M$ is a big line bundle on $X$, then  the map $\phi_k\colon X \ra \P^{N_k}$ induced by the global sections of $\cL^k$ is a birational embedding for some sufficiently large $k$.   It is then immediate that the absolute height $\height_M(x)$ is bounded below on $X(\overline{K})$ away from the locus $Z$ contracted by $\phi_k$, and that there are only finitely many points of $X(K) \backslash Z(K)$ with height below any given bound.  So $h_M$ is generically bounded below.  On the other hand, the pseudoeffective cone is dual to the cone of {\em moving} curves by by a theorem of Boucksom, Demaily, Pa\u{u}n, and Peternell~\cite[Th 0.2]{bdpp} (see \cite[Th 2.22]{lehmannzariski} for the case of characteristic $p$).  So if $M$ is not pseudoeffective, there is a moving curve $Y$ on $X$ on which $M$ has negative degree; if $Z$ is any closed locus, we can move $Y$ to not be contained in $Z$, and then $Y(\overline{K})$ has points away from $Z$ of arbitrarily negative height; in particular, $h_M$ is not generically bounded below.

Since the pseudoeffective cone is the closure of the big cone, we conclude that the infimum of $a$ such that $\height_{a\cL + K_X}(x)$ is generically bounded below is the same as the infimal $a$ such that $a\cL + K_X$ is pseudoeffective, which is the same as the infimal $a$ such that $a\cL + K_X$ is big.  And this $a(\cL)$ is just the usual Fujita invariant appearing in the Batyrev--Manin conjecture for Fano varieties.  So Conjecture~\ref{conj:weak-batman} recovers the (weak form of the) Batyrev--Manin conjecture.

 \subsection{The case where $\cX$ is $BG$: Malle's conjecture}
\label{subsec:BG-Malle}
 
Now suppose $\cX = BG$ over a number field $K$, and $\cV$ is a vector bundle, i.e., a representation of $G$.   In particular, let us assume $\cV$ is a faithful permutation representation corresponding to an embedding $G < S_n$.  Each point $x$ of $BG$ corresponds to a $G$-extension of $K$ (possibly an \'etale algebra), and, via the embedding of $G$ into $S_n$, a degree-$n$ extension $L/K$ whose Galois closure is $G$. We have already computed that
\beq
\height_{\cV}(x) = (1/2) \log |\Delta_{L/K}| = \sum_{v \in R} e_v \log q_v
\eeq
where $R$ is the set of nonarchimedean places of $K$ ramified in $L/K$, and $e_v$ is the local degree of the discriminant.  If $v$ is a place where $L/K$ is tamely ramified, so that tame inertia acts on $\set{1,\ldots, n}$ through a cyclic subgroup $\langle \pi \rangle < S_n$, the ramification $e_v$ is just the {\em index} $\ind(\pi)$, the difference between $n$ and the number of orbits of $\pi$.

First of all, note that $\cV$ is semipositive, since $\height^{\st}_{\cV}$ is identically $0$. 


 It follows from Remark \ref{rmk:edd-special-cases} that for any extension $E/K$ and any point $x \in BG(E)$ corresponding to a degree-$n$ extension $F/E$, we have
\beq
\edd(x) = \sum_v \log q_v
\eeq
where the sum is over nonarchimedean places $v$ of $E$ which are ramified in $F/E$.  Note in particular that, because this is positive, $BG$ is Fanoish; the set of $(L,x \in BG(L))$ with $\edd(x) + m\Delta_{L/K} < B$ involves only the finite set of extensions $L/K$ with discriminant at most $B/m$, and for each $L$, the set of $x \in BG(L)$ with $\edd(x) < B$ is finite since it consists of $G$-extensions of $L$ with bounded discriminant. 

Thus, 
\beq
D_{a,\cV}(x) = a \height_{\cV}(x) - \edd(x) = \sum_v ((1/2)a e_v - 1) \log q_v.
\eeq

Suppose $a \geq 2 \max_{\pi \in G} \ind(\pi)^{-1}$.  Then $(1/2) a e_v -1 \geq 0$ for all tame primes $v$.  The contribution of non-tame primes is bounded below by a constant depending only on $[E:K]$.  Thus the Fujita invariant of $\cV$ is at most $2\max_{\pi \in G} \ind(\pi)^{-1}$.

Suppose, on the other hand, that $a$ is strictly smaller than $2 \ind(\pi)^{-1}$ for some $\pi \in G$.  If $E/K$ is an extension of $K$ and $L/E$ a $G$-extension such that every ramified prime is tame and has tame inertia acting via $\pi$, then the point $x$ has
\beq
\edd(x) = \sum_v ((1/2) a e_v - 1) \log q_v
\eeq
which is bounded above by a {\em negative} constant multiple of $\sum_v \log q_v$.  Heuristically, it seems safe to suppose one can choose such $(E,L)$ with $\sum \log q_v$ as large as one likes, which would mean that $D_{a,\cV}$ was not generically bounded below.  But this is perhaps not completely obvious: for instance, when $G = S_n$, one is saying that there are many field extensions with squarefree discriminant.  One certainly expects this to be the case, but the fact, for example, that there are arbitrarily large squarefree integers which are discriminants of degree-$n$ extensions of $\Q$ is a recent result of Kedlaya~\cite{Kedlaya2011}.  In fact, all we need is that for {\em some} extension $K'$ of $K$ there are extensions $L/K'$ with larger and larger discriminants whose ramification is entirely or almost entirely drawn from the minimal-index conjugacy class in $G$.  One can presumably construct such extensions using the method of regular extensions popular in work on the inverse Galois problem; using the Riemann existence theorem you write down a cover of curves $X \ra \P^1_{\overline{K}}$ with Galois group $G$ and all ramification drawn from the minimal-index conjugacy class, then descend the picture to $X_0 \ra  \P^1_{K'}$ for some finite extension $K'/K$, then show that specialization to points of $\P^1(K')$ yields many extensions of $K'$ with the desired properties.  Since we are just formulating conjectures here, we will not push this argument through in detail.

An argument of the sort sketched in the above paragraph is necessary due to the fact that we defined the Fujita invariant in terms of heights of points over extension fields of $K$; presumably, a more conceptual geometric definition of the Fujita invariant of a vector bundle with zero stable height would automatically assign $\cV$ the value $2 \max_{\pi \in G} \ind(\pi)^{-1}$.

At any rate, if we grant the heuristic argument on the Fujita invariant above, we find that Conjecture~\ref{conj:weak-batman} predicts that the number of degree-$n$ extensions $L/K$ with Galois group $G$ and discriminant at most $B$ -- in other words, the number of points $x$ on $BG(K)$ such that
\beq
\height_{\cV}(x) = (1/2) \log |\Delta_{L/K}| < (1/2) \log B
\eeq
is bounded between $c_\epsilon^{-1} B^a$ and $c_\epsilon B^{a+\epsilon}$, where $a = \max_{\pi \in G} \ind(\pi)^{-1}$.  This is exactly the weak Malle conjecture.

\begin{remark} When $\cV$ is a representation of $G$ which is not a permutation representation, one still has some conjugacy-invariant function $f$ from $G$ to $\R_{>0}$ and an expression
\beq
\height_{\cV}(x) = \sum_{v \in R} c_v \log q_v
\eeq
where, for every tamely ramified prime $v$, the coefficient $c_v$ is the value of $f$ at an element of $G$ generating the tame inertia group at $v$.  In this case, Conjecture~\ref{conj:weak-batman} asserts that the number of points $x \in BG(X)$ with $\height_{\cV}(x) < \log B$ should be on order $B^a$, where $a$ is the reciprocal of the minimal value taken by $f(v)$ on non-identity elements of $G$.  Heuristics of this kind are well-known folk generalizations of Malle (see e.g., \cite[\S 4.2]{ellenberg:countingff}) and have begun to be proved in some cases.  For instance, the striking work of Alt\"{u}g, Shankar, Varma, and Wilson~\cite{d4conductor} can be thought of as proving Conjecture~\ref{conj:weak-batman} in the case where $\cX = BD_4$ and $\cV$ corresponds to the $2$-dimensional action of $D_4$ by rigid motions of the square.  (What they prove is much more refined than what Conjecture~\ref{conj:weak-batman}; they not only compute the power of $B$, but the power of $\log B$, and even the constant!)

The recent work of Alberts~\cite{albertsh1count} on counting classes in $H^1(\Gal(\Q),A)$, where $A$ is an abelian group with Galois action, can perhaps also be thought of in this way. Here, $A$ corresponds to an \'etale but possibly nonconstant group scheme, so the stack $BA$ is {\em geometrically} the classifying stack of the finite abelian group underlying $A$.  In this case, the points of $BA(\Q)$ are just the classes in $H^1(\Gal(\Q),A)$.   The ``$\pi$-discriminant" of \cite[Lemma 1.4]{albertsh1count} is the height attached to the vector bundle on $BA$ descended from the regular representation of the finite group underlying $A$.
\end{remark}

\subsection{Symmetric powers of $\P^n$}

Let $\cX$ be the stack $\Sym^m \P^n = [(\P^n)^m / S_m]$ and let $K$ be a global field of characteristic $0$ or greater than $m$.  For $x$ a point of $\cX(K)$ we have 
\beq
\edd(x) = -\height_{T_{\cX}^\vee}(x) +  \rDisc(x). 
\eeq
Note that we can associate to $x$ a degree-$m$ extension $L_1$ of $K$ and a point $y$ of $\P^n(L_1)$.  

The cotangent bundle $T_{\cX}^\vee$, considered as an $S_m$-equivariant bundle on $(\P^n)^m$, is the direct sum of the $m$ pullbacks of the cotangent bundle from the $m$ projections $\P^n$, and the height associated to the cotangent bundle on $\P^n$ is just the usual height associated to its determinant $\OO(-n-1)$.  So we are in the situation of Section~\ref{s:sympowers}, and we have
\beq
\height_{T_{\cX}^\vee}(x) = \height_{\OO(-n-1)}(y) + (n/2)\log \Delta_{L_1/K}.
\eeq
Thus
\beq
\edd(x) = \height_{\OO(n+1)}(y) + \sum_{v \in R} (1-(n/2)e_v) \log q_v 
\eeq
where, as in \S\ref{subsec:BG-Malle}, $R$ is the set of tamely ramified places and $e_v$ is the power of $v$ in the discriminant of $L_i/K$; the contribution of the wildly ramified places, as in Section~\ref{subsec:BG-Malle}, is bounded by a constant (and if $x$ varies over $\cX(L)$ for some extension $E/K$, the wild contribution is bounded by a constant depending only on $[E:K]$).

We also have
\beq
\height_{T_{\cX}}(x)  = \height_{\OO(n+1)}(y) + (n/2)\log \Delta_{L_1/K} = \edd(x) + \sum_{v \in R} (ne_v-1) \log q_v. 
\eeq

In particular, $\height_{T_{\cX}}(x) - \edd(x)$ is always nonnegative, and $\height_{T_{\cX}}(x) = \edd(x)$ whenever $x$ is a point of $\cX$ in the image of the projection from $(\P^n)^m(K)$ to $\Sym^m \P^n(K)$.  This shows that the Fujita invariant $a(T_{\cX})$ is $1$.  Conjecture~\ref{conj:weak-batman} thus suggests that, away from some proper closed substack, the number of rational points on $\Sym^m \P^n(X)$ with tangential height at most $B$ is between $B^{1-\epsilon}$ and $B^{1+\epsilon}$.

There is a large existing literature about counting points on projective spaces of fixed algebraic degree and bounded height~\cite{schmidt93,gao,masservaaler1,masservaaler2,widmer,lerudulier,grizzardgunther,guignard}.  Most typically, the question being asked is:  how many points are there in $\P^n(\overline{K})$ which have absolute Weil height at most $B$ and which are defined over a field $L_1/\K$ of degree $m$?  As we have seen in \S~\ref{s:sympowers}, we can interpret this question as follows.  Let $\cV$ be the vector bundle on $\Sym^m \P^n$ obtained as in \S~\ref{s:sympowers} taking $V_0$ as $\O_{\P^n}(1)$.  If $y$ is a point of $\P^n(L_1)$ and $x$ the corresponding point of $\Sym^m \P^n$, we have
\beq
\height_{\OO(1)}^{\abs}(y) = m^{-1} \height^{\st}_{\cV}(x).
\eeq
So we are in the situation of Remark~\ref{rem:stablebatman}.  In order to compute the stable Fujita invariant of $\cV$ we need to study the function
\beq
D_{a,\cV}^{\st}(x) = a \height^{\st}_{\cV}(x) - \edd(x) = (a-n-1)\height_{\OO(1)}(y) - \sum_{v \in R} (1-(n/2)e_v) \log q_v.
\eeq

When $n \geq 2$, we note that the local term $\sum_{v \in R} (1-(n/2)e_v) \log q_v$ is always non-positive, and is $0$ when $L_1$ is $K^m$; in particular, the set of $x$ in $\Sym^m \P^n(K)$ with $\edd(x) = (a-n-1) \height^{\st}_{\cV}(x)$ is Zariski dense for every $K$.  Thus, $D^{\st}_{a,\cV}$ will be generically bounded below for any $a \geq n+1$, but is not generically bounded below for any smaller $a$.  So the stable Fujita invariant is $n+1$.  For each $y$ in $\P^n(\overline{K})$ with $[K(y):K] = m$, we write $x_y$ for the point of $\Sym^m \P^n$.  Then Conjecture~\ref{conj:weak-batman} suggests that for every $n \geq 2$ we should expect that, for some open dense $U \in \Sym^m \P^n$,
\beq
c_\epsilon^{-1}  B^{m(n+1)} <  \#\set{y \in \P^n(\overline{K}): [K(y):K] = m, , x_y \in U(K), \height^{\abs}(y) < B}  < c_\epsilon B^{m(n+1) +\epsilon}.
\label{eq:sympn}
\eeq

When $n=1$, the situation is more complicated.  We now have
\beq
D_{a,\cV}^{\st}(x) \geq (a-2)\height_{\OO(1)}(y) - \sum_{v \in R} (1/2) e_v \log q_v = (a-2)\height_{\OO(1)}(y) - (1/2) \log \Delta_{L_1/K}
\eeq
with equality when $L_1/K$ has squarefree discriminant.  In order to understand how large $a$ needs to be for $D_{a,\cV}^{\st}(x)$ to be generically bounded below, we need to know how large $\log \Delta_{L_1/K}$ can be relative to $\height_{\OO(1)}(y)$.  A point $y$ of $\P^1(L_1)$ has a minimal binary $m$-ic form $F = a_0 X^m + \cdots + a_m Y^m$, where the height of the point $(a_0: \ldots :a_m)$ in $\P^m(L)$ is on order $m \height(y)$, since each coefficient is a monomial of degree $m$ in the coordinates of $y$.  The discriminant of $L_1/K$ is at most the discriminant of $F$, with equality if $\disc F$ is squarefree.  The discriminant of $F$ is a product of $m(m-1)$ terms of the form $\alpha_i \beta_j - \alpha_j \beta_i$, where $(\alpha_i: \beta_i)$ and $(\alpha_j: \beta_j)$ are conjugates of $y$ in $\P^1(\overline{K})$.  So the log of  $\disc F$, considered as an element of $\OO_L$, is on order $2m(m-1) \height_L(y)$, and the log of $\disc F$ considered as an element of $\OO_K$ is thus $2(m-1) \height(y)$.  We conclude that
\beq
D_{a,\cV}^{\st}(x) \geq (a-2)\height_{\OO(1)}(y) - (m-1) \height_{\OO(1)}(y) = (a-m-1) \height_{\OO(1)}(y).
\eeq
So $D_{a,\cV}^{\st}$ is generically bounded below when $a \geq m+1$, and as long as there is a Zariski-dense set of choices of $y$ with $\disc F$ squarefree (perhaps this is obvious, but at any rate it follows from standard conjectures) $D_{a,\cV}$ is {\em not} generically bounded below for any smaller $a$.  So the stable Fujita invariant in this case is $m+1$ and Conjecture~\ref{conj:weak-batman} asserts that, for some open dense $U$,
\begin{equation}
c_\epsilon^{-1}  B^{m(m+1)} <  \#\set{y \in \P^1(\overline{K}): [K(y):K] = m, , x_y \in U(K), \height^{\abs}(y) < B}  < c_\epsilon B^{m(m+1) +\epsilon}.
\label{eq:symp1}
\end{equation}

In fact, \eqref{eq:symp1} follows from a theorem of Masser and Vaaler~\cite{masservaaler2}, who prove a much more refined asymptotic, with $U$ the whole of $\Sym^m \P^1$:
\beq
\#\set{y \in \P^1(\overline{K}): [K(y):K] = m, \height^{\abs}(y) < B} \sim A_{m,K} B^{m(m+1)}
\eeq
with an explicit constant $A_{m,K}$.  Of course to compute the constant in the case where $K$ is a number field, one has to be careful about the metrization on $\OO(1)$ in a way we are not attempting here.  Le Rudulier~\cite{lerudulier} generalized the Masser--Vaaler result to the case of an arbitrary metrized line bundle on $\P^1$.

When $n \geq 2$, the asymptotics for points of bounded height on projective $n$-space with algebraic degree $m$ is still the subject of active research.  If $n$ is large enough relative to $m$, the heuristic \eqref{eq:sympn} is known to be correct; indeed, one has
\beq
\#\set{y \in \P^n(\overline{K}): [K(y):K] = m, \height^{\abs}(y) < B} \sim A_{m,n,K} B^{m(n+1)}
\eeq
when $K$ is a number field and $n > (5/2)m + O(1)$, by a result of Widmer~\cite{widmer} and when $n > m+1$ with $m$ prime, by a result of Guignard~\cite{guignard}.   For the function field case, the result is proved by  Thunder and Widmer~\cite{thunderwidmer} when $n > 2m+4$ (and generalized from $\P^n$ to smooth projective toric varieties by Bourqui in \cite{bourqui:toric}).  Schmidt in \cite{schmidt95} showed that \eqref{eq:sympn} holds in case $K=\Q$, $m=2$ and $n=2$; indeed, in that case, the growth rate is $B^6 \log B$, showing that the $\epsilon$ in the exponent is sometimes necessary.  M\^{a}nz\u{a}\textcommabelow{t}eanu~\cite{manzateanu} extended Schmidt's result to function fields $K$ of odd characteristic.

On the other hand, Schmidt in \cite{schmidt93} gives a lower bound
\beq
\#\set{y \in \P^n(\overline{K}): [K(y):K] = m, \height^{\abs}(y) < B}  > A_{m,n,K} B^{m(m +1)}
\eeq
valid for all $n$ and all sufficiently large $B$.  When $m > n$, this is a larger exponent than that predicted in \eqref{eq:sympn}.  But this does not contradict Conjecture~\ref{conj:weak-batman}.  The source of Schmidt's lower bound is the simple observation that any choice of line in $\P^n$ yields an injection of $\Sym^m \P^1(K)$ into $\Sym^m \P^n(K)$, and the former already contains $B^{m(m+1)}$ points of height at most $B$.  But any such point lies on the proper closed substack $Z \subset \Sym^m \P^n(K)$ lying under the locus in $(\P^n)^m$ parametrizing ordered $m$-tuples of {\em collinear} points.  Thus, it remains possible that when some accumulating locus is removed, the asymptotic growth rate of the number of points is smaller.  And indeed, Guignard~\cite[Theorem 1.2.3]{guignard} shows exactly this in the case where $K$ is a number field, $m=3$, and $n=2$.  In this setting, Schmidt's lower bound shows that the number of cubic points on $\P^2$ with absolute height at most $B$ is at least $cB^{12}$.  Guignard shows that if you exclude those cubic points which lie on a $K$-rational line, the number of rational points that remain is bounded above by $c_\epsilon B^{9+\epsilon}$, precisely the exponent predicted by Conjecture~\ref{conj:weak-batman}.

We thus see that the present viewpoint is useful for understanding phenomena of accumulation in a uniform way.  The algebraic points witnessing Schmidt's lower bound are clearly ``non-generic" in some sense; but, considered as points of $\P^n(\overline{K})$, they are Zariski dense.  Considering these points instead as points on $\Sym^m \P^n$ shows that the accumulation is a phenomenon that can be repaired by stripping out a proper closed subvariety, exactly as in the Batyrev--Manin setting.  Of course, one does not need to invoke stacks to adopt this point of view -- for instance, see \S 33.2 of  Le Rudulier's thesis~\cite{lerudulier}, where a degree-$m$ algebraic point of $\P^2$ is thought of as a point on the coarse moduli scheme of $\Sym^m \P^2$ rather than the stack itself; since the two are birational, the observation that the collinear $m$-tuples lie on a subvariety on which rational points accumulate takes the same form for Le Rudulier as it does for us.

\subsection{Footballs and multifootballs}

Proposition~\ref{pr:eddcurve} shows that $\edd$ agrees with tangential height $\height_{T_{\cX}}$ when $\cX$ is a smooth proper $1$-dimensional stack over a number field $K$ which is birational to a curve.  In particular, Proposition~\ref{pr:eddcurve} applies when $\cX$ is a stacky curve birational to $\P^1$ which has $r$ stacky points isomorphic to $B(\mu_{m_1}), \ldots, B(\mu_{m_r})$.  For short we will call such a curve an $(m_1, \ldots, m_r)$-rooted $\P^1$. The football $\F(a,b)$ as in \S~\ref{ss:footballs} is then an $(a,b)$-rooted curve.

Let $\cX$ be an $(m_1, \ldots, m_r)$-rooted $\P^1$.  Now Conjecture~\ref{conj:weak-batman} predicts that, for some open dense $\cU$ in $\cX$, we have
\begin{equation}
c_\epsilon^{-1} B \leq N_{\cU,T_{\cX},K}(B) \leq c_\epsilon B^{1 + \epsilon}.
\label{eq:footballbatman}
\end{equation}

First of all, $U$ is obtained by removing a finite set of points from $\cX$, so we can interpret the above asymptotic as a heuristic for the number of points of $\cX$ of bounded height which are not supported on the stacky locus.

The coarse map $\cX \ra \P^1$ is a birational isomorphism, and so without serious ambiguity we can denote a point $x$ on $\cX(K)$ not contained in stacky locus by its image $(a:b)$ in $\P^1(K)$.  We will now compute tangential height explicitly.  The tangent sheaf $T_{\cX}$ is $2P + \sum_i (1/m_i - 1) P_i$, where $P_i$ is the $i$'th stacky point and $P$ is some other point on $\cX$; the degree of $T_{\cX}$ is thus $d = 2-r + \sum_i (1/m_i)$.  If $N$ is an multiple of every $m_i$, then $NT_{\cX}$ is linearly equivalent to $Nd$ copies of $P$; in other words, it is pulled back from $\OO(Nd)$ on the coarse space $\P^1$.  We thus have
\beq
\height^{\st}_{T_{\cX}}(x) = (1/N) \height^{\st}_{NT_{\cX}}(x) = (1/N) \height_{\OO(Nd)}(a:b) = d \height_{\OO(1)}(a:b).
\eeq

We note, in particular, that $T_{\cX}$ is not semipositive unless $d \geq 0$, so we assume this from now on.

For expositional simplicity, we now restrict to the case $K=\Q$.  So the stable height of $x$ is $d \log \max(|a|,|b|)$ where $a$ and $b$ are now taken to be coprime integers.  It remains to compute the local discrepancies.  The local discrepancy $\delta_v(a:b)$ can be computed as follows.  The tangent bundle $T_{\cX}$ has local degree $1/m_i \in \Q/\Z$ at $P_i$, so the degree of $\overline{x}^* T_{\cX}^\vee$ at the point of the tuning stack $\cC$ over a place $v$ is $-k/m_i$ where $k = \ord_v L_i(a:b)$.  Thus the local degree of the pushforward $\pi_* \overline{x}^* T_{\cX}$ on $C$ is $\lfloor -k/m_i \rfloor = - \lceil k/m_i \rceil$, and so the local discrepancy is given by
\beq
\delta_v = (\lceil k/m_i \rceil - k/m_i) \log q_v.
\eeq
Throw out the bounded contribution of any prime $v$ where two distinct $P_i$ intersect, and denote by $L_i$ the linear form whose zero is at $P_i$.  Then for each prime $p$, there is at most one $L_i(a,b)$ vanishing at $p$, and the local discrepancy is $(1/m_i)\log p^c$ where $c$ is the least integer such that the $p$-adic valuation of $p^c L_i(a,b)$ is a multiple of $m_i$.  

\begin{definition} For integers $m,N$, define $\Phi_m(N)$ to be the unique $m$-th power free integer such that $N \Phi_m(N)$ is an $m$th power.  Alternatively,
\beq
\Phi_m(N) = \prod_p p^{m \lceil \ord_p N / m \rceil - \ord_p N}.
\eeq
\end{definition}

When $m=2$, we have that $\Phi_2(N)$ is the {\em squarefree part} of $N$, denoted $\sqf(N)$.

Putting this all together, we find
\beq
\height_{T_{\cX}}(a:b) = \sum_i (1/m_i) \log \Phi_{m_i}(L_i(a,b)) + (2-r + \sum_i 1/m_i) \log \max(|a|,|b|).
\eeq

When $r$ is small, it is straightforward to see that \eqref{eq:footballbatman} is satisfied.  For example, consider a $\P^1$ rooted only at $0$ with a copy of $B\mu_3$ (that is, $r=1$ and $m_1 = 3$).   Then (taking $\cU$ to be the complement of the stacky locus) $N_{\cU,T_{\cX},K}(B)$ is the number of pairs of coprime $a,b$ such that
\beq
\Phi_3(a)^{1/3} \max(|a|,|b|)^{4/3} < B.
\eeq
We can write $a$ uniquely as $c^3 d_1 d_2^2$ where $d_1,d_2$ are coprime and squarefree, and clearly bounded above by a power of $B$. Then $\Phi_3(a) = d_1^2 d_2$ and we find that up to constants we are counting the positive $c,d_1,d_2,b$ such that
\beq
d_1^{2/3} d_2^{1/3} \max(c^4 d_1^{4/3} d_2^{8/3},b^{4/3}) = \max(c^4 d_1^2 d_2^3, b^{4/3} d_1^{2/3} d_2^{1/3}) < B.
\eeq
For a given choice of coprime $d_1, d_2$ we see that the number of choices for $c$ is $B^{1/4} d_1^{-1/2} d_2^{-3/4}$, while the number of choices for $b$ is $B^{3/4} d_1^{-1/2} d_2^{-1/4}$, so the number of choices for the pair $(c,b)$ is just $B d_1^{-1} d_2^{-1}$; summing this over all coprime pairs $d_1,d_2$ up to some power of $B$ gives an asymptotic for $N_{\cU,T_{\cX},K}(B)$ on order $B \log^2 B$, which agrees with the heuristic prediction~\eqref{eq:footballbatman}.

John Yin has shown (personal communication) that \eqref{eq:footballbatman} holds for a $(2,2)$-rooted $\P^1$; in fact, he addresses the more general case where the degree-2 stacky locus is irreducible over $\Q$ rather than being supported at two rational points, as in the cases discussed here.

Things get more difficult as $r$ grows.  Consider the case of a $(2,2,2)$-rooted $\P^1$ with the half-points located at $0,-1,$ and $\infty$.  Then
\beq
\height_{T_{\cX}}(a:b)  = (1/2) \log (\sqf(a) \sqf(b) \sqf(a+b) \max(|a|,|b|))
\eeq
so $N_{\cU,T_{\cX},K}(B)$ is the number of pairs of coprime $a,b$ such that
\beq
\sqf(a) \sqf(b) \sqf(a+b) \max(|a|,|b|) < B^2.
\eeq
This set contains all pairs of coprime integers in $[0,\sqrt{B}]$, so it has size at least $cB$, as predicted.  

In fact, in recent work, Pierre Le Boudec (in personal communication) and Nasserden--Xiao~\cite{nasserdenS:The-density-of-rational-points-P1-with-three-stacky-points} have independently shown that $N_{\cU,T_{\cX},K}(B)$ is bounded above and below by constant multiples of $B \log B^3$.  This seems a very interesting case to explore further; can one obtain an asmyptotic $N_{\cU,T_{\cX},K}(B) \sim c B \log^3 B$, and if so, what is the constant?  

We also note that some footballs are weighted projective lines; in recently announced work, Darda~\cite{darda2021rational-arxiv} proves counting results for weighted projective spaces.

\subsection{When $\edd$ is negative: a stacky Lang--Vojta conjecture}
Conjecture~\ref{conj:weak-batman} is meant to apply to those ``Fanoish" stacks $\cX$ where $\edd$ is positive in some appropriate sense.  In this section, we consider the opposite scenario:~where $\edd(x)$ is negative. When $X$ is a scheme, this is the situation where the canonical bundle $K_X$ is ample, so that $X$ is of general type; in this case, and assuming $K$ is a number field, Lang's conjecture suggests that $X(K)$ should be supported on a proper closed subvariety of $X$.  (When $K$ is a global field of characteristic $p$, the situation is more subtle -- the famous examples of Shioda show, for instance, that a variety can be of general type and also unirational!  We thus restrict to the number field case for the remainder of the discussion.) 

More precisely, conjectures of Vojta say that, for any $X$, any ample line bundle $L$, and any real $\delta > 0$, the set of rational points on $X(K)$ such that
\beq
\height_{-K_X}(x) + \delta\height_L(x) < 0
\eeq
should be supported on a proper closed subvariety.

This suggests that one might tentatively propose a ``Vojta conjecture for stacks" as follows:  let $\cX$ be a stack over a number field $K$, let $L$ be a line bundle on $\cX$ pulled back from an ample line bundle on the coarse space of $\cX$, and let $\delta > 0$ a real number. 

\begin{conjecture}
The set of rational points of $\cX(K)$ such that
\beq
\edd(x)  +  \delta \height_L(x) < 0
\eeq
is supported on a proper closed substack of $\cX$.
\label{stackyvojta}
\end{conjecture}

For example, if $\cX$ is a $(4,4,4)$-rooted $\P^1$ with the $(1/4)$th-points at $0,\-1,\infty$, then we have
\beq
\edd(a:b) = \log \Phi_4(a)^{1/4} \Phi_4(b)^{1/4} \Phi_4(a+b)^{1/4} \max(|a|,|b|)^{-1/4}
\eeq
and the claim is then that the inequality
\beq
\Phi_4(a) \Phi_4(b) \Phi_4(a+b) < \max(|a|,|b|)^{1 - \delta}
\eeq
holds for only finitely many pairs of coprime integers $a,b$.

Another interesting case is that of a $(2,2,2,2,2)$-rooted $\P^1$ with the half-points at $0,1,2,3,4$.  In this case, Conjecture~\ref{stackyvojta} says there are only finitely many five-term arithmetic progressions $a_1, \ldots, a_5$ such that
\beq
\sqf(a_1 a_2 a_3 a_4 a_5) < \max(a_1, a_5)^{1 - \delta}.
\eeq

As Nasserden and Xiao explain in \cite[Theorem 1.4]{nasserdenS:The-density-of-rational-points-P1-with-three-stacky-points}, the assertion that Conjecture~\ref{stackyvojta} holds for all stacky curves is equivalent to the abc conjecture, with a key ingredient being a result of Granville~\cite{granville:abcsquarefree}; indeed, Granville's result shows immediately that the two examples above satisfy Conjecture~\ref{stackyvojta} conditional on abc.  What is the relation between Vojta's ``more general abc conjecture" from \cite{vojta:generalabc} applied to a divisor $D$ on a scheme $X$, and Conjecture~\ref{stackyvojta} for a stack obtained by rooting a scheme $X$ at $D$?\footnote{We are grateful to Aaron Levin for useful discussions concerning this connection.}  One may hope that individual cases of Conjecture~\ref{stackyvojta}, like those described above, might not be as far out of reach as abc and its generalizations.


We note that a conjecture akin to Conjecture~\ref{stackyvojta} also appears in the work of Abramovich and V\'arilly~\cite[Proposition 3.2]{abramovichVA:level-structures}; they show their conjecture follows from the Vojta conjecture for schemes, and derive from this a finiteness theorem, conditional on Vojta, for principally polarized abelian varieties with full $m$-level structure for large enough $m$.  Their conjecture is expressed in terms of a height on $\cX$ which, in the language of this paper, is $\height^{\st}_{-K_X}$.  And their conjecture, like Conjecture~\ref{stackyvojta}, can be expressed as an assertion that the set of points $x \in \cX(K)$ with
\beq
\height^{\st}_{-K_X}(x)  + \delta \height_L(x) + \sum_v \alpha_v(x) < 0
\eeq
is not Zariski dense, for some local nonnegative contributions $\alpha_v$ supported at the points where $x$ fails to extend to an integral point of $\cX$.  (In fact, their conjecture says more, making an assertion about all algebraic points of a fixed degree $r$.)  The conjecture of Abramovich and V\'arilly is compatible with Conjecture~\ref{stackyvojta} but is not identical to it.  One interesting case where they differ is that of $\cX = A / \pm 1$, with $A$ an abelian variety of dimension $g$ over a number field $K$.  Let $x$ be a point of $(A / \pm 1)(K)$, which is to say a quadratic extension $F/K$ and a point of $A(F)$ with trace zero in $A(K)$.  The stable height can be computed on the pullback to the \'etale cover $A$, where the canonical divisor on $\cX$ is zero, so the Abramovich--V\'arilly conjecture bounds the set of $x \in \cX(K)$ such that
\beq
\delta \height_L(x) + \sum_v \alpha_v(x) < 0.
\eeq
But the left-hand side is positive for all but finitely many $x$ by the ampleness of $L$, so this is easy.  On the other hand, Conjecture~\ref{stackyvojta} says more.  We have
\beq
\edd(x) = -\height_{T_{\cX}^\vee} + \rDisc(x).
\eeq
Near a stacky point $v$, the tuning stack looks like $[(\Spec\cO_{F,v})/\pm1]$ and $\Lambda$, as in \S\ref{subsec:local-stacky-heights}, is given by $\cO_{F,v}^{\oplus g}$ where the $\pm1$-action sends the $i$-th basis vector $e_i$ to $-e_i$; hence, if we let $\overline{\alpha}$ denote the quadratic conjugate of $\alpha\in F_v$, we see $\alpha e_i$ maps to $-\overline{\alpha}e_i$. It follows that if $v$ is not of characteristic $2$, then $\Lambda\cap L_v$ is the set of sums $\sum_i\alpha_ie_i$ with $\alpha_i$ of trace zero. An easy computation then shows the local discrepancy at $v$ is $(1/2) g \log q_v$.  We conclude that
\beq
\height_{T_{\cX}^\vee} = \height^{\st}_{T_{\cX}^\vee} + (g/2)\log \disc_{F/K} = (g/2)\log \disc_{F/K}.
\eeq
Furthermore (still setting aside the bounded contribution of $2$), the conductor $|\supp R_\pi|$ is just equal to $\log \disc_{F/K}$.  So Conjecture~\ref{stackyvojta} says that the set of $x$ with
\beq
(1 - g/2)\log \disc_{F/K} + \delta \height_L(x) < 0
\eeq
is supported on a closed subvariety, for any real $\delta > 0$.  When $g \geq 2$ this is vacuous, but when $g \geq 3$ it has content. By changing $\delta$ we can absorb the constant on the right-hand side, and say that the prediction is as follows:  for any abelian variety $A/K$ of dimension at least $3$, and any real $\delta > 0$, there is a closed subvariety $Z_\delta \subset A$ such that, for any trace-zero quadratic point $P \in A(\overline{K}) \backslash Z(\overline{K})$, the absolute logarithmic height of $P \in A(\overline{K})$ is at least $\delta^{-1} \log \disc_{F/K}$. 

This formulation may seem a bit cumbersome, but it is necessary.  Suppose, for example, that $A$ is the Jacobian of a hyperelliptic curve $X$ over $K$, and suppose $X$ has a rational Weierstrass point so it embeds into $A$ via an Abel--Jacobi map.  Then $X$ provides many quadratic points $P$ on $A$ whose heights are bounded above by $c \log \disc_{K(P)/K}$ for some real $c$.  So if $\delta < c^{-1}$, the exceptional set $Z_\delta$ needs to include $X$.  But if we take $\delta < (1/m)c^{-1}$, then every quadratic point on $A$ lying on the curve $[m]X$ satisfies $\log \disc_{K(P)/K} + \delta \height(x) < 0$, so we need to include not only $X$ but $[2]X,[3]X, \ldots, [m]X$ in the exceptional locus $Z_\delta$.  On the other hand, no matter what $\delta$ is, there should be many quadratic points in $A \backslash Z_\delta$, because (at least under modest assumption on $A$) the functional equation of quadratic twists of $A$ will vary in sign with the twist, which means there will be many quadratic twists $A_d$ of $A$ which under Birch--Swinnerton-Dyer have positive rank.  The heuristics here would suggest that the non-torsion points on such an $A_d$ have very large height relative to $d$.  Is this reasonable?

\subsection{Further questions}

There are many questions about the subject matter here which in the interests of length and time we have not addressed.  
\begin{itemize}
\item How does one compute $\edd(x)$ explicitly when $K$ is the function field of a curve in finite characteristic and $\cX$ is not tame?
\item Is Conjecture~\ref{conj:weak-batman} geometrically consistent in the sense of Lehmann, Sengupta, and Tanimoto \cite{geometricconsistency}?
\item How should one estimate the asymptotic growth of points on $\cX$ which are integral with respect to a divisor $D$?
\item As mentioned earlier in the paper, one might, rather than defining height in terms of the degree of $\pi_* \overline{x}^* \cV$, simply keep track of the vector bundle $\pi_* \overline{x}^* \cV$ itself.  When $K = \Q$ this metrized vector bundle is a lattice of the same rank as $\cV$.  When $\cX$ is a scheme, this point of view has been advanced by Peyre~\cite{peyre:beyondheights} as a more refined means of studying rational points on varieties.  When $\cX = BG$ and $\cV$ is a permutation representation of $G$, this lattice is related to the shape of the integer lattice in the $G$-extension $L/\Q$ corresponding to $x$; the variation of these lattices as one ranges over $G$-extensions of bounded discriminant  has been an object of much recent interest~\cite{bhargavaharron, harron:cubics, guillermoshapes}.  What can be said about intermediate cases, like $\Sym^m \P^n$?
\end{itemize}

\appendix
\section{Metrized Vector bundles on stacks over number fields}

\subsection{Linear Algebra}
\label{ss:linear-algebra-metrics}

An \defi{Hermitian pairing} on a complex vector space $V$ is a bilinear map $\left \langle \,, \right \rangle \colon V \to \mathbb{C}$ such that for all $v,w \in V$, $\left \langle w, v\right> = \overline{\left \langle v,  w\right\rangle}$ (whence $\left \langle v, v \right\rangle \in \mathbb{R}_{\geq 0}$). 
We define the associated \defi{Hermitian norm} $\left\lVert \cdot \right\rVert  \colon V \to \mathbb{R}$ via $\left\lVert v\right\rVert := \sqrt{\left \langle v, v \right\rangle}$. 
We call such a pair $\overline{V} := (V, \left\lVert \cdot \right\rVert_V)$ (or equivalently, $(V, \left \langle \,, \right \rangle_V)$) an \defi{Hermitian space}. For $r \in \mathbb{R}_{\geq 0}$ we define the \defi{ball of radius} $r$ to be $B\left(\overline{V},r\right) := \{v \in V \nmid \left\lVert v \right\rVert \leq r\}$ (and refer to $B\left(\overline{V},1\right)$ as the \defi{unit ball} in $\overline{V}$). 
We define the \defi{standard Hermitian space} to be $\overline{\mathbb{C}^n} := (\mathbb{C}^n, \left \langle \,, \right \rangle)$, where $\left \langle x,y \right \rangle := \sum x_i\overline{y_i}$.

 A \defi{morphism}  $\phi \in \hom\left(\overline{V} , \overline{W}\right)$ of Hermitian spaces is a linear map $\phi\colon V \to W$ such that $\left\lVert \phi(v) \right\rVert_W \leq \left\lVert v \right\rVert_V$ for all $v \in V$. The space $\hom(V,W)$ admits a pairing  
\[
\left \langle \phi, \psi \right \rangle := \sup_{v \in B\left(\overline{V},1\right)} \left \langle \phi(v), \psi(v) \right \rangle_{W}.
\] 
The associated norm is $\left\lVert \phi \right\rVert = \sup_{v \in B\left(\overline{V},1\right)} \left\lVert \phi (v) \right\rVert_{W}$; we let $\underline{\hom}\left(\overline{V},\overline{W}\right)$ be the associated Hermitian space, whence $\hom\left(\overline{V} , \overline{W}\right) := B\left(\underline{\hom}\left(\overline{V},\overline{W}\right), 1\right)$.
We define the \defi{dual} $\overline{V}^{\vee}$ of $\overline{V}$ to be $\underline{\hom}\left(\overline{V},\overline{\mathbb{C}}\right)$. 

Let $\overline{V}$ be an Hermitian space and let $0 \to V' \to V \xrightarrow{\pi} V'' \to 0$ be an exact sequence of complex vector spaces. Then the restriction of $\left\lVert \cdot \right\rVert_V$ to $V'$ is an Hermitian norm $\left\lVert \cdot \right\rVert_{V'}$ on $V'$. 
The orthogonal complement $\left(V'\right)^{\perp}$ of $V'$ is naturally identified with $V''$, inducing a pairing $\left \langle \,, \right \rangle_{V''}$ on $V''$ via restriction of $\left \langle \,, \right \rangle_V$ and this identification; the induced \defi{quotient norm} $\left\lVert \cdot \right\rVert_{V''}$ on $V''$ can thus be computed as $\left\lVert v \right\rVert_{V''} = \inf_{w \in \pi^{-1}(v)} \left\lVert w \right\rVert_{V} $.

Let $\overline{V}$ and $\overline{W}$ be Hermitian spaces. We define the \defi{direct sum} $\overline{V} \oplus \overline{W} := (V \oplus W, \left\lVert \cdot \right\rVert_{V \oplus W})$ via the declaration $\left \langle v, w \right\rangle_{V \oplus W} = 0$ for $v \in V, w \in W$; one then computes that $\left\lVert v \oplus w\right\rVert_{V \oplus W} = \sqrt{\left\lVert v \right\rVert_{V}^2 + \left\lVert w \right\rVert_{W}^2 }$. 
We define the \defi{tensor product} $\overline{V} \otimes \overline{W} := (V \otimes W, \left\lVert \cdot \right\rVert_{V \otimes W})$ via the formula $\left \langle v_1\otimes w_1, v_2 \otimes w_2 \right\rangle_{V \oplus W} = \left \langle v_1, v_2 \right\rangle_{V} \cdot \left \langle w_1, w_2 \right\rangle_{W}$; one then computes that $\left\lVert v \otimes w\right\rVert_{V \otimes W} = \left\lVert v \right\rVert_{V} \cdot \left\lVert w \right\rVert_{W} $. 
We define the \defi{alternating product} $\bigwedge^n\overline{V}$ via  $\left \langle v_1 \wedge \cdots \wedge v_n, w_1 \wedge \cdots \wedge w_n \right\rangle = \det\left(\left\langle v_i, w_j\right\rangle\right)$; this is not exactly equal to the quotient norm of $\left\lVert \cdot \right\rVert_{V^{\otimes n}}$ along the map $V^{\otimes n} \to\bigwedge^n{V}$, but rather is $\sqrt{n!}$ times the quotient norm.
%

\subsection{Analytic spaces}
\label{ss:analytic-spaces-metrics}

Let $X$ be a complex analytic space (as in \cite{GrauertR:coherent-analytic-sheaves}) and let $\mathcal{V}$ be a vector bundle on $X$. Let $\mathcal{C}_X$ denote the sheaf of continuous functions on $X$ valued in $\mathbb{R}_{\geq 0}$. An \defi{Hermitian norm} $|\cdot|$ on $\mathcal{V}$ is a morphism of sheaves
\[
|\cdot|\colon \mathcal{V}\to\mathcal{C}_X,
\]
such that
\begin{enumerate}
\item $|s|(x) = 0$ if and only if $s(x)=0$, 
\item \label{item:norm-multiplicative} for all $f\in\mathcal{O}_{X}(U)$, we have $|fs| = |f||s|$, and
\item for every complex point $x\colon \ast \to X$, the restriction of $|\cdot|$ to $x^*\mathcal{V}$ is Hermitian (when viewed as a norm on $H^0\left(\ast, x^*\mathcal{V}\right)$), 
\end{enumerate}
where, in condition (\ref{item:norm-multiplicative}), $|f|$ is the trivial norm on the line bundle $\mathcal{O}_X$ (i.e., $f\in\mathcal{O}_X(U)$ corresponds to a continuous function
$f\colon U\to \mathbb{C}$, and we define $|f|\colon U\to \RR_{\geq 0}$ by $|f|(x) = |f(x)|$).
We call such a pair $\overline{\mathcal{V}} := (V,|\cdot|)$ a \defi{metrized vector bundle} on the analytic space $X$.
\\

We define direct sums, tensor products, alternating products, and duals via the formulas from (\ref{ss:linear-algebra-metrics}) (locally, and if necessary, we sheafify);
for example, given metrized vector bundles $(\mathcal{V}_1,|\cdot|_{1})$ and $(\mathcal{V}_2,|\cdot|_{2})$, we define 
\[
|\cdot|\colon \mathcal{V}_1 \oplus \mathcal{V}_2\to\mathcal{C}_X,
\]
as 
\[
|v_1 \oplus v_2|(x) :=  \left( \left(|v_1|_{1}(x)\right)^2 + \left(|v_2|_{2}(x)\right)^2 \right)^{1/2}.
\]
Given a morphism $g \colon X \to Y$ of analytic spaces and a metrized vector bundle $\overline{\mathcal{V}} = (\mathcal{V}, |\cdot|)$ on $Y$, we define the \defi{pull back} $g^*\overline{\mathcal{V}}$ to be the pair $((g^*\mathcal{V}), g^*|\cdot|)$, where $g^*|\cdot|$ is adjoint to the composition
\[
 \mathcal{V}\to\mathcal{C}_Y \to g_*\mathcal{C}_X,
\]
and where the second map is given by composition of functions.
If $g$ is unramified and finite (in particular, $g_*\mathcal{V}$ is a vector bundle), we define the \defi{direct image} $g_*\overline{\mathcal{V}}$ to be the pair $((g_*\mathcal{V}), g_*|\cdot|)$, where $g_*|\cdot|$ is defined via the composition
\[
g_* \mathcal{V}\to g_*\mathcal{C}_X \to \mathcal{C}_Y,
\]
and where $g_*\mathcal{C}_X \to \mathcal{C}_Y$ is defined by summation on fibers; in other words, for an open subset $U \subset Y$ and a function $h \in \mathcal{C}_X(g^{-1}(U))$, we define a map $U \to \mathbb{R}_{\geq 0}$ via the formula $y \mapsto \sqrt{\sum_{x \in g^{-1}(y)} h(x)^2}$. 
For a complex point $x\colon \ast \to X$ with image $y\colon \ast \to Y$, the natural map $(g^*\mathcal{V})_{x} \to \mathcal{V}_{y}$ is an isomorphism, and the norm is ``the same'' on these fibers.
%
In contrast, the fiber $(g_*\mathcal{V})_{y}$ of the direct image is naturally isomorphic to $\oplus_{x \in g^{-1}(y)} \mathcal{V}_x$, and the norm on this fiber is the direct sum norm defined in (\ref{ss:linear-algebra-metrics}). 

\subsection{Schemes}
\label{ss:schemes-metrics}

By a \defi{variety over $S$} we mean a scheme of finite type over $S$. To a variety $X$ over $\Spec \mathbb{C}$ and vector bundle $\mathcal{V}$ on $X$, associate the complex analytification $(X^{\an},\mathcal{V}^{\an})$ (as in \cite{GrauertR:coherent-analytic-sheaves}). (We note that one can also associate an analytic space, functorially, to an algebraic space which is locally separated and locally of finite type over $\mathbb{C}$ \cite[Ch. I, 5.17]{Knutson:algebraicSpaces}, and that the setup here extends to that generality without any further modification.)

 Let $K$ be a number field, let $X$ be a $\spec \mathcal{O}_K$ variety, and let $\mathcal{V}$ on $X$ be a vector bundle on $X$. For an embedding $\sigma\colon K\to\mathbb{C}$ (i.e., a map $\sigma\colon \spec \mathbb{C} \to \spec K$), we let $X_\sigma:=X\times_{K,\sigma}\mathbb{C}$ and let $\mathcal{V}_\sigma$ denote the pullback of $\mathcal{V}$ to $X_\sigma$. We define a \defi{metrized vector bundle} on $X$ to be a vector bundle $\mathcal{V}$ together with a choice of Hermitian norm $|\cdot|_\sigma$ on $\mathcal{V}^{\an}_\sigma$ for every embedding $\sigma\colon K\to\mathbb{C}$, with the following property: for every Zariski open $U \subset X$ and section $s\in\mathcal{V}(U)$, we have $|\sigma^*s|_\sigma(p) = |{\bbar\sigma}^* s|_{\bbar \sigma}(\bbar p)$. 

We define direct sums, tensor products, alternating products, and duals via the formulas from (\ref{ss:analytic-spaces-metrics}). 
Given a morphism $g \colon X \to Y$ of $\Spec \mathcal{O}_K$ varieties and an embedding $\sigma \colon K \to \mathbb{C}$, the diagram 
\[
\xymatrix{
X_{\sigma} \ar[r]^{g_{\sigma}} \ar[d] &  Y_{\sigma} \ar[d] \\
X \ar[r]^g & Y
}
\]
commutes. Given a metrized vector bundle $\overline{\mathcal{V}} = (\mathcal{V}, |\cdot|)$ on $Y$, it follows that $\left(g^*\mathcal{V}\right)_{\sigma}$ is canonically isomorphic to $g_{\sigma}^*\left(\mathcal{V}_{\sigma}\right)$, and we define the \defi{pull back} $g^*\overline{\mathcal{V}}$ to have underlying vector bundle $g^*\mathcal{V}$ and metrics $g_{\sigma}^*|\cdot|_{\sigma}$ defined via  (\ref{ss:analytic-spaces-metrics}).
Similarly, if $g$ is finite, flat, and generically \'etale (and in particular locally free, so that $g_*\mathcal{V}$ is a vector bundle), we define the \defi{direct image} $g_*\overline{\mathcal{V}}$ to have underlying bundle $g_*\mathcal{V}$ and metrics $g_{\sigma,*}|\cdot|_{\sigma}$ defined via  (\ref{ss:analytic-spaces-metrics}).

There is an alternative type of direct image, which highlights the choice of base in our definition. Let $K \subset L$ be an inclusion of number fields. Let $X \to \Spec \mathcal{O}_L$ be an $\mathcal{O}_L$ variety and let $\overline{\mathcal{V}}$ be a metrized vector bundle on $X$. We define the \defi{restriction of scalars} of $(X, \overline{\mathcal{V}})$ to be the pair $(\Res_{L/K} X, \Res_{L/K}  \overline{\mathcal{V}})$, where $\Res_{L/K} X$ is the usual restriction of scalars (i.e., $X$ itself, viewed as an $\mathcal{O}_K$ variety via the composition $X \to \Spec \mathcal{O}_L \to \Spec \mathcal{O}_K$), and where $\Res_{L/K}  \overline{\mathcal{V}}$ has the same underlying vector bundle $\mathcal{V}$ and is endowed with a metric in the following way. 
Given an embedding $\sigma \colon K \hookrightarrow \mathbb{C}$, the space $\left(\Res_{L/K} X\right)_{\sigma}$ is isomorphic to $\coprod_{\sigma' \mid \sigma} X_{\sigma'}$, where the coproduct is taken over the set of $\sigma' \colon L \hookrightarrow \mathbb{C}$ extending $\sigma$; similarly, $\left(\Res_{L/K} \mathcal{V}\right)_{\sigma}$ is the vector bundle whose restriction to $X_{\sigma'}$ is $\mathcal{V}_{\sigma}$ (note that, by the sheaf axioms, $\Gamma(X_{\sigma}, \mathcal{V}_{\sigma} ) = \bigoplus_{\sigma' \mid \sigma} \Gamma(X_{\sigma'}, \mathcal{V}_{\sigma'})$), and the norm 
\[
|\cdot|_{\sigma}\colon \left(\Res_{L/K} \mathcal{V}\right)^{\an}_{\sigma} \to \mathcal{C}_{\left(\Res_{L/K} X\right)^{\an}_{\sigma}}
\]
 is the one whose restriction to $X_{\sigma'} \subset \left(\Res_{L/K} X\right)_{\sigma}$ is $|\cdot|_{\sigma'}$.

Similarly, if $K \hookrightarrow L$ is an extension of number fields, $X$ is an $\mathcal{O}_L$ variety, and $\overline{\mathcal{V}}$ is a metrized vector bundle on $X$ \emph{considered as an} $\mathcal{O}_K$ variety (equivalently, a metrized bundle on $\Res_{L/K} X$), we define \defi{base extension} $\overline{\mathcal{V}}_L$ as follows. The underlying bundle is $\mathcal{V}$; for a place $\sigma'$ of $L$ with restriction $\sigma := \sigma'|_{K}$, the map $\phi\colon X_{\sigma'} \to \Res X_{\sigma}$ of $\mathbb{C}$ varieties is an isomorphism, and we define $|\cdot|_{\sigma'}$ to be the same as $|\cdot|_{\sigma}$ (under the identification $\phi$). 

The \defi{degree} of a metrized line bundle $(\mathcal{V},|\cdot|)$ on $\spec \mathcal{O}_K$ (considered as an $\mathcal{O}_K$-variety) is defined to be
\begin{equation}
  \label{eq:metrized-degree}  
\deg(\mathcal{V},|\cdot|) = \log \left| \Gamma(\mathcal{V}) / \mathcal{O}_K \cdot s \right| - \sum_{\sigma\colon K\to\mathbb{C}} \log |\sigma^*s|_\sigma,
\end{equation}
where $s\in\Gamma(\mathcal{V})$ is any non-zero section. Implicit here is that this definition is independent of the choice of $s$. When $(\mathcal{V},|\cdot|)$ is a metrized vector bundle of rank $r > 1$, the degree of $(\mathcal{V},|\cdot|)$ is by definition the degree of the metrized vector bundle $\wedge^n (\mathcal{V},|\cdot|)$. If $K \hookrightarrow L$ is an extension of number fields and $(\mathcal{V},|\cdot|)$ is a metrized line bundle on $\spec \mathcal{O}_L$ \emph{considered as an} $\mathcal{O}_K$-variety, then we define $\deg(\mathcal{V},|\cdot|) := \deg(\mathcal{V}_L,|\cdot|)$, where $\mathcal{V}_L$ is the base extension of $\mathcal{V}$ to $K$.

If $K \subset L$ is a degree $n$ extension of number fields, then the following direct computation shows that 
\begin{equation}
\label{eq:arakelov-pullback-degree}
\deg(\mathcal{V}_L,|\cdot|) = n\cdot \deg(\mathcal{V},|\cdot|).
\end{equation}
Indeed, pullbacks commute with top wedge power, so it suffices to check the equality when $\mathcal{V}$ is a line bundle, in which case 
\[
\sum_{\sigma'\colon L\to\mathbb{C}} \log |(\sigma')^*s|_\sigma = 
\sum_{\sigma\colon K\to\mathbb{C}} \left(\sum_{\sigma' \mid \sigma} \log |\sigma^*s|_\sigma \right)= 
\sum_{\sigma\colon K\to\mathbb{C}} n \cdot \log |\sigma^*s|_\sigma
\]
and, since $\mathcal{O}_L$ is a flat $\mathcal{O}_K$-module, 
\[
 |(\Gamma(\mathcal{V}) \otimes_{\mathcal{O}_K}\mathcal{O}_L) / \mathcal{O}_L \cdot s| = 
|(\Gamma(\mathcal{V}) / \mathcal{O}_K \cdot s) \otimes_{\mathcal{O}_K}\mathcal{O}_L | = 
n\cdot  |\Gamma(\mathcal{V}) / \mathcal{O}_K \cdot s |.
\]
%


\subsection{Stacks}
\label{ss:stackybundles}

This generalizes to stacks in the following fairly formal way. 
\\

  Let $\mathcal{X}$ be an algebraic stack, finite type over $\Spec \mathcal{O}_K$. We define a \defi{metrized vector bundle} $\overline{\mathcal{V}}$ on $\mathcal{X}$ to be a vector bundle $\mathcal{V}$ on $\mathcal{X}$ together with, for every map $f \colon X \to \mathcal{X}$ from a variety $X$,  a choice of metric on $f^*\mathcal{V}$ (in the sense of \ref{ss:schemes-metrics}) which we denote by $f^*| \cdot |$, and which is compatible with compositions in the following sense: for a map $g\colon X' \to X$ from an $\mathcal{O}_K$-variety $X'$, there is a canonical isomorphism $g^* \left(f^*\mathcal{V}\right)  \to (f \circ g)^*\mathcal{V}$, and we require that this isomorphism identifies $g^*\left( f^*| \cdot | \right)$ with $(f \circ g)^*| \cdot |$.

  We again define direct sums, tensor products, alternating products, and duals via the formulas from (\ref{ss:linear-algebra-metrics}).  
Given a morphism $g \colon \mathcal{X} \to \mathcal{Y}$ and a metrized vector bundle $\overline{\mathcal{V}}$ on $\mathcal{Y}$, we define the \defi{pull back} $g^*\overline{\mathcal{V}}$ to have underlying bundle $g^*\mathcal{V}$ and, for a map $f \colon X \to \mathcal{X}$ from an $\mathcal{O}_K$-variety $X$, define $f^*(g^*\overline{\mathcal{V}}) := (g \circ f)^*\overline{\mathcal{V}}$. 
For direct images we restrict to the following special cases.
Let $\overline{\mathcal{V}} = (\mathcal{V}, |\cdot|)$ be a metrized vector bundle on $\mathcal{X}$. If $g$ is finite, flat, and generically \'etale (and in particular representable), we define the \defi{direct image} $g_*\overline{\mathcal{V}}$ to be the metrized vector bundle on $\mathcal{Y}$ which, for a map $f \colon Y \to \mathcal{Y}$ from a variety $Y$ with corresponding fiber product
\[
\xymatrix{
X \ar[r]^{f'} \ar[d]_{g'} & \mathcal{X} \ar[d]^g \\
Y \ar[r]^{f}          & \mathcal{Y}
}
\]
pulls back to $f^*\left(g_*\overline{\mathcal{V}}\right)  := g'_*f^{'*}\overline{\mathcal{V}}$. If instead $g$ is proper, quasi-finite, and birational, and $\mathcal{Y}$ is isomorphic to $\Spec \mathcal{O}_K$, then $g$ is an isomorphism on a non-empty open subset $U \hookrightarrow \mathcal{X}$; we define $g_*\overline{\mathcal{V}}$ to have underlying bundle $g_*\mathcal{V}$ (which is a vector bundle by Proposition \ref{coro:vb-pushforward}) and the metric defined by $g^*|\cdot|$. 

\subsection{A detailed example}
\label{ss:BG-arakelov}

Let $K$ be a number field and let $X = \Spec \mathcal{O}_K$, considered as an $\mathcal{O}_K$ variety. We consider the trivial metrized vector bundle  $(\mathcal{O}_X,|\cdot|)$ (where the trivial norms are defined in Subsection \ref{ss:analytic-spaces-metrics}). Explicitly, for an embedding $\sigma \colon K \hookrightarrow \mathbb{C}$, the scheme $X_{\sigma}$ is simply $\Spec \mathbb{C}$, and the norm  
\[
|\cdot|_{\sigma}\colon \mathcal{O}^{\an}_{X,\sigma} \to \mathcal{C}_{X^{\an}_{\sigma}}
\]
is the complex absolute value $\mathbb{C} \to \mathbb{R}_{\geq 0}$. Given a section $s \in \mathcal{O}_K$, $|\sigma^*s|_{\sigma}$ is equal to the complex absolute value $|\sigma(s)|$. Taking $s = 1$, we compute that the degree
\[
\deg(\mathcal{V},|\cdot|) = \log \left| \mathcal{O}_K / \mathcal{O}_K \cdot 1 \right| - \sum_{\sigma\colon K\to\mathbb{C}} \log |\sigma^*1|_\sigma = 0 - \sum_{\sigma\colon K\to\mathbb{C}} 0
\]
is 0, as one would expect of a trivial bundle.

Next, let $K$ be a number field and again let $X = \Spec \mathcal{O}_K$, but now considered as a variety over $\mathbb{Z}$. We consider the ``trivial'' metrized vector bundle  $(\mathcal{O}_X,|\cdot|)$  (where the trivial norms are defined in Subsection \ref{ss:analytic-spaces-metrics}). This is the same as the pull back of the trivial bundle on $\Spec \mathbb{Z}$ along the map (of $\mathbb{Z}$ varieties) $\Spec \mathcal{O}_K \to \Spec \mathbb{Z}$.
Explicitly, there is only one embedding $\sigma \colon \mathbb{Q} \hookrightarrow \mathbb{C}$, and the scheme $X_{\sigma}$ is isomorphic to the disjoint union $\coprod_{\sigma' \mid \sigma} X_{\sigma'}$, where the coproduct is taken over the set of embeddings $\sigma' \colon K \hookrightarrow \mathbb{C}$ of $K$ and where $X_{\sigma'} = X\times_{K,\sigma'}\mathbb{C}$ (i.e., considered as an $\mathcal{O}_K$ scheme); $X_{\sigma}$ is thus a disjoint union of $[K:\mathbb{Q}]$ copies of $\spec \mathbb{C}$. The norm  
\[
|\cdot|_{\sigma}\colon \mathcal{O}^{\an}_{X,\sigma} \to \mathcal{C}_{X^{\an}_{\sigma}}
\]
is locally (on $X_{\sigma}$) again given by the complex absolute value. 
%
%
%
Label the embeddings $\sigma_1,\ldots,\sigma_n$ and let $s \in \mathcal{O}_K$. Then $\sigma^*s$ is equal to the tuple $(\sigma_1(s),\ldots,\sigma_n(s))$. Given our choice of base, it does not make sense to compute the degree. Note that this description is also the same as the restriction of scalars (as in Subsection \ref{ss:schemes-metrics}) of the trivial metrized bundle on $\Spec \mathcal{O}_K$ (as an $\mathcal{O}_K$ variety) from the previous paragraph. 

Now let $X = \Spec \mathcal{O}_K$ and $Y = \Spec  \mathbb{Z}$, and let $\pi\colon X \to Y$ be the structure map. Consider the direct image $\pi_* \overline{\mathcal{O}_X} = (\pi_*\mathcal{O}_X, \pi_*|\cdot|)$, where we consider $X$ as a variety over $\mathbb{Z}$ and where $|\cdot|$ is the trivial metric. Then $\pi_*\mathcal{O}_X \cong \widetilde{\mathcal{O}_K}$ and $\pi_*|\cdot|$ has the following description. Again, since our base is $\Spec \mathbb{Z}$, there is only one embedding $\sigma \colon \mathbb{Q} \hookrightarrow \mathbb{C}$; the scheme $Y_{\sigma}$ is isomorphic to a single copy of $\Spec \mathbb{C}$, and the norm  
\[
\left(\pi_*|\cdot|\right)_{\sigma}\colon \left(\pi_*\mathcal{O}_{X,\sigma}\right)^{\an} \to \mathcal{C}_{Y^{\an}_{\sigma}}
\]
is now a map of sheaves on a topological space which is a single point, and thus determined by the map of global sections
\[
\mathcal{O}_K \otimes_{\mathbb{Z}} \mathbb{C} \cong \prod_{\sigma' \mid \sigma} \mathbb{C} \to \mathbb{R}_{\geq 0}
\]
where the product is taken over the set of embeddings $\sigma' \colon K \hookrightarrow \mathbb{C}$ of $K$,
which we label as $\sigma_1,\ldots,\sigma_n$. The map $\prod_{\sigma' \mid \sigma} \mathbb{C} \to \mathbb{R}_{\geq 0}$ is given by 
\[
(z_1,\ldots, z_n) \mapsto \sqrt{\sum |z_i|^2}
\]
and the isomorphism $\mathcal{O}_K \otimes_{\mathbb{Z}} \mathbb{C} \cong \prod_{\sigma' \mid \sigma} \mathbb{C}$ is given by 
\[
\alpha \otimes 1 \mapsto (\sigma_1(\alpha),\ldots, \sigma_n(\alpha)).
\]
We now compute the degree of $\pi_* \overline{\mathcal{O}_X}$. Let $\overline{\mathcal{V}} := \bigwedge^n \pi_* \overline{\mathcal{O}_X}$  be the top wedge power of $\pi_* \overline{\mathcal{O}_X}$, and choose a $\mathbb{Z}$ basis $\alpha_1, \ldots, \alpha_n$ of $\mathcal{O}_K$. Then $\bigwedge^n \mathcal{O}_K$ is a free $\mathbb{Z}$ module of rank 1 generated by the section $s = \alpha_1 \wedge \cdots \wedge \alpha_n$. We then compute that the degree is
\[
\log \left| \Gamma(\mathcal{V}) / \mathbb{Z} \cdot s \right| - \log |\sigma^*s|_\sigma = 0 -\log |\sigma^*s|_\sigma.
\]
Next, we compute $\log |\sigma^*s|_\sigma$. The norm $\bigwedge^n\pi_*|\cdot|$ is given by the composition
\[
\left(\bigwedge^n \mathcal{O}_K\right)\otimes_{\mathbb{Z}} \mathbb{C} \cong
\bigwedge^n \left(\mathcal{O}_K\otimes_{\mathbb{Z}} \mathbb{C} \right)\cong
\bigwedge^n \prod_{\sigma' \mid \sigma} \mathbb{C} \cong 
\mathbb{C} \to 
\mathbb{R}_{\geq 0};
\]
following $s$ through these maps
\begin{equation*}
  \begin{split}
(\alpha_1 \wedge \cdots \wedge \alpha_n) \otimes 1  \mapsto & (\alpha_1 \otimes 1) \wedge \cdots \wedge (\alpha_n \otimes 1) \\
\mapsto & (\sigma_1(\alpha_1),\ldots,\sigma_n(\alpha_1)) \wedge \cdots \wedge  (\sigma_1(\alpha_n),\ldots,\sigma_n(\alpha_n)) \\
= & \det (\sigma_j(\alpha_i)) \cdot \left(1 \wedge \cdots \wedge 1 \right) \\
\mapsto &\left|\det (\sigma_j(\alpha_i))\right| = |\Delta_K|^{1/2} 
\end{split}
\end{equation*}
we conclude that $|\sigma^*s|_\sigma = |\Delta_K|^{1/2}$ and that the degree of $\pi_*\overline{\mathcal{O}_K}$ is $-\log |\Delta_K|^{1/2}$. 

Finally: let $C = \Spec \mathbb{Z}$ and let $BG = [C / G]$, with quotient map $p\colon C \to BG$. Let $\overline{\mathcal{V}} = \left(p_* \overline{\mathcal{O}_{C}}\right)^{\vee}$, where  $\overline{\mathcal{O}_{C}}$ is the trivial metrized line bundle on $C$. (We dualize to facilitate the following quick global computation.) Let $x\colon \Spec \mathbb{Q} \to BG$ be a rational point corresponding to an extension $\Q \subset K$, and assume for this example that $K$ is a number field (rather than just an \'etale algebra). We will now show that $\height_{\overline{\mathcal{V}}}(x) = \log |\Delta_K|^{1/2}$. Let $\mathcal{C} = [\Spec \mathcal{O}_K / G]$. Then $\mathcal{C}$ is a tuning stack for $x$, summarized by the following diagram.
\[
\xymatrix{
\spec K\ar[r]\ar[d] & \Spec \mathcal{O}_K\ar[r]^-{g}\ar[d]^{p'} 
& C\ar[d]^p\\
\spec \mathbb{Q}\ar[r]\ar[dr] & \cC\ar[r]^-{\overline{x}}\ar[d]^{\pi} & BG \ar[ld]\\
&\Spec \mathbb{Z}&
}
\]
By definition, $\left(\overline{x}^*\overline{\mathcal{V}}\right)^{\vee} = p'_* g^*\overline{\mathcal{O}_C}$. Moreover, the tuning sheaf $\pi_*p'_* g^*\overline{\mathcal{O}_C}$ is isomorphic to $(p'\circ \pi)_*g^*\overline{\mathcal{O}_C}$, and $g^*\overline{\mathcal{O}_C}$ is the trivial metrized line bundle on $\Spec \mathcal{O}_K$ (as a $\mathbb{Z}$ variety). The height is then, by definition, 
\[
\height_{\overline{V}}(x)  := -\deg \left((p'\circ \pi)_*g^*\overline{\mathcal{O}_C}\right);
\]
we conclude that $\height_{\overline{V}}(x) = \log |\Delta_K|^{1/2}$.

\section{One-dimensional Artin stacks with finite diagonal}
\label{ss:one-dimensional-artin}

In this appendix we discuss a few technical aspects of the types of stacks that appear as the tuning stack of a rational point (Definition \ref{def:tuning-stack}). 
\\

Fix a base scheme $S$. An Artin stack $\mathcal{C}$ (finite type over $S$) with finite diagonal admits a \emph{coarse space} map $\pi\colon \mathcal{C} \to C$ \cite[Corollary 1.3 (1)]{KeelM:Quotients}, which is (by definition) universal for maps to algebraic spaces and is a bijection on geometric points, and is moreover Stein (i.e., $\pi_*\mathcal{O}_{\mathcal{C}} \cong \mathcal{O}_C)$ and a universal homeomorphism \cite[Theorem 6.12]{rydh:quotients}. If $S = \Spec k$ for some field $k$, then we say that $\mathcal{C}$ is \defi{geometric}; if $S \to \Spec \mathbb{Z}$ is finite and flat, then we say that $\mathcal{C}$ is \defi{arithmetic}.

\begin{definition}
\label{def:stacky-curve-def}

 A \defi{stacky curve} is a normal, one-dimensional Artin stack $\mathcal{C}$ with finite diagonal such that the coarse space map $\pi\colon \mathcal{C} \to C$ is birational, and such that $C/S$ is a proper curve if $\cC$ is geometric and finite over $S$ if $\cC$ is arithmetic. 
\end{definition}

Normality of $C$ follows from normality of $\mathcal{C}$, so $C/k$ is a smooth proper curve in the geometric case and $C \cong \coprod \Spec \mathcal{O}_{K_i}$ for some number fields $K_i$ in the arithmetic case. This is somewhat more general than the notion of stacky curve from \cite[Chapter 5]{voightZB:canonical-ring-of-a-stacky-curve}. 
\\

Our beginning lemma was pointed out to us by Sid Mathur.

\begin{lemma}
\label{l:tuning-stack-regular}
Let $\cC$ be a stacky curve. Then $\cC$ is regular.
\end{lemma}
\begin{proof}
Since $\cC$ is an Artin stack, it has a smooth cover $p\colon U\to\cC$. Let $y\in\cC(\Omega)$ be a geometric point. Then $\pi(y)$ is a geometric point of $C$. Since $C$ has dimension at most $1$, the point $\pi(y)$ has codimension at most $1$ in $C$. Therefore, there exists a point $z\in U(\Omega)$ with $\pi \circ p(z)=\pi(y)$ such that $z$ has codimension at most $1$ in $U$. Since $\pi$ is a coarse space map, $p(z)\simeq y$.

Since $\cC$ is normal, $U$ is as well, and so $z$ is a regular point of $U$. Therefore, there is an open neighborhood $V\subseteq U$ of $z$ such that $V$ is regular. Since the image of $p|_V\colon V\to\cC$ contains $p(z)\simeq y$, we have found a smooth cover of a neighborhood of $y\in\cC(\Omega)$ by a regular scheme.
\end{proof}

\begin{proposition}
\label{prop:tuning-finite-cover}
There exists a finite flat surjection $p\colon C'\to\cC$ with $C'$ regular and with irreducible connected components. The composition $\pi\circ p\colon C'\to C$ is finite and flat.
\end{proposition}
\begin{proof}
We may assume that $\cC$ is connected. Since $\cC$ has finite diagonal, we know from \cite[Theorem 2.7]{EdidinHKV} that there is a finite surjective map $p\colon C'\to \cC$ where $C'$ is a scheme. We can assume $C'$ is normal by replacing it with its normalization. Since $\pi$ is proper and quasi-finite, $q:=\pi\circ p$ is proper and quasi-finite, hence finite. 
Since $C$ is of dimension $1$, so is $C'$. As $C'$ is normal, it is regular. Since $q$ is surjective, we can replace $C'$ by one of its irreducible components which surjects onto $C$; note that this maintains surjectivity of $p$, as $\pi$ is a bijection on geometric points. Since $C$ and $C'$ are regular, $q$ is flat by \cite[Corollary 18.17]{Eisenbud:commutativeAlgebra}. Similarly, since $\cC$ is regular, letting $U\to\cC$ be any smooth cover by a scheme, we see the pullback $p_U\colon C'\times_\cC U\to U$ is a finite map between regular schemes. Again, \cite[Corollary 18.17]{Eisenbud:commutativeAlgebra} tells us that $p_U$ is flat and hence $p$ is flat.
\end{proof}

\begin{corollary}
\label{coro:vb-pushforward}
Let $\E$ be a vector bundle on $\cC$.  Then $\pi_*\E$ is a vector bundle. 
\end{corollary}

\begin{proof}
We can assume that $\mathcal{C}$ is connected. We claim that the canonical map $\O_\cC\to p_*\O_{C'}$ is injective. It suffices to check this after passing to a smooth cover $\spec A\to\cC$. We see $C'\times_\cC \spec A \to \spec A$ is finite, so the fiber product is of the form $\spec B$. The induced map $\spec B\to\spec A$ is surjective, hence dominant, and $\spec A$ is regular, hence reduced, so $A\to B$ is injective, proving our claim.

To finish the proof, tensor the injective map $\O_\cC\to p_*\O_{C'}$ by with $\E$. This yields an injection $\E\to\E\otimes p_*\O_{C'} \cong p_*p^*\E$ (where the isomorphism is the projection formula), and hence an injection $\pi_*\E\to q_*p^*\E$. Since $p^*\E$ is a vector bundle and $q$ is finite flat, we see $q_*p^*\E$ is a vector bundle, so $\pi_*\E$ is torsion-free and coherent. As $C$ is regular of dimension 1, this implies $\pi_*\E$ is a vector bundle.
\end{proof}

We now address generalities about of the degree of a line bundle  on an Artin stack. In the geometric case, if $\cC$ is Deligne--Mumford, then Vistoli \cite{vistoli:intersection-theory-on-algebraic-stacks} developed a more general theory of intersection theory (see also \cite[Chapter 5]{voightZB:canonical-ring-of-a-stacky-curve} for just the case of line bundles).  In general, degrees of $0$-cycles on stacks are not defined (see \cite{EdidinGS:artinCycle}), and in the Arakelov setting (as in \ref{ss:linear-algebra-metrics}) some additional attention is needed even in the Deligne--Mumford case. However, we have shown in Proposition \ref{prop:tuning-finite-cover} that every connected stacky curve $\cC$ admits a finite flat surjection $C'\to\cC$ with $C'$ regular and irreducible, and by \cite[Remark 2.8]{EdidinHKV} this is all that one needs to develop intersection theory in our setting.

 \begin{definition}
\label{defn:degree-length}
Let $\mathcal{L}$ be a line bundle (resp.~torsion sheaf) on $\mathcal{C}$ and let $p\colon C'\to\cC$ be a finite and flat surjection from a regular scheme $C'$. We define the \defi{degree} (resp.~\defi{length}) of $\mathcal{L}$ to be $\deg \mathcal{L}=\frac{1}{\deg(p)}\deg p^*\mathcal{L}$ (resp.~$\length \mathcal{L}=\frac{1}{\deg(p)}\length p^*\mathcal{L}$). 
\end{definition}

Again, we emphasize the fact that in the arithmetic setting $\mathcal{L}$ is an Hermitian line bundle and we mean the Arakelov degree. For a torsion sheaf, the Archimedean contributions are 0 so there is no distinction.

\begin{lemma}
\label{lemma:degree-line-bundle-well-definied} 
 The degree (resp.~length) of $\mathcal{L}$ is independent of the choice of $p$.
\end{lemma}

\begin{proof}
    
Let $p_i\colon C_i\to\cC$ be two such covers, and let $C_3$ be the normalization of some irreducible component of $C_1 \times_{\mathcal{C}} C_2$ such that the maps $q_i\colon C_3 \to C_i$ are both surjective (and thus finite and flat). We then have
\begin{equation}
\label{eq:independence-of-degree}
   \frac{\deg p_1^*\mathcal{L}}{\deg p_1} 
=    \frac{\deg q_1^*p_1^*\mathcal{L}}{(\deg q_1) (\deg p_1)} 
=   \frac{\deg q_2^*p_2^*\mathcal{L}}{(\deg q_2) (\deg p_2) } 
= \frac{\deg p_2^*\mathcal{L}}{\deg p_2 }.
\end{equation}
The proof for length is identical.
\end{proof}

\begin{definition}
  Let $f\colon \mathcal{C}' \to \mathcal{C}$ be a quasi-finite map of stacky curves. We define the \defi{degree} of $f$ to be the degree of the induced map $C' \to C$ of coarse spaces. 
\end{definition}

\begin{lemma}
\label{lemma:functorality-of-degree}
  Let $f\colon \mathcal{C}' \to \mathcal{C}$ be a quasi-finite map of stacky curves and let $\mathcal{L}$ be a line bundle (resp.~torsion sheaf) on $\mathcal{C}$. Then  $\deg f^*\mathcal{L} = \deg f \cdot \deg \mathcal{L}$ (resp.~$\length f^*\mathcal{L} = \deg f \cdot \length \mathcal{L}$).
\end{lemma}

\begin{proof}
  If $\mathcal{C}'$ is a scheme then this follows from the definitions of degree.  Let $p\colon C' \to \mathcal{C}'$ be a finite flat cover by a regular scheme $C'$. By 
\cite[\href{https://stacks.math.columbia.edu/tag/0CPT}{Tag 0CPT}]{stacks-project}, $f$ is proper; the composition $f \circ p$ is thus proper, quasi-finite, and flat, and in particular finite. We then have
\[
\deg f^* \mathcal{L} =    \frac{\deg p^*f^*\mathcal{L}}{\deg p}   = \deg f\frac{\deg p^*f^*\mathcal{L}}{\left(\deg p\right) \left(\deg f\right)}   = 
\deg f \cdot  \deg  \mathcal{L}.
\]
The proof for length is identical.
\end{proof}

\begin{proposition}
  \label{proposition:degree-length-exactness}  
  Let $0 \to \mathcal{V}' \to \mathcal{V} \to M \to 0$ be an exact sequence, where
$\mathcal{V}' \to \mathcal{V}$ is a map of vector bundles (metrized, in the Arakelov case) and $M$ is a finitely generated torsion sheaf on $\mathcal{C}$. Then 
\[
\deg \mathcal{V} = \deg \mathcal{V}'  + \length M.
\]
\end{proposition}

\begin{proof}
  In the geometric case this is well known. In the Arakelov case, by Lemma \ref{lemma:functorality-of-degree} we may assume that $\mathcal{C} = \Spec \mathcal{O}_K$ for some number field $K$. Since $M$ is a torsion sheaf and thus has no archimedean metric, the proof follows from the definition of degree (Equation \ref{eq:metrized-degree}).
\end{proof}

\def\cprime{$'$}

\end{document}